\documentclass[12pt,reqno]{amsart}
\usepackage{amsmath}
\usepackage{amssymb}
\usepackage{amstext}
\usepackage{mathrsfs}
\usepackage{a4wide}
\usepackage{verbatim}
\usepackage{bm}
\usepackage{graphicx}
\usepackage{tikz}
\usepackage{pgfplots}
\usepackage{titletoc}
\usepackage[nocompress]{cite}
\usepackage{booktabs}
\usepackage{placeins}
\usepackage{caption}

\allowdisplaybreaks 
\usepackage{color}
\usepackage{cases}
\usepackage{hyperref}
\hypersetup{
	colorlinks=true,
	linkcolor=black,
	anchorcolor=black,
	citecolor=black}

\pgfplotsset{compat=1.18}

\numberwithin{equation}{section}

\newtheorem{theorem}{Theorem}[section]
\newtheorem{proposition}[theorem]{Proposition}
\newtheorem{corollary}[theorem]{Corollary}
\newtheorem{lemma}[theorem]{Lemma}

\theoremstyle{definition}

\newtheorem{assumption}[theorem]{Assumption}
\newtheorem{definition}[theorem]{Definition}

\theoremstyle{remark}
\newtheorem{remark}[theorem]{Remark}

\newcommand{\vertiii}[1]{{\left\vert\kern-0.25ex\left\vert\kern-0.25ex\left\vert #1
		\right\vert\kern-0.25ex\right\vert\kern-0.25ex\right\vert}}

\begin{document}
	
    \title[Water waves with compactly supported vorticity I]{Capillary--gravity water waves with compactly supported vorticity\\ I. Local constructions and regular drift}
	
	\author{Lili Du}
	\address{Department of Mathematics, Sichuan University, Chengdu, Sichuan 610064, P.R. China}
		\email{dulili@scu.edu.cn}
	
	\author{Changjun Zou}
	\address{Department of Mathematics, Sichuan University, Chengdu, Sichuan 610064, P.R. China}
	\email{zouchangjun@amss.ac.cn}

    \begin{abstract}
    	This five-part series establishes the first systematic theory of capillary--gravity water waves carrying compactly supported vorticity with fixed positive circulation. Treating the free surface and the vortex as coupled unknowns, we study the well-posedness of traveling-wave equilibria, including the existence and continuation of solution branches with respect to bifurcation parameters, as well as their variational structure, geometric rigidity, canonical classification, and conditional dynamics. In this first paper of the series, we construct small-speed solitary waves carrying a point vortex of fixed circulation, where the Bernoulli equation on the free surface and the regular-drift balance for the interior point vortex are solved simultaneously through a joint fixed-point argument. For every wave speed below a fixed positive threshold, we derive the free-surface expansion through fourth order, including the first capillary correction, and distinct second-order corrections to the conformal and physical vortex depths. To prepare for the global continuation theory in Part II \cite{DuZouPartII}, we study stagnation points of the regular velocity field in prescribed fluid domains, where variational comparison arguments based on the Kirchhoff--Routh function yield existence and uniform localization. We then regularize the point vortex into a compactly supported vorticity core through a joint Lyapunov--Schmidt reduction, simultaneously determining the free surface and the vortex-core boundary while preserving the circulation and the prescribed vorticity law. In the periodic analogue, we construct a local branch of traveling waves carrying a point vortex under a nonresonance condition arising from the capillary--gravity dispersion relation. Along this branch, the vortex depth diverges logarithmically as the wave speed approaches a positive limit. This distinguishes the periodic setting from the nonperiodic case, where the vortex depth grows in inverse proportion to the wave speed as the speed tends to zero. These results establish the analytical framework and provide the starting point for the subsequent investigations in the series.
    \end{abstract}

	\maketitle
	
	\medskip
	\noindent\textbf{Keywords.}
	Capillary--gravity water waves; compactly supported vorticity; point vortices; Kirchhoff--Routh function; vortex regularization; regular drift; periodic deep-vortex waves

	\smallskip
	\noindent\textbf{2020 Mathematics Subject Classification.}
	Primary 76B15; Secondary 76B47, 35Q35, 35R35, 76B03.
	
	\clearpage
	\tableofcontents
	
	\section{Introduction}
	\label{sec1}
	
	This five-part series develops a systematic theory of capillary--gravity traveling waves carrying compactly supported vorticity at fixed positive circulation. The central problem is a genuinely coupled free-boundary problem: the free surface and the concentrated vorticity distribution must be determined simultaneously. The present paper, Part I, lays the analytical foundation for the series by developing the conformal formulation and the mechanism selecting the vortex position. It also constructs local point-vortex waves and regularizes them into smooth vortex cores. Building on these local constructions, Part II \cite{DuZouPartII} develops a global continuation theory for the coupled solution curve, characterizes possible loss of compactness, analyzes simple folds, and establishes the persistence of smooth-core regularization along compact solution arcs. Using the point-vortex solution curve as a reference, Part III \cite{DuZouPartIII} constructs regularized finite-core waves through a dual variational approach and identifies the effective Hessian governing the coupled equilibria. The resulting construction provides a variational counterpart to the perturbative regularization developed in Part II. Drawing on the preceding analytical and variational theory, Part IV \cite{DuZouPartIV} proves reflection symmetry and core rigidity, establishing uniqueness and canonical classification of concentrated waves. Finally, Part V \cite{DuZouPartV} addresses the time-dependent problem, establishing fixed-impulse coercivity and conditional orbital tracking on logarithmic time scales.
	
	This first paper in the series establishes the starting point for the subsequent developments. We study two-dimensional capillary--gravity traveling waves with compactly supported vorticity in an inviscid, incompressible fluid of infinite depth. In a frame moving with velocity $-We_1$, where $W>0$, $e_1=(1,0)$, and $e_2=(0,1)$, the free surface is stationary and the fluid occupies the unknown domain $\Omega=\{(\xi,y)\in\mathbb R^2:y<\eta(\xi)\}$. Let $v$ and $P$ denote the laboratory-frame velocity and pressure, and set $v_{\rm rel}:=v+We_1$. The steady Euler equations are
	\begin{equation*}
		\begin{cases}
			(v_{\rm rel}\cdot\nabla)v_{\rm rel}+\nabla P+ge_2=0,\\
			\nabla\cdot v_{\rm rel}=0.
		\end{cases}\tag{1a}\label{1a}
	\end{equation*}
	Writing $\omega=\operatorname{curl}v_{\rm rel}$, we obtain the steady vorticity equation
	\[
	(v_{\rm rel}\cdot\nabla)\omega=0,
	\]
	so that vorticity is constant along streamlines for smooth solutions. We allow either an interior point vortex or a smooth vortex core with compactly supported vorticity. The key mechanism selecting the vortex position is the \emph{regular drift}, which is obtained by subtracting the singular self-induced velocity and incorporates the Kirchhoff--Routh function and a conformal correction due to the free-surface geometry. Its vanishing gives the equilibrium condition in the moving frame.
	
	On the free surface, the kinematic condition and the pressure balance with surface tension take the form
	\begin{equation*}
		v_{\rm rel}\cdot n=0,\qquad P=\sigma\kappa\qquad\text{on }\partial\Omega.\tag{1b}\label{1b}
	\end{equation*}
	Here $\sigma>0$ is the surface-tension coefficient, $n=(-\eta',1)/\sqrt{1+(\eta')^2}$ is the upward unit normal, and $\kappa=-\eta''/(1+(\eta')^2)^{3/2}$ is the signed curvature. The steady Euler equations imply that the Bernoulli quantity $\frac12|v_{\rm rel}|^2+gy+P$ is constant along streamlines. Since the free surface is a streamline, we obtain
	\begin{equation*}
		\frac12|v_{\rm rel}|^2+gy+\sigma\kappa=C\qquad\text{on }\partial\Omega.\tag{1c}\label{1c}
	\end{equation*}
	For localized traveling waves satisfying $\eta(\xi)\to0$ and $v\to0$ as $|\xi|\to\infty$, the Bernoulli constant is $C=W^2/2$. It is worth noting that gravity and surface tension act on different Fourier scales. The linearized Bernoulli condition involves the operator $g-\sigma\partial_\xi^2$, whose Fourier symbol is $g+\sigma k^2$. Thus gravity dominates the low Fourier modes of the free surface, while surface tension dominates the high modes.
	
	The mathematical theory of irrotational water waves has a rich history, beginning with the pioneering work of Stokes, Levi-Civita, and Nekrasov; see also \cite{LeviCivita1925,Nekrasov1951,Toland1996,HaziotEtAl2022}. A landmark achievement was the resolution of the Stokes conjecture for extreme periodic gravity waves by Amick, Fraenkel, and Toland \cite{AmickFraenkelToland1982}, while Varvaruca and Weiss subsequently developed a powerful geometric approach that substantially broadened the scope of the theory \cite{VarvarucaWeiss2011}. In the pure-capillary setting, the celebrated exact solutions of Crapper and their extensions by Kinnersley provided fundamental examples of nonlinear waves of finite amplitude \cite{Crapper1957,Kinnersley1976}. Further advances in capillary--gravity waves include rigorous analytical results for two-fluid flows \cite{Sun1997}, the discovery of asymmetric solitary waves and rich symmetry-breaking phenomena in periodic waves \cite{WangVandenBroeckMilewski2014,GaoWangVandenBroeck2017}, and the development of conformal and numerical methods for spatially quasi-periodic traveling waves by Wilkening and Zhao \cite{WilkeningZhao2021JFM,WilkeningZhao2021JNS}. In the pure-gravity regime, Feola and Giuliani established the existence and linear stability of time-quasi-periodic traveling waves through a sophisticated Nash--Moser analysis \cite{FeolaGiuliani2024}. Other important contributions concern edge waves \cite{Constantin2001}, particle trajectories \cite{ConstantinEscher2007,Constantin2010}, and nonexistence and asymptotic results for deep-water solitary waves \cite{Hur2012,Wheeler2018}.
	
	Gerstner's celebrated exact solution \cite{Gerstner1809}, later studied by Henry \cite{Henry2008}, provides a classical example of a rotational water wave. Subsequent work examined how shear currents affect nonlinear waves, both in deep water \cite{SimmenSaffman1985} and at finite depth \cite{TelesDaSilvaPeregrine1988}. A landmark contribution to the rigorous theory was made by Constantin and Strauss \cite{ConstantinStrauss2004}, who established a global continuum of periodic gravity waves with vorticity through a combination of bifurcation and degree theory. A further major advance was achieved by Constantin, Strauss, and Varvaruca \cite{ConstantinStraussVarvaruca2016}, whose Riemann--Hilbert formulation provided a powerful framework for analytic global bifurcation, accommodating large-amplitude waves and critical layers. Other important developments include the existence theory of Groves and Wahl\'en \cite{GrovesWahlen2008} for small-amplitude periodic and solitary gravity waves with general vorticity distributions, and the variational approach of Burton and Toland \cite{BurtonToland2011} to weak steady water waves. Flows with discontinuous vorticity and stagnation points were investigated in \cite{MartinMatioc2016}. More recently, Gui, Wang, Yang, and Zhang \cite{GuiWangYangZhang2026} made a significant extension to deep-water Stokes waves with piecewise smooth vorticity, establishing a global bifurcation theory in the presence of internal streamline interfaces and the associated non-Fredholm difficulties.
	
	Symmetry and the geometry of the interior flow have also received considerable attention. A fundamental contribution was made by Constantin, Ehrnstr\"om, and Wahl\'en \cite{ConstantinEhrnstromWahlen2007}, who established symmetry of periodic gravity waves with monotone surface profiles without imposing restrictions on the vorticity distribution. Ehrnstr\"om and Villari \cite{EhrnstromVillari2008} demonstrated how constant vorticity can profoundly alter particle trajectories and generate internal vortical structures even in the linearized setting. For nonlinear waves, Wahl\'en \cite{Wahlen2009} constructed steady waves with a critical layer. A significant breakthrough was achieved by Ehrnstr\"om, Escher, and Wahl\'en \cite{EhrnstromEscherWahlen2011}, who established the existence of small-amplitude waves with arbitrarily many critical layers, revealing the remarkable complexity of rotational flows beneath the free surface. Building on these developments, Varholm \cite{Varholm2020} made a major advance by establishing an analytic global bifurcation theory for waves with multiple critical layers and arbitrarily many interior stagnation points. Rotational flows also support a rich variety of multimodal wave patterns. Aasen and Varholm \cite{AasenVarholm2018} constructed unimodal and bimodal waves with affine vorticity, Ehrnstr\"om and Wahl\'en \cite{EhrnstromWahlen2015} obtained trimodal waves, and Kozlov and Lokharu \cite{KozlovLokharu2018} established the existence of $N$-modal waves.
	
	The free surface of a large-amplitude rotational wave need not remain a graph, revealing a remarkably rich geometric structure beyond classical Stokes waves. In a striking breakthrough, Hur and Wheeler \cite{HurWheeler2022} gave the first rigorous existence proof of overhanging and touching periodic gravity waves with constant vorticity, including a wave whose surface touches itself tangentially and encloses a bubble of air. Wahl\'en and Weber \cite{WahlenWeber2024} subsequently developed a powerful global bifurcation framework for gravity waves with general vorticity, allowing overhanging surfaces, stagnation points, and critical layers without imposing restrictive geometric assumptions. In a different direction, Berti, Franzoi, and Maspero made a pioneering contribution to time-quasi-periodic traveling waves, establishing their existence for capillary--gravity flows with constant vorticity \cite{BertiFranzoiMaspero2021} and subsequently extending the theory to the pure-gravity regime through a sophisticated Nash--Moser analysis \cite{BertiFranzoiMaspero2024}.
	
	Early studies of submerged vortices interacting with a free surface include those of Ter-Krikorov and Filippov \cite{TerKrikorov1958,Filippov1960,Filippov1961}. Forbes \cite{Forbes1985} investigated the nonlinear deformation of the free surface induced by a submerged point vortex through numerical computations. An important step toward a fully coupled description was taken by Elcrat and Miller \cite{ElcratMiller2006}, who computed finite-amplitude waves by simultaneously enforcing the Bernoulli equation and the point-vortex equilibrium condition through conformal mapping and Kirchhoff--Routh theory. A decisive advance in the rigorous theory was achieved by Shatah, Walsh, and Zeng \cite{ShatahWalshZeng2013}, who established the existence of solitary capillary--gravity waves with point vortices or compactly supported vorticity through small-circulation bifurcation, building on the traveling point-vortex solution in the flat Euler half-plane. They also constructed a global continuum of finite-amplitude periodic point-vortex waves. Varholm \cite{Varholm2016} extended the solitary-wave construction to finite depth with finitely many point vortices. Important developments in stability theory were made by Varholm, Wahl\'en, and Walsh \cite{VarholmWahlenWalsh2020}, who established conditional orbital stability for sufficiently small point-vortex waves, while Le \cite{Le2019} constructed solitary waves with a finite dipole and proved their conditional orbital instability. Another substantial contribution was made by Chen, Walsh, and Wheeler \cite{ChenWalshWheeler2019}, who developed a detailed theory of deep-water solitary waves with localized vorticity, obtaining precise far-field asymptotics, nonexistence results, and new existence families of capillary--gravity waves with compactly supported vorticity. More recently, Shelton and Trinh \cite{SheltonTrinh2023} employed exponential asymptotics to uncover the generation of exponentially small surface waves by submerged point vortices at low Froude numbers.
	
	Two recent developments are particularly relevant to the present work. A major advance was achieved by Chen, Varholm, Walsh, and Wheeler \cite{ChenVarholmWalshWheeler2025}, who established global branches of finite-depth solitary gravity waves carrying a submerged point vortex, bifurcating from an irrotational uniform flow without imposing a graph condition on the free surface. Their framework also accommodates overhanging waves, and at nondegenerate points along these branches they further constructed hollow-vortex waves, replacing the point vortex by a small cavity around which the fluid circulates. Slease, Walsh, and Zhong \cite{SleaseWalshZhong2026} recently studied the existence and conditional orbital stability of capillary--gravity waves carrying point vortices in finite depth. In a striking development of a different kind, D\'avila, del Pino, Musso, and Wheeler \cite{DavilaDelPinoMussoWheeler2026} gave the first rigorous construction of overhanging solitary gravity waves with constant vorticity through an ingenious gluing method, producing fluid domains resembling a disk joined to a strip by a thin neck. In the point-vortex formulation of \cite{ChenVarholmWalshWheeler2025}, the Helmholtz--Kirchhoff equilibrium condition requires the regular velocity, after removal of the singular self-induced field, to vanish at the vortex position in the moving frame. This regular drift depends on the unknown free-surface geometry, including the contribution of the conformal factor associated with the fluid domain. Our analysis makes this geometric dependence explicit in the regime of fixed circulation and small wave speed, and subsequently regularizes the point vortex into a smooth, fluid-filled vortex core. 
	
	A particularly ingenious construction was introduced by Ehrnstr\"om, Walsh, and Zeng \cite{EhrnstromWalshZeng2023}, who obtained stationary capillary--gravity waves in finite depth with smooth, localized, sign-changing vorticity. By adapting spike-layer methods, they exploited the near cancellation of positive and negative vorticity to achieve an effective decoupling between the interior vortical flow and the free surface: both the total circulation and surface deformation are exponentially small, while the kinetic energy remains of order one. This mechanism is distinct from the fixed, nonzero circulation regime considered here.
	
	We now consider capillary--gravity solitary waves in infinite depth carrying a submerged point vortex of fixed positive circulation $\Gamma$, with $W\to0^+$. For a point vortex in the Euler flow on the flat lower half-plane, the equilibrium condition in a frame moving with velocity $-We_1$ is
	\[
	W=\frac{\Gamma}{4\pi a},\qquad a_0(W):=\frac{\Gamma}{4\pi W}.
	\]
	As $W\to0^+$, the vortex depth $a_0(W)$ tends to infinity. The growing separation from the free surface weakens the vortex--surface interaction, while the circulation remains unchanged. This simple balance determines the leading-order vortex depth in our construction of coupled solitary waves.
	
	For a nonflat free surface, the regular-drift balance is modified by the conformal geometry of the fluid domain. Let $Z$ be the conformal map from the lower half-plane onto the fluid domain, and let $z_a=-\mathrm{i}a$ represent the vortex position. The Helmholtz--Kirchhoff condition \cite{Lin1941I,Lin1941II,MarchioroPulvirenti1983,MarchioroPulvirinti1994,Newton2001} then takes the form
	\[
	W-\frac{\Gamma}{4\pi a}+\frac{\mathrm{i}\Gamma}{4\pi}\frac{Z''(z_a)}{Z'(z_a)}=0.
	\]
	Compared with the flat half-plane balance, the additional term involving $Z''/Z'$ describes how the geometry of the free surface modifies the regular drift. Since the conformal map is also determined by the Bernoulli condition on the free surface, we must solve the two equations simultaneously to determine both the surface profile and the vortex position.
	
	To construct the coupled waves, we use Sobolev spaces with norms adapted to the rescaled capillary coefficient $\beta_W=\sigma/(ga_0(W)^2)$. Although $\beta_W\to0$ as $W\to0^+$, these norms yield uniform inverse estimates and the required expansions of the free surface and vortex depths. Positive surface tension is essential to this argument: when $\sigma=0$, the far-field symbol $g-W^2|\xi|$ vanishes at $|\xi|=g/W^2$, and the resulting resonance requires a separate analysis.
	
	More generally, consider $N$ point vortices at $z_1,\ldots,z_N\in\Pi_-:=\{z\in\mathbb C:\Im z<0\}$ with circulations $\Gamma_1,\ldots,\Gamma_N$ in a prescribed fluid domain. Their equilibria in the traveling frame are characterized by the critical points of the Kirchhoff--Routh function
	\[
	\mathscr K_{Z,W}^{(N)}=\sum_{j=1}^N\frac{\Gamma_j^2}{4\pi}\log\!\bigl(2|\Im z_j|\,|Z'(z_j)|\bigr)+\sum_{1\le j<\ell\le N}\frac{\Gamma_j\Gamma_\ell}{2\pi}\log\frac{|z_j-\overline{z_\ell}|}{|z_j-z_\ell|}+W\sum_{j=1}^N\Gamma_j\Im z_j,
	\]
	that is, $\nabla_{z_j}\mathscr K_{Z,W}^{(N)}=0$ for $j=1,\ldots,N$. The pairwise terms describe the effects of the other vortices and their odd reflections across the real axis, while the factors $|Z'(z_j)|$ account for the geometry of the prescribed domain. For simplicity, we study the single-vortex problem in a symmetric domain, where reflection symmetry places the vortex center on the symmetry axis and reduces the equilibrium condition to a one-dimensional critical-point problem. Using variational comparisons, we establish the existence of a maximizing equilibrium at every positive speed and prove its uniqueness at sufficiently small speeds. For uniformly admissible families and speeds bounded away from zero and infinity, we show that all symmetric single-vortex equilibria remain a uniform distance from the boundary and have uniformly bounded depths. Under an additional far-field expansion of $Z$, we derive sharper asymptotics for the equilibrium depth as the speed tends to zero.
	
	A related center-selection problem was studied by Wang, Yan, and Zhou \cite{WangYanZhou2026}, who constructed asymmetric periodic water waves in finite depth with concentrated vortex cores whose circulations tend to zero. Their leading-order center equations are determined by the limiting flat strip. The fixed-circulation framework developed here provides a starting point for constructing non-symmetric waves with several vortices of fixed nonzero circulations. The localization estimates established above also provide a priori control for the global continuation of the coupled point-vortex branch studied in Part~II \cite{DuZouPartII}.
	
	Following the coupled point-vortex construction, a natural question is whether the singular vortex in a coupled point-vortex wave can be replaced by a smooth core without changing its circulation. We prescribe a vorticity law independent of $\varepsilon$ and choose its level so that the total circulation remains exactly $\Gamma$. A joint Lyapunov--Schmidt reduction determines the free surface and the vortex-core boundary simultaneously. Crucially, our construction starts from a nonflat point-vortex wave whose regular-drift balance already incorporates the conformal geometry of the free surface. The remaining depth equation is then a small perturbation of this balance.
	
	The periodic setting completes the picture of free-surface--vortex coupling. A vortex in one period cell corresponds to an infinite periodic row in the covering plane, which changes the far-field drift. Unlike the solitary setting, where the flat-boundary contribution vanishes as the vortex moves to large depth, the periodic row produces a nonzero limiting speed $W_c=\frac{\Gamma}{2L}$. The deep-vortex limit therefore occurs as $W\to W_c$ rather than as $W\to0^+$. Under a nonresonance condition arising from the capillary--gravity dispersion relation, we construct a local branch near $W_c$ and show that the vortex depth diverges logarithmically along this branch, in contrast to the inverse proportionality between depth and speed in the solitary case.
	
	We now state the main results, beginning with the coupled solitary-wave construction.
	
	\smallskip
	
	\subsection{Main results}
	
	We keep the gravitational acceleration $g>0$, surface tension $\sigma>0$, and circulation $\Gamma>0$ fixed. Write $p=(p_1,p_2)$ for points in the physical fluid domain and $dp$ for its area measure. For solitary waves, the reference domain is the lower half-plane $\Pi_-:=\{z=x+\mathrm{i}y\in\mathbb C:y<0\}$, where $x$ denotes the horizontal conformal coordinate and $\xi$ denotes the physical horizontal coordinate. We introduce a conformal formulation that represents different free surfaces on the same reference domain while keeping the physical fluid region unknown through the conformal map.
	
	The free surface is described by a normalized conformal map $Z_Y:\Pi_-\to\Omega_Y$ with boundary trace
	\[
	Z_Y(x)=\Xi_Y(x)+\mathrm{i}Y(x),\qquad \Xi_Y(x):=x-HY(x).
	\]
	Thus $Y$ represents the surface profile in conformal coordinates, while the physical surface is $\eta_Y=Y\circ\Xi_Y^{-1}$. We require the graph condition $c_Y:=\inf_{\mathbb R}\Xi_Y'=\inf_{\mathbb R}(1+|D|Y)>0$, which guarantees that the conformal parametrization is nondegenerate and that the physical surface remains a graph. This condition controls the geometry of the free boundary without imposing monotonicity on the surface profile. 
	For $s>\frac12$, we define the admissible surface space
	\[
	\mathcal O^s:=\left\{Y\in H_{\rm e}^{s+2}(\mathbb R):c_Y>0\right\},
	\]
	and the corresponding configuration space
	\[
	\mathscr U^s:=\mathcal O^s\times(0,\infty)\times(0,\infty),
	\]
	whose coordinates are $(Y,a,W)$.
	
	The point vortex is located at $-\mathrm{i}a$ in the reference half-plane and at $p_a:=Z_Y(-\mathrm{i}a)$ in the physical domain. For the symmetric configurations considered below, we write $p_a=(0,-h)$, where $a$ is the conformal depth and $h$ is the physical depth measured from the undisturbed surface. These two quantities agree in the flat configuration but differ in general because the conformal map changes the physical position of the vortex. The distinction between them will be essential in the small-speed asymptotics.
	
	The surface tension coefficient plays an essential role in the present construction. In the rescaled small-speed problem, the highest-order capillary contribution remains a singularly perturbed elliptic term, and the positive surface tension allows us to obtain uniform invertibility estimates as the vortex depth tends to infinity. The pure-gravity case has a different far-field resonance structure and is not covered by the present argument.
	
	For $Y\in\mathcal O^s$, the normalized conformal map satisfies $Z_Y(z)-z\to0$ and $Z_Y'(z)\to1$ as $|z|\to\infty$ in $\overline{\Pi_-}$. We set
	\[
	J(Y):=|Z_Y'|^2=(1+|D|Y)^2+(Y')^2,
	\]
	and use the curvature convention
	\[
	\kappa(Y):=-\frac{(1+|D|Y)Y''-Y'(|D|Y)'}{J(Y)^{3/2}}.
	\]
	We also define
	\[
	E_ah:=\bigl(e^{-a|D|}h\bigr)(0),\qquad q_{a,W}(x):=W-\frac{\Gamma a}{\pi(x^2+a^2)}.
	\]
	Here $E_a$ evaluates the harmonic extension at the reference vortex point. Consequently, the conformal geometry of the free surface enters both the Bernoulli condition and the velocity balance selecting the vortex position.
	
	The coupled point-vortex equations are expressed through the Bernoulli residual and the vortex-equilibrium residual:
	\begin{equation}
		\begin{aligned}
			B(Y,a,W)&:=gY+\sigma\kappa(Y)+\frac{q_{a,W}^2}{2J(Y)}-\frac{W^2}{2},\\
			P(Y,a,W)&:=W-\frac{\Gamma}{4\pi a}-\frac{\Gamma}{4\pi}\frac{E_a(Y'')}{1+E_a(|D|Y)}.
		\end{aligned}
		\label{1-1}
	\end{equation}
	The equation $B=0$ imposes Bernoulli's condition on the free surface. The equation $P=0$ expresses the vanishing of the regular drift at the vortex: after removing the singular self-induced velocity, the remaining finite velocity must vanish in the traveling frame. The first two terms in $P$ give the flat-boundary balance, while the last term records the correction produced by the deformed free surface through the conformal geometry. Hence, unlike a Kirchhoff--Routh problem in a prescribed domain, the drift condition here is coupled to the unknown free boundary.
	
	By symmetry, the regular relative velocity at the vortex is horizontal, with horizontal component
	\[
	\frac{P(Y,a,W)}{Z_Y'(-\mathrm{i}a)}.
	\]
	Since $Z_Y'(-\mathrm{i}a)=1+E_a(|D|Y)>0$, the condition $P=0$ is exactly the requirement that the vortex remain stationary in the moving frame. After reconstructing the velocity from the stream function, prescribed vorticity, and far-field normalization, the point-vortex traveling-wave problem is equivalent to
	\begin{equation}
		\mathscr F(Y,a,W):=
		\begin{pmatrix}
			B(Y,a,W)\\
			P(Y,a,W)
		\end{pmatrix}
		=0,\qquad (Y,a,W)\in\mathscr U^s .
		\label{1-2}
	\end{equation}
	This formulation separates the free-boundary equation from the vortex-center selection condition while preserving their mutual dependence: the Bernoulli equation determines the surface geometry, and the regular-drift equation determines the admissible vortex position in that geometry.
	
	\smallskip
	\noindent\textit{Small-speed point-vortex waves.}
	Our first result constructs the free surface and vortex position simultaneously in the weak-interaction regime. With the circulation $\Gamma$ fixed, the flat half-plane equilibrium in a frame moving with velocity $-We_1$ gives the distinguished depth
	\[
	a_0(W):=\frac{\Gamma}{4\pi W}.
	\]
	As $W\to0^+$, this depth diverges and also sets the horizontal scale of the wave.
	
	For a vortex at the flat reference depth $a_0(W)$, the leading surface profile is determined by the residual of the Bernoulli equation in the first equation of \eqref{1-1}, evaluated at the flat half-plane configuration. After rescaling the horizontal variable by $X=x/a_0(W)$, the kinetic part of the Bernoulli residual generated by the point vortex is
	\[
	4W^2Q(X),\qquad Q(X):=\frac{1-X^2}{(1+X^2)^2}.
	\]
	Since the leading surface amplitude is of order $W^2$, the curvature contribution is negligible at this order and the Bernoulli balance reduces to $g\eta+4W^2Q=0$. Hence the leading free-surface profile is
	\[
	\eta\sim-\frac{4W^2}{g}Q(X).
	\]
	The function $-Q$ has a central depression and two smaller elevated shoulders before decaying to the undisturbed level, and the parameters $g$ and $\Gamma$ determine the vertical amplitude and horizontal length scale through $4W^2/g$ and $a_0(W)$, respectively.
	
	The capillary term first contributes at order $W^4$, correcting the leading trough--shoulder surface profile. The vortex equilibrium condition gives the conformal depth, while the physical depth is obtained by mapping the vortex position back to the fluid domain. Their difference is
	\[
	h_W-a_W=\frac{W^2}{g}+O(W^4),
	\]
	reflecting the displacement of the free surface.
	
	\begin{theorem}[Small-speed point-vortex waves]
		\label{thm1}
		Fix $g,\sigma,\Gamma>0$. There exists $W_0=W_0(g,\sigma,\Gamma)>0$ such that, for every $0<W\le W_0$, there is a symmetric solitary capillary--gravity wave traveling with velocity $-We_1$ and carrying a single point vortex
		\[
		\omega_W=\Gamma\delta_{p_W},\qquad p_W=(0,-h_W).
		\]
		The free surface is an even graph $p_2=\eta_W(\xi)$ with $\eta_W(\xi)\to0$ as $|\xi|\to\infty$. Let $a_W$ and $h_W$ denote the conformal and physical vortex depths, respectively, and let $a_0(W)$ and $Q$ be as defined above. The following assertions hold.
		
		\begin{enumerate}
			\item[(i)] As $W\to0^+$, the conformal and physical vortex depths satisfy
			\[
			\begin{aligned}
				a_W&=\frac{\Gamma}{4\pi W}+\frac{3}{2g}W^2+O(W^4),\\
				h_W&=\frac{\Gamma}{4\pi W}+\frac{5}{2g}W^2+O(W^4).
			\end{aligned}
			\]
			
			\item[(ii)] For every fixed integer $m\ge0$, the physical surface has the expansion
			\begin{equation}
				\eta_W(a_0(W)X)=-\frac{4W^2}{g}Q(X)-\frac{64\pi^2\sigma}{g^2\Gamma^2}W^4Q''(X)+O_{H^m}(W^5).
				\label{1-3}
			\end{equation}
			The remainder is measured in $H^m(\mathbb R)$ with respect to the rescaled variable $X$.
		\end{enumerate}
	\end{theorem}
	
	Our first theorem constructs symmetric solitary capillary--gravity waves carrying a point vortex of fixed circulation and gives precise asymptotic expansions for the free surface and vortex depths as the wave speed tends to zero. In the normalized case $g=\sigma=\Gamma=1$, the quantitative contraction estimate \eqref{3-9a} allows us to take $W_0(1,1,1)\approx\frac{1}{12\pi}$.
	
	Figure~\ref{fig:newton} illustrates the relative scales of the free surface and vortex position predicted by Theorem~\ref{thm1}. We take $g=\sigma=\Gamma=1$ and $W=(12\pi)^{-1}$, so that the leading flat-balance depth is $a_0=3$. We solve the coupled Bernoulli and point-vortex equations numerically and reconstruct the resulting wave in physical coordinates. The left panel shows the full configuration, where the surface deformation is barely visible on the scale of the vortex depth. The right panel enlarges the surface region, revealing the central depression and two elevated shoulders described by the leading-order profile $-Q$.
	
	\begin{figure}[!htbp]
		\centering
		\begin{minipage}[t]{0.485\textwidth}
			\centering
			\vspace{0pt}
			\includegraphics[width=\linewidth]{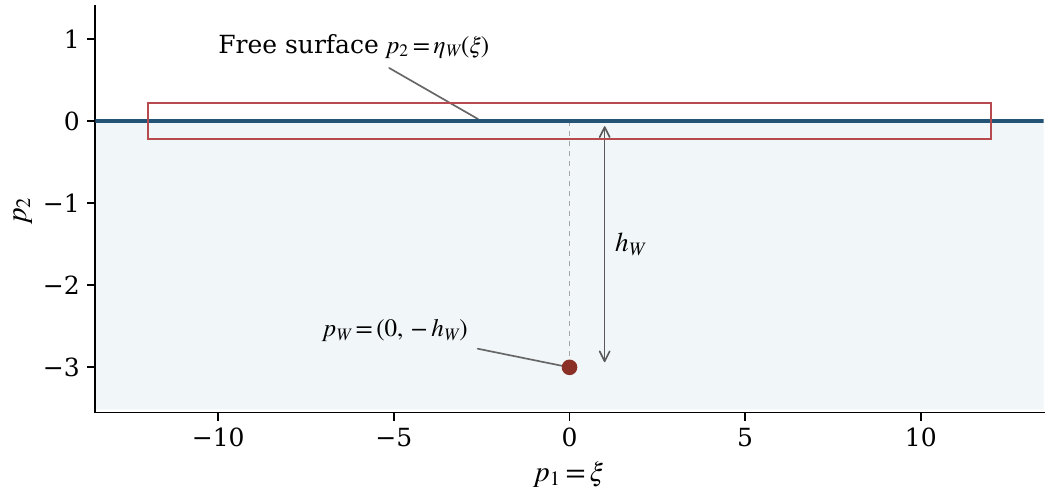}
			\smallskip
			{\small (a) Free surface and point vortex.}
		\end{minipage}
		\hfill
		\begin{minipage}[t]{0.485\textwidth}
			\centering
			\vspace{0pt}
			\includegraphics[width=\linewidth]{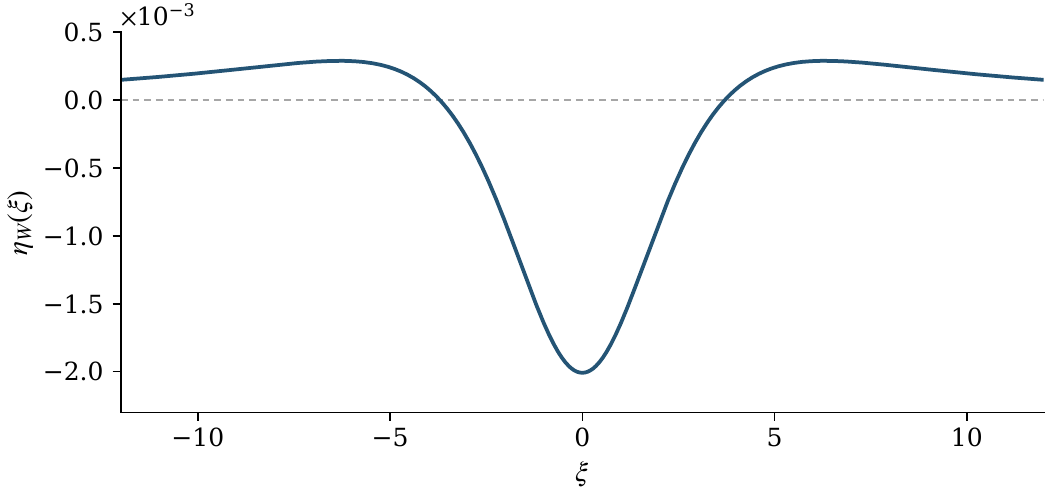}
			\smallskip
			{\small (b) Enlargement of the red-box region in (a).}
		\end{minipage}
		\caption{A numerical point-vortex wave for $g=\sigma=\Gamma=1$ and $W=(12\pi)^{-1}$, for which $a_0=3$, $h_W\approx3.00136$, and $\eta_W(0)\approx-2.00861\times10^{-3}$. Panel (b) enlarges the red-box region in (a).}
		\label{fig:newton}
	\end{figure}
	\FloatBarrier
	
	\smallskip
	\noindent\textit{Vortex equilibria in prescribed domains.} Theorem~\ref{thm1} treats the fully coupled problem in the small-speed regime, where the vortex lies far below the free surface. To investigate the equilibrium mechanism at more general wave speeds, we temporarily decouple the vortex-position equation from the surface Bernoulli equation by prescribing the fluid domain. 
	
	In conformal coordinates, the Kirchhoff--Routh function incorporates the prescribed domain geometry through the factor $|Z'|$, and its critical-point condition recovers the regular-drift balance. This provides a variational formulation of the vortex equilibrium problem in a fixed domain. We work with the class $\mathfrak A^{k,\alpha}$ of admissible symmetric asymptotically flat domains introduced in Definition~\ref{def4-1}. Each domain is described by a normalized conformal map $Z:\Pi_-\to\Omega$ with a $C^{k,\alpha}$ boundary extension and $Z'$ bounded above and away from zero. The boundary need not be a graph, and no Bernoulli condition is imposed. For uniformly admissible families, as defined in Definition~\ref{def4-20}, the corresponding bounds hold uniformly across the family. This uniformity allows us to control the location of all symmetric vortex equilibria at speeds bounded away from zero, and to obtain uniform small-speed asymptotics for maximizing equilibria.
	
	\begin{theorem}[Vortex equilibria in prescribed domains]
		\label{thm2}
		Let $k\ge3$, $\alpha\in(0,1)$, and $\Gamma>0$. Let $\Omega\in\mathfrak A^{k,\alpha}$ be an admissible symmetric asymptotically flat domain, represented by its normalized conformal map $Z:\Pi_-\to\Omega$. For each $W>0$, define the Kirchhoff--Routh function on the symmetry axis by
		\[
		K_{\Omega,W}(a):=\frac{\Gamma^2}{4\pi}\log\bigl(2aZ'(-\mathrm{i}a)\bigr)-\Gamma Wa,\qquad a>0.
		\]
		For each depth $a>0$, write
		\[
		p_a:=Z(-\mathrm{i}a)=(0,-h(a)).
		\]
		The critical points of $K_{\Omega,W}$ correspond to symmetric point-vortex equilibria at $p_a$ in the frame moving with velocity $-We_1$. Define the set of maximizing depths by
		\[
		\mathcal A_{\Omega,W}:=\left\{a>0:K_{\Omega,W}(a)=\max_{s>0}K_{\Omega,W}(s)\right\}.
		\]
		Then the following assertions hold.
		
		\begin{enumerate}
			\item[(i)] For every $W>0$, the Kirchhoff--Routh function attains its global maximum, and $\mathcal A_{\Omega,W}$ is nonempty and compact. Every $a\in\mathcal A_{\Omega,W}$ gives a symmetric point-vortex equilibrium.
			
			\item[(ii)] Let $\mathfrak F\subset\mathfrak A^{k,\alpha}$ be a uniformly admissible family in the sense of Definition~\ref{def4-20}, and fix $0<W_-<W_+<\infty$. There exist constants $0<a_*<a^*<\infty$ and $0<d_*<d^*<\infty$, depending only on $\Gamma$, $W_\pm$, and the uniform admissibility bounds, such that every symmetric equilibrium in every $\Omega\in\mathfrak F$, with $W\in[W_-,W_+]$, satisfies
			\[
			a_*\le a\le a^*,\qquad
			d_*\le\operatorname{dist}(p_a,\partial\Omega)\le d^*.
			\]
			Thus, throughout any such family, equilibria cannot approach the boundary or escape to infinite depth when the wave speed stays bounded away from zero and infinity.
			
			\item[(iii)] For each fixed $\Omega\in\mathfrak A^{k,\alpha}$, there exists $W_\Omega>0$ such that the maximizing depth is unique whenever $0<W<W_\Omega$. Denoting it by $a_{\Omega,W}$ and writing $h_{\Omega,W}:=h(a_{\Omega,W})$, we have
			\[
			a_{\Omega,W}=\frac{\Gamma}{4\pi W}+o(1),\qquad
			h_{\Omega,W}=\frac{\Gamma}{4\pi W}+o(1)
			\]
			as $W\to0^+$. For uniformly admissible families, these estimates hold uniformly over the domains and their maximizing depths.
			
			\item[(iv)] Suppose, in addition, that the normalized conformal map of a fixed domain $\Omega$ has the far-field expansion
			\[
			Z(z)=z+\frac{c_1}{z}+\frac{c_2}{z^2}+\frac{c_3}{z^3}+\frac{c_4}{z^4}+R(z),
			\]
			where $R^{(\ell)}(z)=O(|z|^{-5-\ell})$ for $\ell=0,1,2,3$ as $|z|\to\infty$ in $\overline{\Pi_-}$. Symmetry implies $c_1,c_3\in\mathbb R$ and $c_2,c_4\in\mathrm{i}\mathbb R$. Then the unique maximizing depth from part~{\rm(iii)} and its physical counterpart satisfy
			\[
			\begin{aligned}
				a_{\Omega,W}={}&\frac{\Gamma}{4\pi W}-\frac{8\pi c_1}{\Gamma}W-6\mathrm{i}c_2\left(\frac{4\pi W}{\Gamma}\right)^2\\
				&+(12c_3-6c_1^2)\left(\frac{4\pi W}{\Gamma}\right)^3+\mathrm{i}(20c_4-50c_1c_2)\left(\frac{4\pi W}{\Gamma}\right)^4+O(W^5),
			\end{aligned}
			\]
			and
			\[
			\begin{aligned}
				h_{\Omega,W}={}&\frac{\Gamma}{4\pi W}-\frac{12\pi c_1}{\Gamma}W-7\mathrm{i}c_2\left(\frac{4\pi W}{\Gamma}\right)^2\\
				&+(13c_3-8c_1^2)\left(\frac{4\pi W}{\Gamma}\right)^3+\mathrm{i}(21c_4-60c_1c_2)\left(\frac{4\pi W}{\Gamma}\right)^4+O(W^5).
			\end{aligned}
			\]
			These higher-order expansions require the additional far-field assumption above; the conclusions in parts~{\rm(i)}--{\rm(iii)} do not.
		\end{enumerate}
	\end{theorem}
	
	Although the maximizing set in part~{\rm(i)} is nonempty and compact, it need not consist of a single point at general speeds, where strict concavity is not guaranteed. The bounds in part~{\rm(ii)} control every symmetric equilibrium for uniformly admissible families and speeds in a compact subset of $(0,\infty)$, ruling out both boundary collision and escape to infinite depth and thereby providing the a priori location control needed for the global continuation theory developed in Part~II \cite{DuZouPartII}.
	
	In the small-speed limit, the maximizing equilibrium admits sharper depth estimates. Part~{\rm(iii)} shows that the maximizing equilibrium is unique for sufficiently small speeds under the basic admissibility assumptions, with both conformal and physical depths differing from $\Gamma/(4\pi W)$ by $o(1)$. The same depth estimates hold uniformly over maximizing equilibria in uniformly admissible families. Thus the flat half-plane balance determines the leading depth even for nonflat boundaries, without requiring a quantitative far-field expansion of the conformal map.
	
	Part~{\rm(iv)} makes the higher-order geometric dependence explicit. Under the additional Laurent expansion, both depths admit corrections through order $W^4$ in terms of the coefficients $c_j$. As in Theorem~\ref{thm1}, the regular-drift equation selects the conformal depth, while evaluation of $Z$ gives the physical depth. For a fixed domain, the first correction is of order $W$ when $c_1\ne0$, whereas along the coupled branch of Theorem~\ref{thm1}, the varying free surface produces corrections of order $W^2$. The two results therefore describe the same equilibrium mechanism in different geometric regimes. Section~\ref{sec4} also extends the formulation to non-symmetric configurations and several vortices with fixed nonzero circulations.
	
	\smallskip
	\noindent\textit{Regularization by smooth concentrated vorticity.}
	We now regularize the point-vortex waves constructed in Theorem~\ref{thm1} by replacing the singular vortex with a smooth concentrated vorticity profile. Fix $g,\sigma,\Gamma>0$ and a wave speed $0<W\le W_0$, and prescribe a vorticity law $f$ independent of the core scale $\varepsilon$. We describe the flow by its relative stream function $\Psi_\varepsilon$, with relative velocity and vorticity given by
	\[
	v_{\mathrm{rel},\varepsilon}=-\nabla^\perp\Psi_\varepsilon,\qquad
	\omega_\varepsilon=-\Delta\Psi_\varepsilon
	=\varepsilon^{-2}f(\Psi_\varepsilon-\mu_\varepsilon).
	\]
	The free surface $\partial\Omega_\varepsilon$ and the boundary $\partial D_\varepsilon$ of the vortex core are both unknown. On the free surface, $\Psi_\varepsilon=0$ and
	\[
	\frac12|\nabla\Psi_\varepsilon|^2+gp_2+\sigma\kappa_\varepsilon=\frac{W^2}{2},
	\]
	whereas the core boundary is the level curve $\Psi_\varepsilon=\mu_\varepsilon$. The flux constant $\mu_\varepsilon$ is chosen so that the total circulation is exactly $\Gamma$. Thus $\varepsilon\to0^+$ is taken around a fixed nonflat point-vortex wave, with the wave speed and circulation unchanged. Figure~\ref{fig:core-regularization} illustrates the two free boundaries and the limiting vortex-core profile.
	
	\begin{figure}[!htbp]
		\centering
		\captionsetup{font=small,labelfont=sc,width=\textwidth,justification=justified,singlelinecheck=false}
		\includegraphics[width=\textwidth]{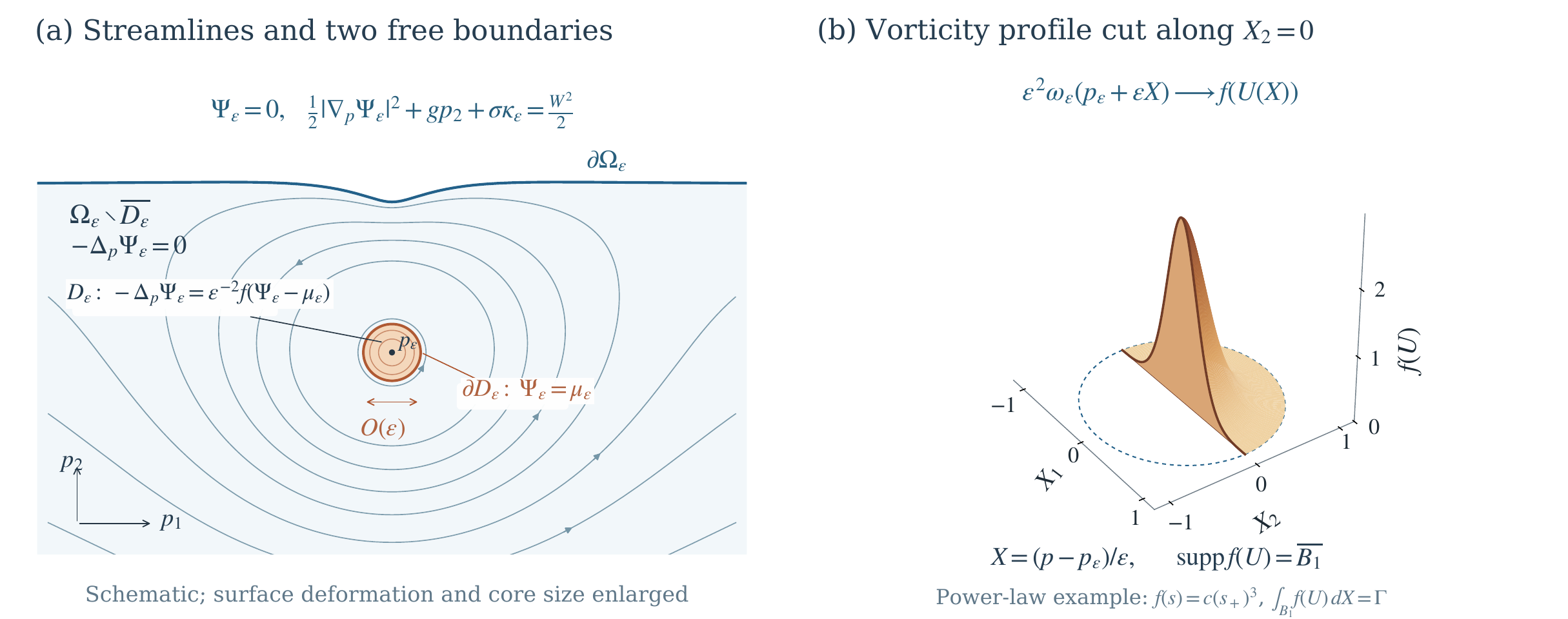}
		\caption{Smooth-core regularization. (a) Schematic streamlines, the free surface $\Psi_\varepsilon=0$, and the vortex boundary $\Psi_\varepsilon=\mu_\varepsilon$; the surface deformation and core size are enlarged for visibility. (b) A cutaway of the rescaled vorticity $f(U)$ along $X_2=0$, computed with $f(s)=c(s_+)^3$, unit core radius, and $\Gamma=1$.}
		\label{fig:core-regularization}
	\end{figure}
	
	The regularized problem involves two unknown free interfaces. The outer free boundary is the water surface $\partial\Omega_\varepsilon$, where $\Psi_\varepsilon=0$ and the Bernoulli condition holds. The inner free boundary is $\partial D_\varepsilon=\{\Psi_\varepsilon=\mu_\varepsilon\}$, separating the active vortex region from the surrounding irrotational fluid. The vorticity is determined by the prescribed law $f$, which vanishes on $(-\infty,0]$ and is positive on $(0,\infty)$. Thus the stream function satisfies a semilinear elliptic equation inside the core and is harmonic outside it. These two boundaries remain coupled: the concentrated vorticity determines the velocity entering the surface Bernoulli equation, while the resulting fluid geometry changes the regular drift that selects the core position. Neither interface is prescribed in advance.
	
	On the inner scale, the reference core is described by a radial function $U$ satisfying
	\[
	-\Delta U=f(U)\quad\text{in }\mathbb R^2,\qquad \{U>0\}=B_1,\qquad \int_{\mathbb R^2}f(U)\,dX=\Gamma.
	\]
	Thus $f(U)$ gives the limiting vorticity distribution after recentering and rescaling, with the core radius normalized to one. Assumption~\ref{ass5-1} specifies the required regularity and nondegeneracy of this profile, together with the transversality of the circulation constraint. In particular, translations are the only neutral modes of the isolated core; reflection symmetry removes the horizontal mode, while the remaining vertical mode will determine the core depth through a scalar solvability condition.

	\begin{theorem}[Smooth-core regularization at fixed circulation]
		\label{thm3}
		Fix $g,\sigma,\Gamma>0$ and a wave speed $0<W\le W_0$, where $W_0$ is the threshold in Theorem~\ref{thm1}. Let $(\Omega_W,p_W)$ be the point-vortex traveling wave given by Theorem~\ref{thm1}. Let $f$ be a prescribed vorticity law, independent of $\varepsilon$, and suppose that $f$ and its normalized radial profile $U$ satisfy Assumption~\ref{ass5-1} for some $\alpha\in(0,1)$. There exists $\varepsilon_0>0$ such that, for every $0<\varepsilon<\varepsilon_0$, there is a symmetric traveling wave $(\Omega_\varepsilon,\Psi_\varepsilon,\mu_\varepsilon)$ with velocity $-We_1$. The following properties hold.
		
		\begin{enumerate}
			\item[(i)] The relative stream function and vorticity satisfy
			\[
			-\Delta_p\Psi_\varepsilon=\omega_\varepsilon,\qquad \omega_\varepsilon=\varepsilon^{-2}f(\Psi_\varepsilon-\mu_\varepsilon)\ge0\quad\text{in }\Omega_\varepsilon.
			\]
			On the free surface,
			\[
			\Psi_\varepsilon=0,\qquad \frac12|\nabla_p\Psi_\varepsilon|^2+gp_2+\sigma\kappa_\varepsilon=\frac{W^2}{2}\quad\text{on }\partial\Omega_\varepsilon,
			\]
			where $\kappa_\varepsilon$ is the signed curvature with the convention fixed above. The level $\mu_\varepsilon$ is selected so that
			\[
			\int_{\Omega_\varepsilon}\omega_\varepsilon\,dp=\Gamma
			\]
			for every $0<\varepsilon<\varepsilon_0$.
			
			\item[(ii)] Define the active vortex region by
			\[
			D_\varepsilon:=\{p\in\Omega_\varepsilon:\Psi_\varepsilon(p)>\mu_\varepsilon\}=\{p\in\Omega_\varepsilon:\omega_\varepsilon(p)>0\}.
			\]
			There are centers $p_\varepsilon\to p_W$ and a constant $C>0$, independent of $\varepsilon$, such that, after extending $\omega_\varepsilon$ by zero outside $\Omega_\varepsilon$, we have
			\[
			\operatorname{supp}\omega_\varepsilon= \overline{D_\varepsilon}\subset B_{C\varepsilon}(p_\varepsilon).
			\]
			In particular, we have
			\begin{equation}
				\omega_\varepsilon\,dp\rightharpoonup\Gamma\delta_{p_W}\qquad\text{as }\varepsilon\to0^+
				\label{eq1-4}
			\end{equation}
			in the sense of Radon measures.
			
			\item[(iii)] On the inner scale, set
			\[
			X:=\frac{p-p_\varepsilon}{\varepsilon},\qquad \widehat D_\varepsilon:=\frac{D_\varepsilon-p_\varepsilon}{\varepsilon}.
			\]
			Then
			\begin{equation}
				\overline{\widehat D_\varepsilon}\longrightarrow\overline{B_1}
				\label{eq1-5}
			\end{equation}
			in the Hausdorff sense. More precisely, for all sufficiently small $\varepsilon$,
			\[
			\partial\widehat D_\varepsilon=\{(1+\rho_\varepsilon(\theta))(\cos\theta,\sin\theta):\theta\in\mathbb S^1\},
			\]
			where $\|\rho_\varepsilon\|_{C^{1,\alpha'}(\mathbb S^1)}\to0$ for every $0<\alpha'<\alpha$. Moreover,
			\begin{equation}
				\varepsilon^2\omega_\varepsilon(p_\varepsilon+\varepsilon X)\longrightarrow f(U(X))
				\label{eq1-6}
			\end{equation}
			locally uniformly in $X\in\mathbb R^2$.
			
			\item[(iv)] On the outer scale, after identifying the nearby domains through their conformal parametrizations,
			\begin{equation}
				\Psi_\varepsilon\longrightarrow\Gamma G_{\Omega_W}(\,\cdot\,,p_W)+\Psi^0_{\Omega_W,W}
				\label{eq1-7}
			\end{equation}
			locally in $C^1$ on $\Omega_W\setminus\{p_W\}$. In particular, the velocity fields converge locally away from $p_W$. Writing the free surfaces as $p_2=\eta_\varepsilon(p_1)$ and $p_2=\eta_W(p_1)$, we also have
			\[
			\eta_\varepsilon\longrightarrow\eta_W\qquad\text{in }H^3(\mathbb R).
			\]
		\end{enumerate}
	\end{theorem}

	Theorem 1.3 shows that a fixed coupled point-vortex wave can be realized as the limit of traveling waves carrying compactly supported vorticity with the regularity specified by Assumption~5.1. As stated in part~{\rm(i)}, these waves satisfy the interior vorticity equation and the Bernoulli condition on the unknown free surface simultaneously, with total circulation exactly $\Gamma$. Parts~{\rm(ii)} and~{\rm(iii)} describe the concentration at two scales. The core shrinks to $p_W$ in physical coordinates, while magnification by $\varepsilon^{-1}$ reveals an asymptotically circular boundary and the radial vorticity profile $f(U)$. Its radius is of order $\varepsilon$ and its vorticity is of order $\varepsilon^{-2}$. On the outer scale, part~{\rm(iv)} shows that the free surface converges to the fixed, nontrivial profile $\eta_W$ of Theorem~\ref{thm1}, while the velocity fields converge away from the vortex center.
	
	The joint Lyapunov--Schmidt reduction determines the core depth through a scalar equation obtained after solving the complementary equations. Unlike the construction of Ehrnstr\"om, Walsh, and Zeng \cite{EhrnstromWalshZeng2023}, where the free-surface deformation enters as an exponentially small perturbation of a nearly flat configuration, the geometry of our nonflat wave enters the inner elliptic problem through the regular drift at order one with respect to $\varepsilon$. As $\varepsilon\to0^+$, the reduced scalar equation converges to the regular-drift equation for the point-vortex wave constructed in Theorem~\ref{thm1}, with the finite-$\varepsilon$ equation differing from it by a small perturbation. The nondegeneracy of the point-vortex equilibrium then allows us to continue it to a smooth-core solution, determining the free surface and vortex-core boundary simultaneously. This establishes the regularization of the coupled point-vortex wave.

	\smallskip
	\noindent\textit{Periodic deep-vortex waves.}
	We finally turn to periodic waves. Fix a period $L>0$ and set $k:=2\pi/L$. One point vortex of circulation $\Gamma$ in each period cell gives an infinite vortex row in the covering plane. For a flat free surface, the reflected vortex row produces the equilibrium relation
	\[
	W=W_c\coth(ka),\qquad W_c:=\frac{\Gamma k}{4\pi}=\frac{\Gamma}{2L}.
	\]
	Unlike the solitary drift $\Gamma/(4\pi a)$, which vanishes as $a\to\infty$, the periodic drift approaches the strictly positive limit $W_c$. Thus the deep-vortex regime occurs near $W=W_c$, rather than near zero speed as in Theorem~\ref{thm1}. The conformal correction to the regular drift retains the same form, but the periodic geometry changes the limiting balance.
	
	As the vortex row moves deeper, its influence along the free surface becomes increasingly uniform, while the spatially varying part of the induced velocity decays exponentially with depth. The leading periodic variation has size proportional to $e^{-ka}$, where $k=2\pi/L$. This suggests introducing
	\[
	\rho:=e^{-ka}.
	\]
	The parameter $\rho$ thus measures the weakening interaction between the deep vortex row and the free surface. As $a\to\infty$, we have $\rho\to0$, so the infinite-depth limit is represented by the finite endpoint $\rho=0$. Higher Fourier modes decay as successive powers $\rho^n$, making $\rho$ a natural parameter for the local solution curve.
	
	At the limiting flat state, the linearized surface equation acts on the Fourier mode $\cos(nkx)$ through the coefficient
	\[
	\lambda_n:=g+\sigma n^2k^2-W_c^2nk,\qquad n\ge1.
	\]
	The resonance condition $\lambda_n=0$ is equivalent to
	\[
	W_c^2=\frac{g}{nk}+\sigma nk,
	\]
	which states that the limiting speed $W_c$ equals the linear gravity--capillary phase speed of the $n$th mode. We therefore impose the nonresonance condition $\lambda_n\ne0$ for all $n\ge1$, so that the surface linearization is invertible on the even mean-zero periodic space. The coefficients $\lambda_n$ need not be positive; only their nonvanishing is required. This condition allows us to solve for the free surface and wave speed simultaneously as functions of $\rho$.
	
	\begin{theorem}[Periodic deep-vortex waves]
		\label{thm4}
		Fix $g,\sigma,\Gamma,L>0$ and $s>3/2$, and assume that $\lambda_n\ne0$ for every $n\ge1$. Then the coupled periodic problem admits a unique smooth local solution curve near the limiting state $(\rho,Y,W)=(0,0,W_c)$, parametrized by
		\[
		[0,\rho_0)\ni\rho\longmapsto(Y_\rho,W_\rho)\in H^{s+2}_{\mathrm e,0}(\mathbb T_L)\times\mathbb R
		\]
		for some $\rho_0>0$. Here $\mathbb T_L:=\mathbb R/L\mathbb Z$, and $H^{s+2}_{\mathrm e,0}(\mathbb T_L)$ consists of even periodic profiles with zero mean in the conformal coordinate. For every $0<\rho<\rho_0$, this curve determines a symmetric periodic capillary--gravity traveling wave with velocity $-W_\rho e_1$, carrying one point vortex of circulation $\Gamma$ per period at conformal depth
		\[
		a_\rho=-\frac1k\log\rho.
		\]
		The following assertions hold.
		
		\begin{enumerate}
			
			\item[(i)] As $\rho\to0^+$, the conformal free-surface profile satisfies
			\[
			Y_\rho=-\frac{4W_c^2}{\lambda_1}\rho\cos(kx)+O_{H^{s+2}(\mathbb T_L)}(\rho^2).
			\]
			
			\item[(ii)] The wave speed admits the expansion
			\[
			W_\rho=W_c+c_L\rho^2+O(\rho^3),
			\]
			where
			\[
			c_L:=2W_c+\frac{4W_c^3k}{\lambda_1}=2W_c\frac{g+\sigma k^2+W_c^2k}{g+\sigma k^2-W_c^2k}\ne0.
			\]
			In particular, for sufficiently small $\rho>0$, the branch lies above $W_c$ if $\lambda_1>0$ and below $W_c$ if $\lambda_1<0$.
			
			\item[(iii)] The map $\rho\mapsto W_\rho$ is strictly monotone for sufficiently small $\rho>0$. Consequently, the conformal vortex depth can be expressed as a function of speed along the local branch, with
			\[
			a(W)=\frac1{2k}\log\frac{|c_L|}{|W-W_c|}+o(1)
			\]
			as $W\to W_c$ from the side specified in part~{\rm(ii)}.
			
		\end{enumerate}
	\end{theorem}

	Theorem~\ref{thm4} shows that the leading surface deformation of the periodic branch is governed by the first Fourier mode, in contrast to the trough--shoulder profile of the solitary problem. By part~{\rm(i)}, $Y_\rho(0)=-4W_c^2\rho/\lambda_1+O(\rho^2)$. Thus the surface has a depression at the symmetry axis when $\lambda_1>0$ and an elevation when $\lambda_1<0$, for sufficiently small $\rho>0$. In part~{\rm(ii)}, the quadratic speed correction has two contributions. The term $2W_c$ comes from the flat periodic drift,
	\[
	W_c\coth(ka)=W_c+2W_c\rho^2+O(\rho^4),
	\]
	while $4W_c^3k/\lambda_1$ arises from the feedback of the free surface. Although the surface deformation is of order $\rho$, its effect on the vortex position is only of order $\rho^2$ because the vortex lies deep below the surface. Both contributions therefore enter the drift at order $\rho^2$, and the sign of $\lambda_1$ determines whether the branch lies above or below $W_c$.
	
	Since $\rho=e^{-ka}$, the nonvanishing of $c_L$ yields the logarithmic depth divergence as $W\to W_c$ established in part~{\rm(iii)}, in contrast to the solitary-wave case in Theorem~\ref{thm1}, where the depth grows in inverse proportion to the wave speed as $W\to0^+$. Using a similar joint Lyapunov--Schmidt reduction framework, we expect that the periodic point-vortex waves constructed in Theorem~\ref{thm4} can also be regularized into smooth vortex-core waves under suitable nondegeneracy conditions.
	
	\smallskip
	
	\subsection{Notation and conventions}
	We use the following notation and conventions throughout the paper.
	
	\begin{itemize}
		\item Our Fourier transform and its inverse are defined by
		\[
		\widehat f(\zeta):=\int_{\mathbb R}e^{-\mathrm{i}x\zeta}f(x)\,dx,\qquad f(x)=\frac1{2\pi}\int_{\mathbb R}e^{\mathrm{i}x\zeta}\widehat f(\zeta)\,d\zeta.
		\]
		The Hilbert transform is defined by $\widehat{Hf}(\zeta)=\mathrm{i}\operatorname{sgn}(\zeta)\widehat f(\zeta)$, so that $|D|=-H\partial_x$. For $y<0$, the operator $e^{y|D|}$ gives the harmonic extension into the lower half-plane.
		
		\item We use the Bessel-potential Sobolev norm
		\[
		\|f\|_{H^s(\mathbb R)}^2:=\frac1{2\pi}\int_{\mathbb R}(1+|\zeta|^2)^s|\widehat f(\zeta)|^2\,d\zeta.
		\]
		The subscripts ${\rm e}$ and ${\rm o}$ denote even and odd subspaces, respectively, while the subscript $0$ denotes the zero-mean subspace in the periodic setting. The same Fourier convention is used in the quantitative estimates of Section~\ref{sec3}.
		
		\item For the physical relative stream function $\Psi$, we use the conventions
		\[
		\nabla_p^\perp\Psi:=(-\partial_{p_2}\Psi,\partial_{p_1}\Psi),\qquad v_{\rm rel}=-\nabla_p^\perp\Psi,\qquad \omega=\operatorname{curl}_p v_{\rm rel}=-\Delta_p\Psi.
		\]
		Positive circulation corresponds to counterclockwise rotation. For a wave traveling with velocity $-We_1$, the relative and laboratory-frame velocities are related by $v_{\rm rel}=v+We_1$.
		
		\item The Dirichlet Green function and Robin function are normalized by
		\[
		G_\Omega(p,q)=\frac1{2\pi}\log\frac1{|p-q|}-H_\Omega(p,q),\qquad R_\Omega(p):=H_\Omega(p,p).
		\]
		We denote the Kirchhoff--Routh function on the physical domain by $\mathscr K_{\Omega,W}$, and its restriction to the symmetry axis by $K_{\Omega,W}(a)$.
		
		\item We use $u$ for the rescaled free-surface profile in Section~\ref{sec3} and $U$ for the radial whole-plane vortex profile in Section~\ref{sec5}. The rescaled coordinate $X$ belongs to $\mathbb R$ in Section~\ref{sec3} and to $\mathbb R^2$ in Section~\ref{sec5}. The symbol $\kappa$ denotes signed curvature, with the convention specified above.
		
		\item We write $A\lesssim B$ if $A\le CB$ for a constant $C$ independent of the relevant small parameter; $\gtrsim$ denotes the reverse inequality, and $\asymp$ denotes two-sided bounds. Constants may depend on the fixed data unless uniformity is explicitly stated. In Sections~\ref{sec3} and~\ref{sec4}, the small-speed limit is $W\to0^+$ with circulation $\Gamma$ fixed. In Section~\ref{sec5}, the speed $W$, circulation $\Gamma$, and vorticity law $f$ are fixed as $\varepsilon\to0^+$. In Section~\ref{sec6}, the deep-vortex limit is $\rho\to0^+$, with $W_\rho\to W_c$ along the periodic branch.
	\end{itemize}
	
	\smallskip
	
	\subsection{Organization of the paper}
	Section~\ref{sec2} formulates the point-vortex problem in conformal coordinates, bringing the free-surface Bernoulli condition and the regular-drift balance into a common framework. Based on this formulation, Section~\ref{sec3} uses a joint fixed-point argument to construct the small-speed solitary waves with fixed circulation and derive their asymptotic expansions. We then turn to the Kirchhoff--Routh function in prescribed domains in Section~\ref{sec4}, where variational comparisons allow us to study the vortex-center problem independently of the Bernoulli equation. This fixed-domain analysis also provides the location estimates needed for the global continuation theory in Part~II. Section~\ref{sec5} returns to the coupled waves and regularizes the point vortex into a smooth vorticity profile through a joint Lyapunov--Schmidt reduction. Section~\ref{sec6} extends the analysis to periodic domains, where the periodic vortex array leads to a different limiting speed and a logarithmic depth law.
	
	\bigskip

	\section{Conformal reduction of the point-vortex water-wave problem}
	\label{sec2}
	
	A conformal map from the lower half-plane $\Pi_-:=\{z=x+\mathrm{i}y\in\mathbb C:y<0\}$ onto the fluid domain puts the Bernoulli and vortex-equilibrium equations on a fixed reference domain.
	
	\subsection{Conformal parametrization of the free surface}
	
	For a real-valued even function $Y$ on $\mathbb R$, define
	\[
	\Phi_Y(x+\mathrm{i}y):=-e^{y|D|}HY(x)+\mathrm{i}e^{y|D|}Y(x),\qquad y<0,
	\]
	and set $Z_Y(z):=z+\Phi_Y(z)$. The Fourier and Hilbert transform conventions are those of Section~\ref{sec1}.
	
	\begin{lemma}
		\label{lem2-1}
		Let $s>\frac12$ and $Y\in H_{\rm e}^{s+2}(\mathbb R)$. Then $\Phi_Y$ is holomorphic in $\Pi_-$ and satisfies $\Phi_Y(-\overline z)=-\overline{\Phi_Y(z)}$. On the boundary,
		\[
		Z_Y(x)=\Xi_Y(x)+\mathrm{i}Y(x),\qquad \Xi_Y(x):=x-HY(x).
		\]
		Consequently, $\Xi_Y'=1+|D|Y$, and the boundary Jacobian factor is $J(Y):=|Z_Y'|^2=\bigl(1+|D|Y\bigr)^2+(Y')^2$. The signed curvature of the free surface is
		\begin{equation}
			\kappa(Y)=-\frac{\bigl(1+|D|Y\bigr)Y''-Y'(|D|Y)'}{\left[\bigl(1+|D|Y\bigr)^2+(Y')^2\right]^{3/2}},
			\label{2-1}
		\end{equation}
		with $D\kappa(0)[f]=-f''$.
		
		If $1+|D|Y\ge c_0>0$ on $\mathbb R$, then $Z_Y$ maps $\Pi_-$ conformally onto a simply connected domain $\Omega_Y$. Moreover, $\Xi_Y$ is a strictly increasing diffeomorphism of $\mathbb R$ onto itself, and the free surface $Z_Y(\mathbb R)$ is a global graph in physical coordinates.
	\end{lemma}
	
	\begin{proof}
		The operator $e^{y|D|}$ gives harmonic extension to $\Pi_-$. Together with $|D|=-H\partial_x$ and $H|D|=\partial_x$, this verifies the Cauchy--Riemann equations for the real and imaginary parts of $\Phi_Y$, so $\Phi_Y$ is holomorphic. The parity of $Y$ and $HY$ gives the stated reflection symmetry, while the boundary trace yields $Z_Y(x)=x-HY(x)+\mathrm{i}Y(x)$. Differentiating gives the formulas for $\Xi_Y'$ and $J(Y)$.
		
		The free surface is parametrized by $\gamma(x)=(\Xi_Y(x),Y(x))$. Using $\Xi_Y'=1+|D|Y$ and $\Xi_Y''=(|D|Y)'$ in the usual curvature formula gives \eqref{2-1}, whose linearization at $Y=0$ is $D\kappa(0)[f]=-f''$.
		
		Suppose now that $1+|D|Y\ge c_0>0$. Positivity of the Poisson kernel implies
		\[
		\operatorname{Re}Z_Y'(x+\mathrm{i}y)=1+e^{y|D|}|D|Y(x)\ge c_0\qquad\text{in }\Pi_-.
		\]
		For distinct $z_1,z_2\in\Pi_-$, convexity of the half-plane gives
		\[
		\operatorname{Re}\frac{Z_Y(z_2)-Z_Y(z_1)}{z_2-z_1}
		=\operatorname{Re}\int_0^1Z_Y'\bigl(z_1+t(z_2-z_1)\bigr)\,dt\ge c_0.
		\]
		Hence $Z_Y$ is injective and therefore conformal onto its image. Since $\Xi_Y'\ge c_0$ and $HY(x)\to0$ as $|x|\to\infty$, the map $\Xi_Y$ is a diffeomorphism of $\mathbb R$ onto itself. This also gives the graph property.
	\end{proof}
	
	\smallskip
	
	\subsection{The Bernoulli equation in conformal coordinates}
	
	For $a>0$, consider a point vortex of circulation $\Gamma>0$ at $z_a=-\mathrm{i}a$ in the lower half-plane. Its Dirichlet stream function is
	\[
	\psi_a(x,y):=\frac{\Gamma}{4\pi}\log\frac{x^2+(y-a)^2}{x^2+(y+a)^2},
	\]
	which satisfies $-\Delta\psi_a=\Gamma\delta_{(0,-a)}$ in $\Pi_-$ and vanishes on the real axis. In a frame moving with velocity $-We_1$, the relative stream function becomes
	\[
	\Psi_{a,W}(x,y):=\psi_a(x,y)+Wy.
	\]
	Since $\Psi_{a,W}(x,0)=0$, the real axis is a streamline, and the boundary velocity is determined by
	\[
	q_{a,W}(x):=\partial_y\Psi_{a,W}(x,0)=W-\frac{\Gamma a}{\pi(x^2+a^2)}.
	\]
	
	Let $Y$ satisfy the assumptions of Lemma~\ref{lem2-1}, with $Z_Y$ univalent in $\Pi_-$. We define the relative stream function in the physical domain $\Omega_Y$ by
	\[
	\widetilde\Psi_{a,W}:=\Psi_{a,W}\circ Z_Y^{-1}.
	\]
	
	\begin{lemma}
		\label{lem2-2}
		The relative velocity $v_{\rm rel}:=-\nabla^\perp\widetilde\Psi_{a,W}$ satisfies the kinematic boundary condition on $Z_Y(\mathbb R)$. The dynamic boundary condition is equivalent to
		\begin{equation}
			\mathcal B(Y,a,W):=gY+\sigma\kappa(Y)+\frac{q_{a,W}^2}{2J(Y)}-\frac{W^2}{2}=0.
			\label{2-2}
		\end{equation}
	\end{lemma}
	
	\begin{proof}
		Since $\widetilde\Psi_{a,W}=0$ on $Z_Y(\mathbb R)$, the relative velocity is tangent to the free surface, and the kinematic condition follows. By conformality,
		\[
		\left|\nabla\widetilde\Psi_{a,W}(Z_Y(x))\right|^2=\frac{|\nabla\Psi_{a,W}(x,0)|^2}{|Z_Y'(x)|^2}=\frac{q_{a,W}(x)^2}{J(Y)(x)}.
		\]
		The Bernoulli constant on the free surface is $W^2/2$, as follows from the decay of the surface deformation and the laboratory-frame velocity at infinity. Substituting the boundary velocity and the physical surface height $Y(x)$ into the Bernoulli condition gives \eqref{2-2}.
	\end{proof}
	
	At the flat equilibrium depth $a_0(W)=\Gamma/(4\pi W)$, we have
	\[
	q_{a_0,W}(x)=W\frac{(x/a_0)^2-3}{1+(x/a_0)^2},\qquad \mathcal B(0,a_0,W)=4W^2\frac{1-(x/a_0)^2}{\bigl(1+(x/a_0)^2\bigr)^2}.
	\]
	Even at the equilibrium vortex depth, the flat surface does not satisfy Bernoulli's condition. The residual changes sign at $|x|=a_0$, indicating the central depression and two elevated shoulders of the small-speed surface profile.
	
	We turn to the vortex equilibrium condition, the second equation of the coupled system.
	
	\smallskip
	
	\subsection{The point-vortex equilibrium condition}
	For the vortex equilibrium calculation, write $z_a=-\mathrm{i}a$ and $p_a:=Z_Y(z_a)$. For a velocity field $v=(v_1,v_2)$, we use the complex notation $\mathcal V:=v_1-\mathrm{i}v_2$.
	
	\begin{lemma}
		\label{lem2-3}
		Let $Y$ satisfy the assumptions of Lemma~\ref{lem2-1}, and suppose that $Z_Y:\Pi_-\to\Omega_Y$ is univalent. The regular relative complex velocity of the physical point vortex at $p_a=Z_Y(-\mathrm{i}a)$ is
		\[
		\mathcal V_{\rm reg}(p_a)=\frac{\mathcal P(Y,a,W)}{Z_Y'(-\mathrm{i}a)},
		\]
		where
		\[
		\mathcal P(Y,a,W):=W-\frac{\Gamma}{4\pi a}+\frac{\mathrm{i}\Gamma}{4\pi}\frac{Z_Y''(-\mathrm{i}a)}{Z_Y'(-\mathrm{i}a)}.
		\]
		The quantity $\mathcal P(Y,a,W)$ is real and admits the representation
		\begin{equation}
			\mathcal P(Y,a,W)=W-\frac{\Gamma}{4\pi a}-\frac{\Gamma}{4\pi}\frac{\bigl(e^{-a|D|}Y''\bigr)(0)}{1+\bigl(e^{-a|D|}|D|Y\bigr)(0)}.
			\label{2-3}
		\end{equation}
		Consequently, the point vortex is stationary in the moving frame if and only if $\mathcal P(Y,a,W)=0$.
	\end{lemma}
	
	\begin{proof}
		The relative complex velocity in the reference half-plane is
		\[
		\mathcal V_{a,W}(z)=W+\frac{\mathrm{i}\Gamma}{2\pi}\left(\frac{1}{z-\mathrm{i}a}-\frac{1}{z+\mathrm{i}a}\right).
		\]
		Writing $\zeta=z-z_a$, we have
		\[
		\mathcal V_{a,W}(z)=-\frac{\mathrm{i}\Gamma}{2\pi\zeta}+W-\frac{\Gamma}{4\pi a}+O(\zeta).
		\]
		Set $\alpha:=Z_Y'(z_a)$ and $\gamma:=Z_Y''(z_a)$. Near the vortex,
		\[
		Z_Y(z)-p_a=\alpha\zeta+\frac12\gamma\zeta^2+O(\zeta^3).
		\]
		The physical relative complex velocity is $\widetilde{\mathcal V}_{a,W}(Z_Y(z))=\mathcal V_{a,W}(z)/Z_Y'(z)$. Subtracting the singular self-induced velocity in physical coordinates gives
		\[
		\mathcal V_{\rm reg}(p_a)=\lim_{z\to z_a}\left[\frac{\mathcal V_{a,W}(z)}{Z_Y'(z)}+\frac{\mathrm{i}\Gamma}{2\pi}\frac1{Z_Y(z)-p_a}\right].
		\]
		Using
		\[
		\frac1{Z_Y'(z)}=\frac1\alpha-\frac{\gamma}{\alpha^2}\zeta+O(\zeta^2),\qquad \frac1{Z_Y(z)-p_a}=\frac1{\alpha\zeta}-\frac{\gamma}{2\alpha^2}+O(\zeta),
		\]
		the singular terms cancel, and we obtain
		\[
		\mathcal V_{\rm reg}(p_a)=\frac1\alpha\left(W-\frac{\Gamma}{4\pi a}+\frac{\mathrm{i}\Gamma}{4\pi}\frac{\gamma}{\alpha}\right).
		\]
		
		By reflection symmetry, $Z_Y'(z_a)$ is real and $Z_Y''(z_a)$ is purely imaginary, so $\mathcal P(Y,a,W)$ is real. Moreover, the holomorphic extension of $Z_Y'$ gives
		\[
		Z_Y'(-\mathrm{i}a)=1+\bigl(e^{-a|D|}|D|Y\bigr)(0),\qquad Z_Y''(-\mathrm{i}a)=\mathrm{i}\bigl(e^{-a|D|}Y''\bigr)(0).
		\]
		Substitution gives \eqref{2-3}. Since $Z_Y'(z_a)\ne0$, zero regular velocity is equivalent to $\mathcal P=0$.
	\end{proof}
	
	The conformal correction in $\mathcal P(Y,a,W)$ arises because changing coordinates and removing the singular self-induced velocity do not commute. The Robin-function transformation in Section~\ref{sec4} gives the same geometric contribution.
	
	Combining Lemmas~\ref{lem2-2} and \ref{lem2-3}, we obtain the coupled point-vortex system
	\begin{equation}
		\begin{cases}
			\mathcal B(Y,a,W)=0,\\
			\mathcal P(Y,a,W)=0.
		\end{cases}
		\label{2-4}
	\end{equation}
	The speed $W>0$ is prescribed, while $Y$ and the conformal vortex depth $a$ are unknown. The first equation imposes the Bernoulli condition, and the second determines the vortex position through the regular drift. At the flat reference configuration $(Y,a)=(0,a_0(W))$, the vortex equilibrium condition holds but the Bernoulli condition does not, as shown in Section~\ref{sec2}. We construct a nearby solution of both equations in the small-speed regime.
	
	\bigskip
	
	\section{Small-speed construction of point-vortex waves}
	\label{sec3}
	We solve \eqref{2-4} with $g,\sigma,\Gamma>0$ fixed and $W\to0^+$. The flat equilibrium depth sets the horizontal scale, and a joint fixed-point argument yields the coupled solutions.
	
	\subsection{Small-speed scaling and the rescaled system}
	Recall that $a_0(W)=\Gamma/(4\pi W)$. At this depth, the boundary velocity and the Bernoulli residual satisfy
	\[
	q_{a_0,W}(a_0X)=W\left(1-\frac4{1+X^2}\right),\qquad \mathcal B(0,a_0,W)(a_0X)=4W^2\frac{1-X^2}{(1+X^2)^2}.
	\]
	The gravity balance sets the vertical scale $\rho_W=4W^2/g$. We therefore set
	\begin{equation}
		a=a_0(1+\lambda),\qquad x=a_0X,\qquad Y(x)=\rho_Wu(X),\qquad a_0=\frac{\Gamma}{4\pi W},\qquad \rho_W=\frac{4W^2}{g}.
		\label{3-1}
	\end{equation}
	Here $\lambda$ measures the relative correction to the vortex depth, while $u$ describes the rescaled free surface.
	
	The rescaled conformal and capillary parameters are
	\[
	\delta_W:=\frac{\rho_W}{a_0}=\frac{16\pi W^3}{g\Gamma},\qquad \beta_W:=\frac{\sigma}{ga_0^2}=\frac{16\pi^2\sigma W^2}{g\Gamma^2}.
	\]
	In particular, $\delta_W=O(W^3)$ and $\beta_W=O(W^2)$. We also introduce
	\[
	\mu_W:=\frac{\delta_W}{\sqrt{\beta_W}}=\frac{4W^2}{\sqrt{g\sigma}},
	\]
	which will appear in the uniform estimates below.
	
	Unless otherwise stated, derivatives and Fourier multipliers act on the variable $X$. Define
	\[
	r_\lambda(X):=1-\frac{4(1+\lambda)}{X^2+(1+\lambda)^2},\qquad J_\delta(u):=(1+\delta|D|u)^2+\delta^2(u')^2,
	\]
	where $|D|=(-\partial_X^2)^{1/2}$. Then $q_{a,W}(a_0X)=Wr_\lambda(X)$, and $J_\delta(u)$ is the rescaled conformal factor.
	
	\begin{lemma}
		\label{lem3-1}
		Let $W>0$, and let $(Y,a)$ and $(u,\lambda)$ be related by \eqref{3-1}. Then
		\[
		\mathcal B(Y,a,W)=0,\qquad \mathcal P(Y,a,W)=0
		\]
		if and only if
		\begin{equation}
			\mathscr F_W(u,\lambda):=
			\begin{pmatrix}
				\mathscr B_{\delta_W,\beta_W}(u,\lambda)\\[1mm]
				\mathscr P_{\delta_W}(u,\lambda)
			\end{pmatrix}
			=0,
			\label{3-2}
		\end{equation}
		where
		\[
		\mathscr B_{\delta,\beta}(u,\lambda)=u-\beta\frac{(1+\delta|D|u)u''-\delta u'(|D|u)'}{J_\delta(u)^{3/2}}+\frac18\frac{r_\lambda^2}{J_\delta(u)}-\frac18
		\]
		and
		\[
		\mathscr P_\delta(u,\lambda)=1-\frac1{1+\lambda}-\delta\frac{\bigl(e^{-(1+\lambda)|D|}u''\bigr)(0)}{1+\delta\bigl(e^{-(1+\lambda)|D|}|D|u\bigr)(0)}.
		\]
		At $(\delta,\beta)=(0,0)$, the limiting system has the solution
		\[
		(u_*,\lambda_*)=(-Q,0),\qquad Q(X):=\frac{1-X^2}{(1+X^2)^2}.
		\]
	\end{lemma}
	
	\begin{proof}
		Under the rescaling \eqref{3-1}, we have
		\[
		|D_x|Y=\delta|D|u,\qquad Y_x=\delta u',\qquad Y_{xx}=\frac{\rho_W}{a_0^2}u'',
		\]
		so that $J(Y)=J_\delta(u)$. The curvature formula in Lemma~\ref{lem2-1} becomes
		\[
		\kappa(Y)=-\frac{\rho_W}{a_0^2}\frac{(1+\delta|D|u)u''-\delta u'(|D|u)'}{J_\delta(u)^{3/2}}.
		\]
		Since $g\rho_W=4W^2$, $\beta=\sigma/(ga_0^2)$, and $q_{a,W}(a_0X)=Wr_\lambda(X)$, dividing the Bernoulli equation by $4W^2$ gives $\mathscr B_{\delta,\beta}(u,\lambda)=0$.
		
		For the vortex equilibrium equation, we use
		\[
		\frac{\Gamma}{4\pi a}=\frac{W}{1+\lambda},\qquad e^{-a|D_x|}=e^{-(1+\lambda)|D|}.
		\]
		The conformal correction satisfies
		\[
		\bigl(e^{-a|D_x|}Y_{xx}\bigr)(0)=\frac{\rho_W}{a_0^2}\bigl(e^{-(1+\lambda)|D|}u''\bigr)(0),\qquad \bigl(e^{-a|D_x|}|D_x|Y\bigr)(0)=\delta\bigl(e^{-(1+\lambda)|D|}|D|u\bigr)(0).
		\]
		Using $\Gamma\rho_W/(4\pi Wa_0^2)=\delta$, division by $W$ yields $\mathscr P_\delta(u,\lambda)=0$.
		
		Finally, at $\delta=\beta=0$, the system reduces to
		\[
		u+\frac18(r_\lambda^2-1)=0,\qquad 1-\frac1{1+\lambda}=0.
		\]
		Thus $\lambda=0$, and since
		\[
		\frac18(r_0^2-1)=\frac{1-X^2}{(1+X^2)^2}=Q(X),
		\]
		we obtain $u=-Q$.
	\end{proof}
	
	We retain the capillary term in the reference solution rather than treating it as part of the residual.
	
	\smallskip
	
	\subsection{The reference profile}
	
	To define the reference state, set $\delta=0$ but keep the capillary operator $L_\beta:=I-\beta\partial_X^2$. For the auxiliary problem, the point-vortex equation gives $\lambda=0$, and the Bernoulli equation reduces to
	\[
	L_\beta u=-Q.
	\]
	We therefore define
	\[
	u_\beta:=-L_\beta^{-1}Q,\qquad U_\beta:=(u_\beta,0).
	\]
	This profile solves the $\delta=0$ system exactly for each $\beta\ge0$, unlike $-Q$; restoring $\delta=\delta_W$ then produces only the geometric residual estimated in Lemma~\ref{lem3-7}.
	
	\begin{proposition}
		\label{prop3-2}
		For every $\beta\ge0$, the function $u_\beta$ is even and
		\begin{equation}
			\widehat{u_\beta}(\xi)=-\pi\frac{|\xi|e^{-|\xi|}}{1+\beta\xi^2}.
			\label{3-3}
		\end{equation}
		In particular,
		\[
		u_\beta(0)=-\int_0^\infty\frac{\xi e^{-\xi}}{1+\beta\xi^2}\,d\xi<0.
		\]
		Moreover, for every fixed $m\ge0$,
		\[
		u_\beta\longrightarrow -Q\qquad\text{in }H^m(\mathbb R)
		\]
		as $\beta\to0$.
	\end{proposition}
	
	\begin{proof}
		The Fourier transform of $Q$ is $\widehat Q(\xi)=\pi|\xi|e^{-|\xi|}$. Taking the Fourier transform of $L_\beta u_\beta=-Q$ gives \eqref{3-3}; evenness follows immediately. Fourier inversion at $X=0$ yields the stated integral, which is strictly negative. Finally, dominated convergence, using the exponential decay of $\widehat Q$, gives $u_\beta\to-Q$ in every fixed Sobolev space $H^m(\mathbb R)$.
	\end{proof}
	
	The exact reference state $U_\beta$ tends to $(-Q,0)$ as $\beta\to0$ and will serve as the base point for the uniform linear estimates.
	
	\smallskip
	
	\subsection{Uniform linear theory}
	
	The linearized system must be inverted uniformly as $\beta\to0$. Because its highest-order term has coefficient $\beta$, we use Sobolev norms adapted to this limit.
	
	For $s>\frac12$ and $0<\beta\le1$, define
	\begin{equation}
		\|u\|_{X_\beta^s}:=\|u\|_{H^s}+\sqrt{\beta}\,\|u\|_{H^{s+1}}+\beta\,\|u\|_{H^{s+2}}.
		\label{3-4}
	\end{equation}
	We denote by $X_{\beta,\mathrm e}^s$ the even subspace and set $\mathcal X_\beta^s:=X_{\beta,\mathrm e}^s\times\mathbb R$ and $\mathcal Y^s:=H_{\mathrm e}^s\times\mathbb R$, with their natural product norms. Recall that $L_\beta=I-\beta\partial_X^2$.
	
	\begin{lemma}
		\label{lem3-3}
		Let $s>\frac12$ and $0<\beta\le1$. Then $L_\beta:X_\beta^s\to H^s$ is an isomorphism, and
		\[
		\|L_\beta^{-1}f\|_{X_\beta^s}\le C\|f\|_{H^s},
		\]
		where $C$ is independent of $\beta$.
	\end{lemma}
	
	\begin{proof}
		The inverse is the Fourier multiplier $(1+\beta\xi^2)^{-1}$. Since $\beta(1+\xi^2)\le1+\beta\xi^2$ for $0<\beta\le1$, we have
		\[
		\frac1{1+\beta\xi^2}\le1,\qquad \frac{\sqrt{\beta}(1+\xi^2)^{1/2}}{1+\beta\xi^2}\le1,\qquad \frac{\beta(1+\xi^2)}{1+\beta\xi^2}\le1.
		\]
		Plancherel's theorem therefore gives
		\[
		\|L_\beta^{-1}f\|_{H^s}+\sqrt{\beta}\|L_\beta^{-1}f\|_{H^{s+1}}+\beta\|L_\beta^{-1}f\|_{H^{s+2}}\le3\|f\|_{H^s}.
		\]
		Conversely, $\|L_\beta u\|_{H^s}\le\|u\|_{H^s}+\beta\|u\|_{H^{s+2}}$, proving the assertion.
	\end{proof}
	
	At the reference state $U_\beta=(u_\beta,0)$, the linearization of the auxiliary system is triangular:
	\[
	D_u\mathscr B_{0,\beta}(u_\beta,0)=L_\beta,\qquad D_\lambda\mathscr B_{0,\beta}(u_\beta,0)=r_0Q,\qquad D_u\mathscr P_0(u_\beta,0)=0,\qquad D_\lambda\mathscr P_0(u_\beta,0)=1.
	\]
	Hence
	\[
	\mathcal L_\beta:=D_U\mathscr F_{0,\beta}(U_\beta)=\begin{pmatrix}L_\beta&r_0Q\\0&1\end{pmatrix}.
	\]
	
	\begin{lemma}
		\label{lem3-4}
		Let $s>\frac12$ and $0<\beta\le1$. Then $\mathcal L_\beta:\mathcal X_\beta^s\to\mathcal Y^s$ is an isomorphism, with
		\[
		\mathcal L_\beta^{-1}\binom{f}{c}=\binom{L_\beta^{-1}(f-cr_0Q)}{c}.
		\]
		Moreover,
		\[
		\|\mathcal L_\beta^{-1}(f,c)\|_{\mathcal X_\beta^s}\le C_s\bigl(\|f\|_{H^s}+|c|\bigr),
		\]
		where $C_s$ is independent of $\beta$.
	\end{lemma}
	
	\begin{proof}
		The inverse formula follows from the triangular structure. Since $r_0Q$ belongs to $H^s(\mathbb R)$, Lemma~\ref{lem3-3} gives the uniform bound.
	\end{proof}
	
	We compare the full linearization with $\mathcal L_\beta$. The conformal terms involve one derivative with coefficient $\delta$, so the relevant ratio is
	\[
	\mu:=\frac{\delta}{\sqrt{\beta}}.
	\]
	For the physical parameters, $\mu_W=\delta_W/\sqrt{\beta_W}=4W^2/\sqrt{g\sigma}\to0$ as $W\to0^+$.
	
	\begin{lemma}
		\label{lem3-5}
		Let $s>\frac12$. There exist $C_s>0$ and $W_0>0$ such that, for $0<W\le W_0$,
		\begin{equation}
			\left\|D_U\mathscr F_W(U_{\beta_W})-\mathcal L_{\beta_W}\right\|_{\mathcal L(\mathcal X_{\beta_W}^s,\mathcal Y^s)}\le C_s\mu_W.
			\label{3-5}
		\end{equation}
	\end{lemma}
	
	\begin{proof}
		Write $\beta=\beta_W$, $\delta=\delta_W$, and $\mu=\delta/\sqrt{\beta}$. We take $W_0$ small enough that $0<\beta\le1$ and $\mu\le1$. By Proposition~\ref{prop3-2},
		\[
		\|u_\beta\|_{H^m}\le\|Q\|_{H^m}\qquad\text{for every fixed }m\ge0.
		\]
		We use the algebra and smooth composition estimates in $H^s(\mathbb R)$, valid for $s>\frac12$, with constants independent of $W$.
		
		At $u_\beta$, set
		\[
		A:=1+\delta|D|u_\beta,\qquad B:=\delta u_\beta',\qquad J:=A^2+B^2,\qquad N:=Au_\beta''-B(|D|u_\beta)'.
		\]
		The uniform Sobolev bounds imply
		\[
		\|J-1\|_{H^s}\le C_s\delta,\qquad \|N\|_{H^s}\le C_s.
		\]
		For sufficiently small $W$, we have $J\ge\frac12$ and
		\[
		\|J^{-p}-1\|_{H^s}\le C_{s,p}\delta
		\]
		for every fixed $p>0$.
		
		For $h\in X_{\beta,\mathrm e}^s$, differentiation gives
		\[
		\dot J[h]=2\delta A|D|h+2\delta Bh',
		\]
		and
		\[
		\dot N[h]=Ah''-B(|D|h)'+\delta(|D|h)u_\beta''-\delta h'(|D|u_\beta)'.
		\]
		Consequently,
		\[
		\|\dot J[h]\|_{H^s}\le C_s\delta\|h\|_{H^{s+1}},\qquad \|\dot N[h]-h''\|_{H^s}\le C_s\delta\|h\|_{H^{s+2}}.
		\]
		The difference between the linearized curvature term and $-\beta h''$ is
		\[
		-\beta\left((J^{-3/2}-1)h''+J^{-3/2}(\dot N[h]-h'')-\frac32NJ^{-5/2}\dot J[h]\right),
		\]
		whose $H^s$ norm is bounded by
		\[
		C_s\beta\delta\left(\|h\|_{H^{s+2}}+\|h\|_{H^{s+1}}\right)\le C_s\delta\|h\|_{X_\beta^s}.
		\]
		
		For the kinetic term, using $r_0^2=1+8Q$, we obtain
		\[
		\left\|\frac18r_0^2J^{-2}\dot J[h]\right\|_{H^s}\le C_s\delta\|h\|_{H^{s+1}}\le C_s\mu\|h\|_{X_\beta^s}.
		\]
		Thus
		\[
		\|D_u\mathscr B_{\delta,\beta}(u_\beta,0)[h]-L_\beta h\|_{H^s}\le C_s\mu\|h\|_{X_\beta^s}.
		\]
		Similarly, since $\partial_\lambda r_\lambda|_{\lambda=0}=4Q$,
		\[
		D_\lambda\mathscr B_{\delta,\beta}(u_\beta,0)=r_0QJ^{-1},\qquad \|D_\lambda\mathscr B_{\delta,\beta}(u_\beta,0)-r_0Q\|_{H^s}\le C_s\delta.
		\]
		
		For the point-vortex component, the smoothing estimate
		\[
		\left|\bigl(e^{-t|D|}|D|^jh\bigr)(0)\right|\le C_{s,j}\|h\|_{H^s},\qquad \frac12\le t\le\frac32,
		\]
		holds for every fixed integer $j\ge0$, together with its derivatives in $t$. Define
		\[
		T_\lambda h:=\bigl(e^{-(1+\lambda)|D|}h''\bigr)(0),\qquad S_\lambda h:=\bigl(e^{-(1+\lambda)|D|}|D|h\bigr)(0),\qquad d:=1+\delta S_0u_\beta.
		\]
		Since $d=1+O(\delta)$, we have
		\[
		D_u\mathscr P_\delta(u_\beta,0)[h]=-\frac{\delta T_0h}{d}+\frac{\delta^2(T_0u_\beta)(S_0h)}{d^2},
		\]
		and
		\[
		D_\lambda\mathscr P_\delta(u_\beta,0)=1-\delta\left.\partial_\lambda\left(\frac{T_\lambda u_\beta}{1+\delta S_\lambda u_\beta}\right)\right|_{\lambda=0}.
		\]
		The uniform evaluation bounds give
		\[
		|D_u\mathscr P_\delta(u_\beta,0)[h]|\le C_s\delta\|h\|_{H^s},\qquad |D_\lambda\mathscr P_\delta(u_\beta,0)-1|\le C_s\delta.
		\]
		Combining the Bernoulli and point-vortex estimates, and using $\delta\le\mu$ for $\beta\le1$, proves \eqref{3-5}.
	\end{proof}
	
	The previous lemmas show that the full linearization is a small perturbation of the uniformly invertible operator $\mathcal L_{\beta_W}$.
	
	\begin{proposition}
		\label{prop3-6}
		Let $s>\frac12$. There exist $W_0>0$ and $C_s>0$ such that, for every $0<W\le W_0$,
		\[
		D_U\mathscr F_W(U_{\beta_W}):\mathcal X_{\beta_W}^s\longrightarrow\mathcal Y^s
		\]
		is an isomorphism and
		\begin{equation}
			\left\|D_U\mathscr F_W(U_{\beta_W})^{-1}\right\|_{\mathcal L(\mathcal Y^s,\mathcal X_{\beta_W}^s)}\le C_s.
			\label{3-6}
		\end{equation}
	\end{proposition}
	
	\begin{proof}
		Writing $\beta=\beta_W$, we have
		\[
		D_U\mathscr F_W(U_\beta)=\mathcal L_\beta+\mathcal R_W=\mathcal L_\beta\bigl(I+\mathcal L_\beta^{-1}\mathcal R_W\bigr).
		\]
		By Lemmas~\ref{lem3-4} and \ref{lem3-5},
		\[
		\|\mathcal L_\beta^{-1}\mathcal R_W\|_{\mathcal L(\mathcal X_\beta^s)}\le C_s\mu_W.
		\]
		Since $\mu_W\to0$, we choose $W_0$ so that this bound is at most $\frac12$. The Neumann series then gives
		\[
		D_U\mathscr F_W(U_\beta)^{-1}=\bigl(I+\mathcal L_\beta^{-1}\mathcal R_W\bigr)^{-1}\mathcal L_\beta^{-1},
		\]
		and the uniform estimate \eqref{3-6} follows.
	\end{proof}
	
	\smallskip
	
	\subsection{Local solvability of the rescaled system}
	We solve the nonlinear system near the reference state $U_{\beta_W}=(u_{\beta_W},0)$. Proposition~\ref{prop3-6} provides a uniform inverse for the linearization, reducing the argument to the residual and the variation of the linearization.
	
	\begin{lemma}
		\label{lem3-7}
		Let $s>\frac12$. There exist $C_s>0$ and $W_0>0$ such that, for every $0<W\le W_0$,
		\begin{equation}
			\|\mathscr F_W(U_{\beta_W})\|_{\mathcal Y^s}\le C_s\delta_W.
			\label{3-7}
		\end{equation}
	\end{lemma}
	
	\begin{proof}
		Write $\beta=\beta_W$, $\delta=\delta_W$, and $u_\beta=u_{\beta_W}$. By Proposition~\ref{prop3-2}, $\|u_\beta\|_{H^m}\le C_m$ uniformly for $0<\beta\le1$ and every fixed $m\ge0$. Setting $A_\beta:=|D|u_\beta$, $B_\beta:=u_\beta'$, and $J_{\delta,\beta}:=(1+\delta A_\beta)^2+\delta^2B_\beta^2$, the algebra and composition estimates in $H^s(\mathbb R)$ give
		\[
		\|J_{\delta,\beta}-1\|_{H^s}+\|J_{\delta,\beta}^{-1}-1\|_{H^s}+\|J_{\delta,\beta}^{-3/2}-1\|_{H^s}\le C_s\delta.
		\]
		Since $u_\beta-\beta u_\beta''+(r_0^2-1)/8=0$, the Bernoulli residual is
		\[
		\mathscr B_{\delta,\beta}(u_\beta,0)=-\beta\left[\frac{(1+\delta A_\beta)u_\beta''-\delta B_\beta A_\beta'}{J_{\delta,\beta}^{3/2}}-u_\beta''\right]+\frac18r_0^2(J_{\delta,\beta}^{-1}-1).
		\]
		The uniform Sobolev bounds therefore imply $\|\mathscr B_{\delta,\beta}(u_\beta,0)\|_{H^s}\le C_s\delta$. For the point-vortex component,
		\[
		\mathscr P_\delta(u_\beta,0)=-\delta\frac{(e^{-|D|}u_\beta'')(0)}{1+\delta(e^{-|D|}|D|u_\beta)(0)}.
		\]
		The exponential smoothing and the uniform bounds on $u_\beta$ give $|\mathscr P_\delta(u_\beta,0)|\le C_s\delta$, proving \eqref{3-7}.
	\end{proof}
	
	The residual suggests a correction of order $\delta_W$. For $R>0$, define the closed ball
	\[
	\mathbb B_W(R):=\left\{U\in\mathcal X_{\beta_W}^s:\|U-U_{\beta_W}\|_{\mathcal X_{\beta_W}^s}\le R\delta_W\right\}.
	\]
	The linearization varies little within this ball.
	
	\begin{lemma}
		\label{lem3-8}
		Let $s>\frac12$ and $R>0$. There exist $C_{s,R}>0$ and $W_0>0$ such that, for every $0<W\le W_0$ and $U\in\mathbb B_W(R)$,
		\begin{equation}
			\left\|D_U\mathscr F_W(U)-D_U\mathscr F_W(U_{\beta_W})\right\|_{\mathcal L(\mathcal X_{\beta_W}^s,\mathcal Y^s)}\le C_{s,R}\delta_W.
			\label{3-8}
		\end{equation}
		Moreover, $J_{\delta_W}(u)\ge\frac12$ for every $U=(u,\lambda)\in\mathbb B_W(R)$.
	\end{lemma}
	
	\begin{proof}
		Write $\beta=\beta_W$, $\delta=\delta_W$, and $\mu=\delta/\sqrt{\beta}$. We take $W_0$ small enough that $0<\beta\le1$, $\mu\le1$, and $R\delta\le1/2$. For $U=(u,\lambda)=U_\beta+(\varphi,\lambda)\in\mathbb B_W(R)$, the weighted norm and the uniform bounds on $u_\beta$ give
		\[
		\delta\bigl(\||D|u\|_{H^s}+\|u'\|_{H^s}\bigr)\le C_{s,R}\delta.
		\]
		Hence $J_\delta(u)=1+O_{L^\infty}(\delta)$ uniformly on the ball, and in particular $J_\delta(u)\ge\frac12$ for sufficiently small $W$.
		
		We estimate the second derivative of $\mathscr F_W$. Introduce
		\[
		v(u):=(\delta|D|u,\delta u'),\qquad m(u):=(\beta u'',\beta(|D|u)').
		\]
		The weighted Sobolev bounds give
		\[
		\|v(u)\|_{H^s\times H^s}\le C_{s,R}\delta,\qquad \|m(u)\|_{H^s\times H^s}\le C_{s,R},
		\]
		while their derivatives satisfy
		\[
		\|Dv(u)[h]\|_{H^s\times H^s}\le C_s\mu\|h\|_{X_\beta^s},\qquad \|Dm(u)[h]\|_{H^s\times H^s}\le C_s\|h\|_{X_\beta^s}.
		\]
		Writing the capillary term as
		\[
		\mathscr C(u)=-c_1(v(u))m_1(u)-c_2(v(u))m_2(u),
		\]
		where
		\[
		c_1(v)=\frac{1+v_1}{((1+v_1)^2+v_2^2)^{3/2}},\qquad c_2(v)=-\frac{v_2}{((1+v_1)^2+v_2^2)^{3/2}},
		\]
		the product and composition estimates yield
		\[
		\|D^2\mathscr C(u)[h,k]\|_{H^s}\le C_{s,R}(\mu+\mu^2)\|h\|_{X_\beta^s}\|k\|_{X_\beta^s}.
		\]
		For the kinetic term $\mathscr K(u,\lambda)=r_\lambda^2/(8J_\delta(u))-1/8$, the same estimates apply since $r_\lambda^2-1$ and its first two $\lambda$ derivatives are uniformly bounded in $H^s$ for $|\lambda|\le1/2$. Thus the second derivative of the Bernoulli component is uniformly bounded on $\mathbb B_W(R)$.
		
		For the point-vortex component, write
		\[
		\mathscr P_\delta(u,\lambda)=1-\frac1{1+\lambda}-\delta\frac{A(u,\lambda)}{1+\delta C(u,\lambda)},
		\]
		where
		\[
		A(u,\lambda):=(e^{-(1+\lambda)|D|}u'')(0),\qquad C(u,\lambda):=(e^{-(1+\lambda)|D|}|D|u)(0).
		\]
		Since $|\lambda|\le R\delta\le1/2$, the Fourier multipliers involved in $A$ and $C$, together with their first two $\lambda$ derivatives, are uniformly smoothing. In particular, $A$, $C$, and their first two derivatives are bounded on the ball, while $1+\delta C(u,\lambda)\ge1/2$. Differentiating the quotient therefore gives a uniform bound for $D^2\mathscr P_\delta$.
		
		Combining these estimates, we obtain
		\[
		\|D_U^2\mathscr F_W(U)[H,K]\|_{\mathcal Y^s}\le C_{s,R}\|H\|_{\mathcal X_\beta^s}\|K\|_{\mathcal X_\beta^s}.
		\]
		Integrating along the segment joining $U_\beta$ to $U$ and using $\|U-U_\beta\|_{\mathcal X_\beta^s}\le R\delta$ proves \eqref{3-8}.
	\end{proof}
	
	The residual estimates, together with the inverse from Proposition~\ref{prop3-6}, give a fixed-point construction.
	
	\begin{proposition}
		\label{prop3-9}
		Let $s>\frac12$. There exist $W_0>0$, $R>0$, and $C_s>0$ such that, for every $0<W\le W_0$, the equation $\mathscr F_W(u,\lambda)=0$ has a unique solution $U_W=(u_W,\lambda_W)$ in $\mathbb B_W(R)$. Moreover,
		\begin{equation}
			\|u_W-u_{\beta_W}\|_{X_{\beta_W}^s}+|\lambda_W|\le C_s\delta_W.
			\label{3-9}
		\end{equation}
	\end{proposition}
	
	\begin{proof}
		Set $U_\beta:=U_{\beta_W}$ and $A_W:=D_U\mathscr F_W(U_\beta)$. By Proposition~\ref{prop3-6}, $A_W$ is invertible with $\|A_W^{-1}\|\le M_s$, uniformly for sufficiently small $W$. Define the frozen Newton map
		\[
		\mathcal T_W(U):=U-A_W^{-1}\mathscr F_W(U).
		\]
		Let $C_s^{\rm res}$ be the constant in \eqref{3-7}, and choose $R:=2M_sC_s^{\rm res}+1$. Then
		\[
		\|\mathcal T_W(U_\beta)-U_\beta\|_{\mathcal X_{\beta_W}^s}\le M_sC_s^{\rm res}\delta_W.
		\]
		By Lemma~\ref{lem3-8}, for every $U\in\mathbb B_W(R)$,
		\[
		\|D\mathcal T_W(U)\|_{\mathcal L(\mathcal X_{\beta_W}^s)}\le M_sC_{s,R}\delta_W.
		\]
		We choose $W_0$ so that $M_sC_{s,R}\delta_{W_0}\le\frac12$. Since $\delta_W$ increases with $W$, the map $\mathcal T_W$ has Lipschitz constant at most $\frac12$ on $\mathbb B_W(R)$ for all $0<W\le W_0$. Furthermore,
		\[
		\|\mathcal T_W(U)-U_\beta\|_{\mathcal X_{\beta_W}^s}\le\frac12R\delta_W+M_sC_s^{\rm res}\delta_W\le R\delta_W.
		\]
		Thus $\mathcal T_W$ maps the complete ball $\mathbb B_W(R)$ into itself and is a contraction. Its unique fixed point $U_W$ solves $\mathscr F_W(U_W)=0$ and satisfies
		\[
		\|U_W-U_\beta\|_{\mathcal X_{\beta_W}^s}\le2M_sC_s^{\rm res}\delta_W,
		\]
		which gives \eqref{3-9}.
	\end{proof}
	
	\medskip
	\noindent\emph{A quantitative small-speed range.}
	In the normalized case $g=\sigma=\Gamma=1$, we may take $W_0=(12\pi)^{-1}$. Indeed, at $s=1$, use the weighted product norms
	\[
	\|(u,\lambda)\|_{\mathcal X_{\beta,30}}:=\|u\|_{X_\beta^1}+30|\lambda|,\qquad \|(f,c)\|_{\mathcal Y_{30}}:=\|f\|_{H^1}+30|c|,
	\]
	and the ball $\widehat{\mathbb B}_W:=\{U:\|U-U_{\beta_W}\|_{\mathcal X_{\beta_W,30}}\le400\delta_W\}$. The Fourier multiplier bounds, the Sobolev product estimates, and the smoothing estimates used above give, uniformly for $0<W\le W_0$,
	\begin{equation}
		\begin{gathered}
			\|\mathcal L_{\beta_W}^{-1}\|\le2.54,\qquad \|\mathscr F_W(U_{\beta_W})\|_{\mathcal Y_{30}}\le64\delta_W,\\
			\|\mathcal L_{\beta_W}^{-1}(A_W-\mathcal L_{\beta_W})\|\le0.16,\qquad \sup_{U\in\widehat{\mathbb B}_W}\|\mathcal L_{\beta_W}^{-1}(D_U\mathscr F_W(U)-\mathcal L_{\beta_W})\|\le0.20.
		\end{gathered}
		\label{3-9a}
	\end{equation}
	Consequently,
	\[
	\|A_W^{-1}\|<3.03,\qquad \|\mathcal T_W(U_{\beta_W})-U_{\beta_W}\|<194\delta_W,\qquad \sup_{U\in\widehat{\mathbb B}_W}\|D\mathcal T_W(U)\|\le\frac37.
	\]
	Since $194+400(3/7)<400$, the map is a strict contraction of $\widehat{\mathbb B}_W$ into itself. The same estimates give $|\lambda|<0.012907$ and $\|(\delta_W|D|u,\delta_Wu')\|_{L^\infty}<0.003146$, ensuring $1+\delta_W|D|u>0$ and $J_{\delta_W}(u)>0$ throughout the ball. This proves the stated sufficient threshold, including its endpoint.
	
	\smallskip
	
	We return to physical coordinates and derive the small-speed asymptotics.
	
	\smallskip
	
	\subsection{Physical reconstruction and small-speed asymptotics}
	We reconstruct the traveling waves from the rescaled solutions of Proposition~\ref{prop3-9} and derive the expansions in Theorem~\ref{thm1}. Undoing the scaling gives the conformal profile and depth; the physical coordinates require one further change of variables.
	
	We begin with a refinement of the rescaled solution. Fix $s_0>\frac12$, and let $U_W^{(s)}$ denote the solution constructed in Proposition~\ref{prop3-9} at a Sobolev index $s\ge s_0$. The continuous embeddings $X_\beta^s\hookrightarrow X_\beta^{s_0}$ are uniform for $0<\beta\le1$, so these solutions lie in a common ball of radius $R_s\delta_W$ about $U_{\beta_W}$ in $\mathcal X_{\beta_W}^{s_0}$. Proposition~\ref{prop3-6} and Lemma~\ref{lem3-8} imply that the frozen Newton map has Lipschitz constant at most $C_{s_0,R_s}\delta_W$ on this ball. For sufficiently small $W$, uniqueness therefore gives $U_W^{(s)}=U_W^{(s_0)}$. Thus the same solution branch belongs to every prescribed Sobolev space, with a small-speed threshold that may depend on the index.
	
	\begin{corollary}
		\label{cor3-10}
		For every fixed $m\ge0$, the solution $(u_W,\lambda_W)$ of Proposition~\ref{prop3-9} satisfies
		\begin{equation}
			u_W=-Q-\beta_WQ''+O_{H^m}\bigl(\delta_W+\beta_W^2\bigr),
			\label{3-10}
		\end{equation}
		and
		\begin{equation}
			\lambda_W=\frac38\delta_W+O\bigl(\delta_W\beta_W+\delta_W^2\bigr)
			\label{3-11}
		\end{equation}
		as $W\to0^+$. Consequently, the conformal vortex depth $a_W:=a_0(W)(1+\lambda_W)$ satisfies
		\begin{equation}
			a_W=\frac{\Gamma}{4\pi W}+\frac{3}{2g}W^2+O(W^4).
			\label{3-12}
		\end{equation}
	\end{corollary}
	
	\begin{proof}
		For $\beta=\beta_W$, the reference profile satisfies the exact expansion
		\[
		u_\beta=-Q-\beta Q''-\beta^2(I-\beta\partial_X^2)^{-1}Q^{(4)}.
		\]
		Since $\|(I-\beta\partial_X^2)^{-1}\|_{H^m\to H^m}\le1$, Proposition~\ref{prop3-9}, applied at a sufficiently high fixed Sobolev index, gives \eqref{3-10}.
		
		The vortex equilibrium equation can be written as
		\[
		\frac{\lambda_W}{1+\lambda_W}=\delta_W\frac{\bigl(e^{-(1+\lambda_W)|D|}u_W''\bigr)(0)}{1+\delta_W\bigl(e^{-(1+\lambda_W)|D|}|D|u_W\bigr)(0)}.
		\]
		Using $\lambda_W=O(\delta_W)$, \eqref{3-10}, and the smoothing properties of the Poisson semigroup, we obtain
		\[
		\bigl(e^{-(1+\lambda_W)|D|}u_W''\bigr)(0)=\bigl(e^{-|D|}(-Q)''\bigr)(0)+O(\beta_W+\delta_W).
		\]
		The Fourier formula $\widehat Q(\xi)=\pi|\xi|e^{-|\xi|}$ gives
		\[
		\bigl(e^{-|D|}(-Q)''\bigr)(0)=\int_0^\infty\xi^3e^{-2\xi}\,d\xi=\frac38.
		\]
		Substituting into the equilibrium equation proves \eqref{3-11}. Finally, since $a_0(W)=\Gamma/(4\pi W)$ and $\delta_W=16\pi W^3/(g\Gamma)$, we have
		\[
		a_0(W)\frac38\delta_W=\frac{3}{2g}W^2,\qquad a_0(W)(\delta_W\beta_W+\delta_W^2)=O(W^4),
		\]
		which yields \eqref{3-12}.
	\end{proof}
	
	Returning to conformal coordinates, define
	\[
	Y_W(x):=\rho_Wu_W\left(\frac{x}{a_0(W)}\right),\qquad a_W:=a_0(W)(1+\lambda_W),\qquad \rho_W=\frac{4W^2}{g}.
	\]
	Lemma~\ref{lem3-1} shows that $(Y_W,a_W)$ satisfies both equations of the coupled system \eqref{2-4}. We verify that the conformal map gives a physical fluid domain.
	
	\begin{lemma}
		\label{lem3-11}
		For all sufficiently small $W>0$,
		\begin{equation}
			1+|D_x|Y_W(x)\ge\frac12\qquad\text{for all }x\in\mathbb R.
			\label{3-13}
		\end{equation}
		Consequently, $Z_{Y_W}$ is univalent in $\Pi_-$ and maps $\Pi_-$ conformally onto a fluid domain $\Omega_W$ with a global graph as its free boundary. The physical vortex position $p_W:=Z_{Y_W}(-\mathrm{i}a_W)$ lies on the symmetry axis, and $(\Omega_W,p_W)$ is a capillary--gravity traveling wave with point-vortex vorticity $\omega_W=\Gamma\delta_{p_W}$.
	\end{lemma}
	
	\begin{proof}
		By scaling,
		\[
		|D_x|Y_W(x)=\delta_W|D|u_W\left(\frac{x}{a_0(W)}\right).
		\]
		The reference profiles are uniformly bounded in every fixed Sobolev space. Proposition~\ref{prop3-9} also gives
		\[
		\|u_W-u_{\beta_W}\|_{H^{s+1}}\le C_s\frac{\delta_W}{\sqrt{\beta_W}}=C_s\mu_W.
		\]
		For $s>\frac12$, Sobolev embedding therefore yields $\||D|u_W\|_{L^\infty}\le C_s$, and hence $\||D_x|Y_W\|_{L^\infty}\le C_s\delta_W$. Since $\delta_W\to0$, we obtain \eqref{3-13}.
		
		Lemma~\ref{lem2-1} now gives the univalence of $Z_{Y_W}$ and the graph property of its boundary. The Bernoulli equation and kinematic condition follow from Lemma~\ref{lem2-2}, while Lemma~\ref{lem2-3} identifies $\mathcal P(Y_W,a_W,W)=0$ with the vanishing of the regular relative velocity at the vortex. The resulting configuration is therefore a traveling wave, and reflection symmetry places $p_W$ on the vertical axis.
	\end{proof}
	
	Write $p_W=(0,-h_W)$ for the physical vortex position. The following corollary gives its depth in terms of the conformal coordinate $a_W$.
	
	\begin{corollary}
		\label{cor3-12}
		The physical vortex depth satisfies
		\begin{equation}
			h_W=a_W-\rho_W\bigl(e^{-(1+\lambda_W)|D|}u_W\bigr)(0),
			\label{3-14}
		\end{equation}
		and consequently
		\begin{equation}
			h_W=\frac{\Gamma}{4\pi W}+\frac{5}{2g}W^2+O(W^4).
			\label{3-15}
		\end{equation}
	\end{corollary}
	
	\begin{proof}
		Since $Y_W$ is even, the harmonic extension of $HY_W$ vanishes on the imaginary axis. Evaluating the conformal map at the vortex gives
		\[
		Z_{Y_W}(-\mathrm{i}a_W)=-\mathrm{i}a_W+\mathrm{i}\bigl(e^{-a_W|D_x|}Y_W\bigr)(0),
		\]
		which proves \eqref{3-14} after rescaling.
		
		By \eqref{3-10} and \eqref{3-11},
		\[
		\bigl(e^{-(1+\lambda_W)|D|}u_W\bigr)(0)=-\bigl(e^{-|D|}Q\bigr)(0)+O(\beta_W+\delta_W).
		\]
		Using the Fourier transform of $Q$, we find
		\[
		\bigl(e^{-|D|}Q\bigr)(0)=\int_0^\infty\xi e^{-2\xi}\,d\xi=\frac14.
		\]
		Thus
		\[
		h_W=a_W+\frac14\rho_W+O\bigl(\rho_W\beta_W+\rho_W\delta_W\bigr).
		\]
		Since $\rho_W=4W^2/g$, the last error is $O(W^4)$, and \eqref{3-12} yields \eqref{3-15}.
	\end{proof}
	
	The $W^2/g$ contribution to the physical depth comes from evaluating the conformal map at the vortex.
	
	To express the free surface in physical coordinates, set
	\[
	\Xi_W(x):=x-HY_W(x),\qquad \eta_W(\xi):=Y_W\bigl(\Xi_W^{-1}(\xi)\bigr).
	\]
	By \eqref{3-13}, $\Xi_W'=1+|D_x|Y_W\ge1/2$, so $\Xi_W$ is a diffeomorphism of $\mathbb R$. With $x=a_0(W)X$, define
	\begin{equation}
		\chi_W(X):=X-\delta_WH u_W(X).
		\label{3-16}
	\end{equation}
	The physical and conformal surface profiles are then related by
	\begin{equation}
		\eta_W\bigl(a_0(W)\chi_W(X)\bigr)=\rho_Wu_W(X).
		\label{3-17}
	\end{equation}
	
	\begin{proposition}
		\label{prop3-13}
		For every fixed $m\ge0$,
		\begin{equation}
			\frac1{\rho_W}\eta_W(a_0(W)X)=-Q(X)-\beta_WQ''(X)+O_{H^m}\bigl(\delta_W+\beta_W^2\bigr)
			\label{3-18}
		\end{equation}
		as $W\to0^+$. Equivalently,
		\begin{equation}
			\eta_W(a_0(W)X)=-\frac{4W^2}{g}Q(X)-\frac{64\pi^2\sigma}{g^2\Gamma^2}W^4Q''(X)+O_{H^m}(W^5).
			\label{3-19}
		\end{equation}
		In particular,
		\begin{equation}
			\frac{g}{4W^2}\eta_W\left(\frac{\Gamma}{4\pi W}X\right)\longrightarrow-Q(X)=\frac{X^2-1}{(1+X^2)^2}
			\label{3-20}
		\end{equation}
		in $H^m(\mathbb R)$ for every fixed $m\ge0$.
	\end{proposition}
	
	\begin{proof}
		Fix $m\ge0$ and choose an integer $n>m+2$. Equation \eqref{3-17} gives
		\[
		\frac1{\rho_W}\eta_W(a_0(W)X)=u_W\bigl(\chi_W^{-1}(X)\bigr).
		\]
		The uniform Sobolev bounds on $u_W$ and \eqref{3-16} imply
		\[
		\|\chi_W-\operatorname{Id}\|_{H^n}\le C_n\delta_W.
		\]
		Since $\chi_W'$ is uniformly close to one, the Sobolev inversion estimates yield
		\[
		\|\chi_W^{-1}-\operatorname{Id}\|_{H^n}\le C_n\delta_W.
		\]
		Writing $\zeta_W:=\chi_W^{-1}-\operatorname{Id}$, we have
		\[
		u_W(X+\zeta_W(X))-u_W(X)=\zeta_W(X)\int_0^1u_W'(X+t\zeta_W(X))\,dt.
		\]
		The uniform composition and product bounds in $H^{n-1}$ therefore give
		\[
		\|u_W\circ\chi_W^{-1}-u_W\|_{H^{n-1}}\le C_n\delta_W.
		\]
		Combining this with \eqref{3-10} proves \eqref{3-18}. Finally,
		\[
		\rho_W=\frac{4W^2}{g},\qquad \rho_W\beta_W=\frac{64\pi^2\sigma}{g^2\Gamma^2}W^4,\qquad \rho_W(\delta_W+\beta_W^2)=O(W^5),
		\]
		which gives \eqref{3-19}. The convergence \eqref{3-20} follows since $\beta_W,\delta_W\to0$.
	\end{proof}
	
	The proof of Theorem~\ref{thm1} is complete. Proposition~\ref{prop3-9} and Lemma~\ref{lem3-11} give the coupled traveling waves, while Corollaries~\ref{cor3-10} and \ref{cor3-12} and Proposition~\ref{prop3-13} establish the depth and free-surface expansions. For $g=\sigma=\Gamma=1$, the quantitative contraction estimate \eqref{3-9a} allows the choice $W_0=(12\pi)^{-1}$.
	
	\bigskip
	
	\section{Kirchhoff--Routh equilibria in general water-wave domains}
	\label{sec4}
	Here the fluid domain is prescribed and Bernoulli's condition is not imposed on its boundary. The Green and Robin functions give a Kirchhoff--Routh description of vortex equilibrium. Reflection symmetry then reduces the problem to a one-dimensional variational equation in the conformal depth.
	
	Maximizing equilibria exist at every positive speed; at small speed they are unique and admit refined depth expansions. We also establish uniform location bounds for symmetric equilibria in uniformly admissible families, providing the a priori estimates used in the global continuation theory of Part~II \cite{DuZouPartII}.
	
	\subsection{Admissible water-wave domains}
	We specify the geometric assumptions on the prescribed domain. Its boundary need not be a graph or satisfy the Bernoulli condition, and we do not require the conformal map to be close to the identity. Identify $\mathbb R^2$ with $\mathbb C$ and write
	\[
	\Pi_-:=\{z\in\mathbb C:\Im z<0\}.
	\]
	
	\begin{definition}
		\label{def4-1}
		Let $k\ge3$ and $\alpha\in(0,1)$. A simply connected domain $\Omega\subset\mathbb C$ is called an \emph{admissible symmetric asymptotically flat water-wave domain} if there exists a conformal map
		\[
		Z:\Pi_-\longrightarrow\Omega
		\]
		with the following properties.
		\begin{enumerate}
			\item[(i)] The map $Z$ extends to a $C^{k,\alpha}$ diffeomorphism from $\overline{\Pi_-}$ onto $\overline{\Omega}$ and satisfies $Z'(z)\neq0$ on $\overline{\Pi_-}$.
			\item[(ii)] The conformal parametrization is symmetric with respect to the vertical axis:
			\begin{equation}
				Z(-\overline z)=-\overline{Z(z)}.
				\label{4-1}
			\end{equation}
			\item[(iii)] The map is hydrodynamically normalized at infinity:
			\begin{equation}
				Z(z)-z\longrightarrow0,\qquad Z'(z)\longrightarrow1
				\label{4-2}
			\end{equation}
			as $|z|\to\infty$ in $\overline{\Pi_-}$.
			\item[(iv)] There exist constants $0<c_Z\le C_Z<\infty$ such that
			\begin{equation}
				c_Z\le |Z'(z)|\le C_Z\qquad\text{for all }z\in\overline{\Pi_-}.
				\label{4-3}
			\end{equation}
		\end{enumerate}
		The class of all such domains will be denoted by $\mathfrak A^{k,\alpha}$.
	\end{definition}
	
	The normalization \eqref{4-2} fixes the conformal coordinates; symmetry maps the reference imaginary axis onto the physical vertical axis. We state \eqref{4-3} explicitly to track constants across uniformly admissible families, even though these bounds follow from the assumptions for a fixed domain.
	
	\begin{lemma}
		\label{lem4-2}
		Let $\Omega\in\mathfrak A^{k,\alpha}$ and let $Z:\Pi_-\to\Omega$ be its normalized conformal map. Then, for every $a>0$,
		\begin{equation}
			Z(-\mathrm{i}a)\in\mathrm{i}\mathbb R,\qquad Z'(-\mathrm{i}a)\in(0,\infty),\qquad Z''(-\mathrm{i}a)\in\mathrm{i}\mathbb R.
			\label{4-4}
		\end{equation}
		Moreover,
		\[
		Z\bigl(\{-\mathrm{i}a:a>0\}\bigr)=\Omega\cap\mathrm{i}\mathbb R.
		\]
		Thus $p(a):=Z(-\mathrm{i}a)$ parametrizes the entire symmetry axis of $\Omega$, and its physical vertical coordinate is strictly decreasing as $a$ increases.
	\end{lemma}
	
	\begin{proof}
		For $z=-\mathrm{i}a$, the symmetry relation \eqref{4-1} gives $Z(z)=-\overline{Z(z)}$, so $Z(-\mathrm{i}a)$ is purely imaginary. Differentiating \eqref{4-1} yields
		\[
		Z'(-\overline z)=\overline{Z'(z)},\qquad Z''(-\overline z)=-\overline{Z''(z)}.
		\]
		Hence $Z'(-\mathrm{i}a)$ is real and $Z''(-\mathrm{i}a)$ is purely imaginary. Since $Z'$ never vanishes and tends to $1$ at infinity, continuity gives $Z'(-\mathrm{i}a)>0$.
		
		Conversely, let $p\in\Omega\cap\mathrm{i}\mathbb R$ and write $p=Z(z)$ for some $z\in\Pi_-$. By symmetry,
		\[
		Z(-\overline z)=-\overline{Z(z)}=p=Z(z).
		\]
		The injectivity of $Z$ implies $z=-\overline z$, so $z=-\mathrm{i}a$ for some $a>0$. This proves the claimed parametrization. Finally,
		\[
		p'(a)=-\mathrm{i}Z'(-\mathrm{i}a),
		\]
		which points strictly downward by \eqref{4-4}.
	\end{proof}
	
	For the remainder of this section, we write $p_a:=Z(-\mathrm{i}a)$ for the point on the physical symmetry axis corresponding to conformal depth $a$. The Green and Robin functions of $\Omega$ then determine its regular drift.
	
	\smallskip
	
	\subsection{Green and Robin functions}
	
	Let $\Omega\in\mathfrak A^{k,\alpha}$ and let $Z:\Pi_-\to\Omega$ be its normalized conformal map. We normalize the Dirichlet Green function of $-\Delta$ by
	\[
	G_\Omega(p,q)=\frac1{2\pi}\log\frac1{|p-q|}-H_\Omega(p,q),
	\]
	and define the Robin function by $R_\Omega(p):=H_\Omega(p,p)$.
	
	\begin{lemma}
		\label{lem4-3}
		For $z,\zeta\in\Pi_-$ with $z\ne\zeta$,
		\begin{equation}
			G_\Omega\bigl(Z(z),Z(\zeta)\bigr)=G_{\Pi_-}(z,\zeta),
			\label{4-5}
		\end{equation}
		where
		\[
		G_{\Pi_-}(z,\zeta)=\frac1{2\pi}\log\frac{|z-\overline\zeta|}{|z-\zeta|}.
		\]
		Consequently,
		\[
		H_\Omega\bigl(Z(z),Z(\zeta)\bigr)=H_{\Pi_-}(z,\zeta)-\frac1{2\pi}\log\left|\frac{Z(z)-Z(\zeta)}{z-\zeta}\right|,
		\qquad
		H_{\Pi_-}(z,\zeta)=-\frac1{2\pi}\log|z-\overline\zeta|.
		\]
		In particular,
		\begin{equation}
			R_\Omega(Z(z))=R_{\Pi_-}(z)-\frac1{2\pi}\log|Z'(z)|=-\frac1{2\pi}\log\!\left(2|\Im z|\,|Z'(z)|\right).
			\label{4-6}
		\end{equation}
	\end{lemma}
	
	\begin{proof}
		For fixed $\zeta\in\Pi_-$, the function $G_\Omega(Z(z),Z(\zeta))$ is harmonic in $\Pi_-\setminus\{\zeta\}$, vanishes on $\partial\Pi_-$, and has the required logarithmic singularity at $\zeta$. Uniqueness of the Dirichlet Green function gives \eqref{4-5}, with the explicit half-plane formula following from the method of images. Comparing the Green function decompositions and letting $\zeta\to z$ gives
		\[
		R_\Omega(Z(z))=R_{\Pi_-}(z)-\frac1{2\pi}\log|Z'(z)|.
		\]
		Since $R_{\Pi_-}(z)=-\frac1{2\pi}\log(2|\Im z|)$, we obtain \eqref{4-6}.
	\end{proof}
	
	On the symmetry axis, \eqref{4-6} has a particularly simple form.
	
	\begin{corollary}
		\label{cor4-4}
		For $a>0$, let $p_a:=Z(-\mathrm{i}a)$. Then
		\begin{equation}
			R_\Omega(p_a)=-\frac1{2\pi}\log\!\left(2a\,Z'(-\mathrm{i}a)\right),
			\label{4-7}
		\end{equation}
		and
		\begin{equation}
			\frac{d}{da}R_\Omega(p_a)=-\frac1{2\pi a}+\frac{\mathrm{i}}{2\pi}\frac{Z''(-\mathrm{i}a)}{Z'(-\mathrm{i}a)}.
			\label{4-8}
		\end{equation}
	\end{corollary}
	
	\begin{proof}
		The first identity follows from \eqref{4-6} and $Z'(-\mathrm{i}a)>0$. Differentiating it gives \eqref{4-8}.
	\end{proof}
	
	The two terms in \eqref{4-8} recover the flat-boundary drift and its geometric correction from Section~\ref{sec2}.
	
	\begin{lemma}
		\label{lem4-5}
		Let $d(p):=\operatorname{dist}(p,\partial\Omega)$, and let $q(p)$ be the unique nearest boundary point in a sufficiently small tubular neighborhood. As $d(p)\to0$ from inside $\Omega$,
		\begin{equation}
			R_\Omega(p)=-\frac1{2\pi}\log\bigl(2d(p)\bigr)+O(d(p)),
			\label{4-9}
		\end{equation}
		and
		\begin{equation}
			\nabla R_\Omega(p)=\frac1{2\pi d(p)}\,n_{\rm out}(q(p))+O(1).
			\label{4-10}
		\end{equation}
		The estimates are locally uniform along $\partial\Omega$.
	\end{lemma}
	
	\begin{proof}
		Near a fixed boundary point, use tubular coordinates $p=q-d\,n_{\rm out}(q)$, so that $\nabla d=-n_{\rm out}(q(p))$. Let $z=Z^{-1}(p)$ and set $\tau(p):=-\Im z$ and $\ell(p):=|Z'(z)|$. The function $\tau\ell$ vanishes on the boundary and has inward normal derivative one. Since $k\ge3$, we may write
		\[
		\tau(p)\ell(p)=d(p)b(p),
		\qquad b|_{\partial\Omega}=1,
		\]
		with $b$ of class $C^1$ and $\nabla b=O(1)$ locally. Thus $b=1+O(d)$, and \eqref{4-6} becomes
		\[
		R_\Omega(p)=-\frac1{2\pi}\log\bigl(2d(p)\bigr)-\frac1{2\pi}\log b(p).
		\]
		This gives \eqref{4-9}. Differentiating yields
		\[
		\nabla R_\Omega(p)=\frac{n_{\rm out}(q(p))}{2\pi d(p)}-\frac{\nabla b(p)}{2\pi b(p)},
		\]
		which proves \eqref{4-10}. The estimates are uniform on smaller compact boundary neighborhoods.
	\end{proof}
	
	The normalization at infinity also controls the Robin function at large depth.
	
	\begin{lemma}
		\label{lem4-6}
		As $a\to\infty$,
		\begin{equation}
			R_\Omega(p_a)=-\frac1{2\pi}\log(2a)+o(1),
			\label{4-11}
		\end{equation}
		and
		\begin{equation}
			\frac{d}{da}R_\Omega(p_a)=-\frac1{2\pi a}+o\!\left(\frac1a\right).
			\label{4-12}
		\end{equation}
		Moreover,
		\begin{equation}
			p_a=-\mathrm{i}a+o(1),
			\label{4-13}
		\end{equation}
		so the physical and conformal depths are asymptotically equivalent.
	\end{lemma}
	
	\begin{proof}
		The normalization at infinity gives $Z'(-\mathrm{i}a)\to1$, so \eqref{4-7} implies \eqref{4-11}. For the derivative, the Cauchy estimate applied to $Z'-1$ on $B_{a/2}(-\mathrm{i}a)\subset\Pi_-$ yields
		\[
		|Z''(-\mathrm{i}a)|\le\frac2a\sup_{B_{a/2}(-\mathrm{i}a)}|Z'(z)-1|=o(a^{-1}).
		\]
		Substituting into \eqref{4-8} proves \eqref{4-12}. Finally, $Z(-\mathrm{i}a)+\mathrm{i}a\to0$ gives \eqref{4-13}.
	\end{proof}
	
	With our sign convention, $R_\Omega(p)\to+\infty$ as $p$ approaches the boundary, whereas $R_\Omega(p_a)\to-\infty$ as $a\to\infty$. The Kirchhoff--Routh function therefore contains the self-interaction term $-\Gamma^2R_\Omega/2$. We now combine it with the harmonic background flow generated by the moving frame.
	
	\smallskip
	
	\subsection{The Kirchhoff--Routh function}
	
	Let $\Omega\in\mathfrak A^{k,\alpha}$ and let $Z:\Pi_-\to\Omega$ be its normalized conformal map. For $W>0$, the background stream function in the moving frame is
	\begin{equation}
		\Psi_{\Omega,W}^0(p):=W\,\Im Z^{-1}(p),\qquad p\in\Omega.
		\label{4-14}
	\end{equation}
	It vanishes on $\partial\Omega$ and satisfies $\Psi_{\Omega,W}^0(p)=W\Im p+o(1)$ at infinity.
	
	For a point vortex of circulation $\Gamma>0$ at $p\in\Omega$, set
	\[
	\Psi_{p,W}(q):=\Gamma G_\Omega(q,p)+\Psi_{\Omega,W}^0(q),\qquad q\ne p,
	\]
	and $v_{p,W}:=-\nabla^\perp\Psi_{p,W}$. Its singular velocity is $\frac{\Gamma}{2\pi}\frac{(q-p)^\perp}{|q-p|^2}$, so the regular velocity at the vortex is
	\[
	v_{\rm reg}(p):=\lim_{q\to p}\left[v_{p,W}(q)-\frac{\Gamma}{2\pi}\frac{(q-p)^\perp}{|q-p|^2}\right].
	\]
	
	\begin{lemma}
		\label{lem4-7}
		For every $p\in\Omega$,
		\begin{equation}
			v_{\rm reg}(p)=\frac{\Gamma}{2}\nabla^\perp R_\Omega(p)-\nabla^\perp\Psi_{\Omega,W}^0(p).
			\label{4-15}
		\end{equation}
	\end{lemma}
	
	\begin{proof}
		Using the Green function decomposition from Section~\ref{sec4},
		\[
		v_{p,W}(q)=\frac{\Gamma}{2\pi}\frac{(q-p)^\perp}{|q-p|^2}+\Gamma\nabla_q^\perp H_\Omega(q,p)-\nabla^\perp\Psi_{\Omega,W}^0(q).
		\]
		Subtracting the singular term and taking $q\to p$ gives
		\[
		v_{\rm reg}(p)=\Gamma\nabla_1^\perp H_\Omega(p,p)-\nabla^\perp\Psi_{\Omega,W}^0(p).
		\]
		Since $H_\Omega(p,q)=H_\Omega(q,p)$, we have $\nabla R_\Omega(p)=2\nabla_1H_\Omega(p,p)$, proving \eqref{4-15}.
	\end{proof}
	
	This leads to the Kirchhoff--Routh function
	\begin{equation}
		\mathscr K_{\Omega,W}(p):=-\frac{\Gamma^2}{2}R_\Omega(p)+\Gamma\Psi_{\Omega,W}^0(p).
		\label{4-16}
	\end{equation}
	
	\begin{proposition}
		\label{prop4-8}
		The regular velocity satisfies
		\begin{equation}
			v_{\rm reg}(p)=-\frac1{\Gamma}\nabla^\perp\mathscr K_{\Omega,W}(p).
			\label{4-17}
		\end{equation}
		Consequently, a point vortex is stationary in the moving frame precisely at the critical points of $\mathscr K_{\Omega,W}$.
	\end{proposition}
	
	\begin{proof}
		Differentiating \eqref{4-16} gives
		\[
		-\frac1{\Gamma}\nabla^\perp\mathscr K_{\Omega,W}
		=\frac{\Gamma}{2}\nabla^\perp R_\Omega-\nabla^\perp\Psi_{\Omega,W}^0.
		\]
		The result follows from Lemma~\ref{lem4-7}.
	\end{proof}
	
	Formula \eqref{4-6} gives the Kirchhoff--Routh function in conformal coordinates.
	
	\begin{corollary}
		\label{cor4-9}
		For every $z\in\Pi_-$,
		\begin{equation}
			\mathscr K_{\Omega,W}(Z(z))=\frac{\Gamma^2}{4\pi}\log\!\left(2|\Im z|\,|Z'(z)|\right)+\Gamma W\Im z.
			\label{4-18}
		\end{equation}
	\end{corollary}
	
	\begin{proof}
		Substituting \eqref{4-6} and $\Psi_{\Omega,W}^0(Z(z))=W\Im z$ into \eqref{4-16} gives \eqref{4-18}.
	\end{proof}
	
	In \eqref{4-18}, the logarithmic factor $2|\Im z|$ describes the self-interaction associated with the flat half-plane, while $|Z'(z)|$ accounts for the geometry of the prescribed domain. Differentiating this geometric factor produces the conformal correction $Z''/Z'$ in the regular-drift equation of Section~\ref{sec2}.
	
	\begin{remark}[Several vortices and non-symmetric configurations]
		The same representation applies to $N$ point vortices at distinct points $z_1,\ldots,z_N\in\Pi_-$ with circulations $\Gamma_1,\ldots,\Gamma_N$, without imposing symmetry. Writing $p_j=Z(z_j)$, the Kirchhoff--Routh function in the moving frame is
		\[
		\begin{aligned}
			\mathscr K_{Z,W}^{(N)}(z_1,\ldots,z_N)
			={}&\sum_{j=1}^N\frac{\Gamma_j^2}{4\pi}
			\log\!\left(2|\Im z_j|\,|Z'(z_j)|\right)\\
			&+\sum_{1\le j<\ell\le N}
			\frac{\Gamma_j\Gamma_\ell}{2\pi}
			\log\frac{|z_j-\overline{z_\ell}|}{|z_j-z_\ell|}
			+W\sum_{j=1}^N\Gamma_j\Im z_j.
		\end{aligned}
		\]
		The vortex centers are stationary precisely when
		\[
		\nabla_{z_j}\mathscr K_{Z,W}^{(N)}=0,\qquad j=1,\ldots,N.
		\]
		The pairwise terms account for the other vortices and their reflected images; the factors $|Z'(z_j)|$ encode the domain geometry. In a non-symmetric free-boundary problem, the Bernoulli equation would determine $Z$, and these $2N$ equilibrium equations would determine the vortex centers. The symmetric single-vortex problem considered below is a one-dimensional reduction of this system.
	\end{remark}
	
	We restrict $\mathscr K_{\Omega,W}$ to the symmetry axis. With $p_a=Z(-\mathrm{i}a)$, define
	\[
	K_{\Omega,W}(a):=\mathscr K_{\Omega,W}(p_a),\qquad a>0.
	\]
	Since $Z'(-\mathrm{i}a)>0$, \eqref{4-18} becomes
	\begin{equation}
		K_{\Omega,W}(a)=\frac{\Gamma^2}{4\pi}\log\!\left(2aZ'(-\mathrm{i}a)\right)-\Gamma Wa.
		\label{4-19}
	\end{equation}
	
	\begin{proposition}
		\label{prop4-10}
		For every $a>0$,
		\begin{equation}
			K_{\Omega,W}'(a)=-\Gamma\mathcal P_\Omega(a,W),
			\label{4-20}
		\end{equation}
		where
		\begin{equation}
			\mathcal P_\Omega(a,W):=W-\frac{\Gamma}{4\pi a}+\frac{\mathrm{i}\Gamma}{4\pi}\frac{Z''(-\mathrm{i}a)}{Z'(-\mathrm{i}a)}.
			\label{4-21}
		\end{equation}
		Thus $K_{\Omega,W}'(a)=0$ precisely when the regular-drift balance holds at $p_a$.
	\end{proposition}
	
	\begin{proof}
		Differentiating \eqref{4-19} and using $\frac{d}{da}Z'(-\mathrm{i}a)=-\mathrm{i}Z''(-\mathrm{i}a)$ gives
		\[
		K_{\Omega,W}'(a)=\frac{\Gamma^2}{4\pi a}-\frac{\mathrm{i}\Gamma^2}{4\pi}\frac{Z''(-\mathrm{i}a)}{Z'(-\mathrm{i}a)}-\Gamma W,
		\]
		which is \eqref{4-20}.
	\end{proof}
	
	Symmetry ensures that a critical point of the restricted function is also critical for the full Kirchhoff--Routh function.
	
	\begin{lemma}
		\label{lem4-11}
		If $a>0$ satisfies $K_{\Omega,W}'(a)=0$, then
		\[
		\nabla\mathscr K_{\Omega,W}(p_a)=0.
		\]
		Hence $p_a$ is an equilibrium position of the point vortex in the traveling frame.
	\end{lemma}
	
	\begin{proof}
		Symmetry gives $\mathscr K_{\Omega,W}(-\overline p)=\mathscr K_{\Omega,W}(p)$, so the horizontal derivative vanishes on the symmetry axis. Moreover, $p_a'=-\mathrm{i}Z'(-\mathrm{i}a)$ is a nonzero vertical tangent vector by Lemma~\ref{lem4-2}. Since
		\[
		K_{\Omega,W}'(a)=\nabla\mathscr K_{\Omega,W}(p_a)\cdot p_a',
		\]
		the condition $K_{\Omega,W}'(a)=0$ forces the vertical derivative to vanish as well. Thus $\nabla\mathscr K_{\Omega,W}(p_a)=0$.
	\end{proof}
	
	The symmetric equilibrium problem is thus governed by the one-dimensional function $K_{\Omega,W}$. The next subsection uses its behavior near the boundary and at large depth to prove the existence of an interior maximizing equilibrium and obtain its small-speed location.
	
	\smallskip
	
	\subsection{Existence and small-speed localization}
	
	We prove the existence of maximizing equilibria for every $W>0$ and use variational comparisons to locate them as $W\to0^+$. Recall that
	\[
	K_{\Omega,W}(a)=\frac{\Gamma^2}{4\pi}\log\!\left(2aZ'(-\mathrm{i}a)\right)-\Gamma Wa.
	\]
	
	\begin{lemma}
		\label{lem4-12}
		Let $\Omega\in\mathfrak A^{k,\alpha}$ and $W>0$. Then
		\begin{equation}
			K_{\Omega,W}(a)\longrightarrow-\infty\qquad\text{as }a\to0^+\text{ or }a\to\infty.
			\label{4-22}
		\end{equation}
		More precisely,
		\[
		K_{\Omega,W}(a)=\frac{\Gamma^2}{4\pi}\log(2a)+\frac{\Gamma^2}{4\pi}\log Z'(0)+o(1)
		\]
		as $a\to0^+$, while
		\[
		K_{\Omega,W}(a)=\frac{\Gamma^2}{4\pi}\log(2a)-\Gamma Wa+o(1)
		\]
		as $a\to\infty$.
	\end{lemma}
	
	\begin{proof}
		The boundary regularity and hydrodynamic normalization of $Z$ give $Z'(-\mathrm{i}a)=Z'(0)+o(1)$ as $a\to0^+$ and $Z'(-\mathrm{i}a)=1+o(1)$ as $a\to\infty$. Substituting these relations into the formula for $K_{\Omega,W}$ yields the stated expansions. The logarithmic term tends to $-\infty$ at the boundary, while the negative linear term $-\Gamma Wa$ dominates at infinite depth, proving \eqref{4-22}.
	\end{proof}
	
	Define
	\[
	m_{\Omega,W}:=\sup_{a>0}K_{\Omega,W}(a)
	\]
	and
	\begin{equation}
		\mathcal A_{\Omega,W}:=\{a>0:K_{\Omega,W}(a)=m_{\Omega,W}\}.
		\label{4-23}
	\end{equation}
	
	\begin{theorem}
		\label{thm4-13}
		Let $\Omega\in\mathfrak A^{k,\alpha}$, $\Gamma>0$, and $W>0$. Then $\mathcal A_{\Omega,W}$ is a nonempty compact subset of $(0,\infty)$. Every $a_{\Omega,W}\in\mathcal A_{\Omega,W}$ determines an equilibrium point
		\[
		p_{\Omega,W}:=Z(-\mathrm{i}a_{\Omega,W})
		\]
		satisfying
		\[
		\nabla\mathscr K_{\Omega,W}(p_{\Omega,W})=0.
		\]
		Equivalently, a point vortex of circulation $\Gamma$ placed at $p_{\Omega,W}$ is stationary in the traveling frame.
	\end{theorem}
	
	\begin{proof}
		By Lemma~\ref{lem4-12}, every maximizing sequence lies in a compact subinterval of $(0,\infty)$. Continuity of $K_{\Omega,W}$ therefore implies that $\mathcal A_{\Omega,W}$ is nonempty and compact. Each maximizer is an interior critical point, so Proposition~\ref{prop4-10} gives $\mathcal P_\Omega(a_{\Omega,W},W)=0$. Lemma~\ref{lem4-11} then yields $\nabla\mathscr K_{\Omega,W}(p_{\Omega,W})=0$, and the equilibrium property follows from Proposition~\ref{prop4-8}.
	\end{proof}
	
	To locate the maximizing depths as $W\to0^+$, set $a_0(W):=\Gamma/(4\pi W)$ and introduce
	\[
	\Phi(s):=\log s-s+1,\qquad s>0.
	\]
	The function $\Phi$ is nonpositive, vanishes only at $s=1$, and tends to $-\infty$ at both endpoints of $(0,\infty)$.
	
	\begin{lemma}
		\label{lem4-14}
		Let $a_0=\Gamma/(4\pi W)$. Then, for every $s>0$,
		\begin{equation}
			K_{\Omega,W}(a_0s)-K_{\Omega,W}(a_0)=\frac{\Gamma^2}{4\pi}\left[\Phi(s)+\log\frac{Z'(-\mathrm{i}a_0s)}{Z'(-\mathrm{i}a_0)}\right].
			\label{4-24}
		\end{equation}
	\end{lemma}
	
	\begin{proof}
		The identity follows directly from the formula for $K_{\Omega,W}$ and $\Gamma Wa_0=\Gamma^2/(4\pi)$.
	\end{proof}
	
	The comparison identity bounds all maximizing depths, whether or not the maximizer is unique.
	
	\begin{lemma}
		\label{lem4-15}
		There exist constants $0<c_\Omega<1<C_\Omega<\infty$ and $W_0>0$ such that, for every $0<W<W_0$,
		\begin{equation}
			\mathcal A_{\Omega,W}\subset[c_\Omega a_0(W),C_\Omega a_0(W)].
			\label{4-25}
		\end{equation}
		In particular,
		\[
		\inf_{a\in\mathcal A_{\Omega,W}}a\longrightarrow+\infty
		\]
		as $W\to0^+$.
	\end{lemma}
	
	\begin{proof}
		For $a_W\in\mathcal A_{\Omega,W}$, write $a_W=a_0s_W$. By maximality and \eqref{4-24},
		\[
		0\le\Phi(s_W)+\log\frac{Z'(-\mathrm{i}a_0s_W)}{Z'(-\mathrm{i}a_0)}.
		\]
		The admissibility bounds $c_Z\le Z'(-\mathrm{i}a)\le C_Z$ imply
		\[
		\Phi(s_W)\ge-\log\frac{C_Z}{c_Z}.
		\]
		Since $\Phi$ tends to $-\infty$ at both endpoints, this inequality confines $s_W$ to a compact interval $[c_\Omega,C_\Omega]\subset(0,\infty)$, proving \eqref{4-25}. The last assertion follows from $a_0(W)\to\infty$.
	\end{proof}
	
	The normalization at infinity sharpens the depth estimate and gives uniqueness.
	
	\begin{theorem}
		\label{thm4-16}
		Let $\Omega\in\mathfrak A^{k,\alpha}$ be fixed. Then, for all sufficiently small $W>0$, the maximizing set consists of a single point,
		\[
		\mathcal A_{\Omega,W}=\{a_{\Omega,W}\}.
		\]
		Writing $Z(-\mathrm{i}a_{\Omega,W})=-\mathrm{i}h_{\Omega,W}$, we have
		\begin{equation}
			a_{\Omega,W}=\frac{\Gamma}{4\pi W}+o(1),\qquad h_{\Omega,W}=\frac{\Gamma}{4\pi W}+o(1)
			\label{4-26}
		\end{equation}
		as $W\to0^+$.
	\end{theorem}
	
	\begin{proof}
		Set $A:=\Gamma/(4\pi)$, $a_0:=A/W$, and $F(z):=Z(z)-z$. The hydrodynamic normalization gives $F(z)\to0$ as $|z|\to\infty$. Applying the Cauchy estimates on $B_{a/2}(-\mathrm{i}a)\subset\Pi_-$ yields, for $j=1,2,3$,
		\[
		|F^{(j)}(-\mathrm{i}a)|\le\frac{C_j}{a^j}\sup_{B_{a/2}(-\mathrm{i}a)}|F(z)|=o(a^{-j}).
		\]
		In particular,
		\[
		Z'(-\mathrm{i}a)=1+o(1),\qquad Z''(-\mathrm{i}a)=o(a^{-2}),\qquad Z'''(-\mathrm{i}a)=o(a^{-3}).
		\]
		By Lemma~\ref{lem4-15}, every maximizer satisfies $c_\Omega a_0\le a\le C_\Omega a_0$, and hence tends uniformly to infinity as $W\to0^+$. At such a point, Proposition~\ref{prop4-10} gives
		\[
		W-\frac{A}{a}=-\mathrm{i}A\frac{Z''(-\mathrm{i}a)}{Z'(-\mathrm{i}a)}=o(a^{-2}).
		\]
		Consequently,
		\[
		|a-a_0|=\frac{aa_0}{A}\left|W-\frac{A}{a}\right|=o(1)
		\]
		uniformly over the maximizing set.
		
		To prove uniqueness, let $\alpha(a):=Z'(-\mathrm{i}a)>0$. The same Cauchy estimates give
		\[
		\frac{d^2}{da^2}\log\alpha(a)=-\frac{Z'''(-\mathrm{i}a)}{Z'(-\mathrm{i}a)}+\left(\frac{Z''(-\mathrm{i}a)}{Z'(-\mathrm{i}a)}\right)^2=o(a^{-2}),
		\]
		and therefore
		\[
		K_{\Omega,W}''(a)=\frac{\Gamma^2}{4\pi}\left(-\frac1{a^2}+o(a^{-2})\right)<0
		\]
		for all sufficiently large $a$, independently of $W$. For small $W$, the interval $[c_\Omega a_0,C_\Omega a_0]$ lies entirely in this region, so $K_{\Omega,W}$ is strictly concave there. Since all maximizers belong to this interval, the maximizing depth is unique.
		
		Finally, $Z(-\mathrm{i}a)+\mathrm{i}a=o(1)$ as $a\to\infty$, which gives $h(a)-a=o(1)$ and proves the physical-depth estimate in \eqref{4-26}.
	\end{proof}
	
	The basic admissibility assumptions suffice for uniqueness and the additive depth estimate. Higher-order corrections require more information at infinity.
	
	\smallskip
	
    \subsection{Higher-order asymptotics and uniform families}
    Theorem~\ref{thm4-16} gives the leading depth asymptotics under the basic admissibility assumptions. We first use a far-field expansion of the conformal map to refine the depth asymptotics. We then return to weaker hypotheses to obtain estimates uniform over families of domains.
    
    Assume that
    \refstepcounter{equation}
    \[
    Z(z)=z+\frac{c_1}{z}+\frac{c_2}{z^2}+\frac{c_3}{z^3}+\frac{c_4}{z^4}+R(z).
    \tag{\theequation}\label{4-27}
    \]
    Here $R^{(\ell)}(z)=O(|z|^{-5-\ell})$ for $\ell=0,1,2,3$ as $|z|\to\infty$ in $\overline{\Pi_-}$. By the symmetry relation \eqref{4-1}, $c_1,c_3\in\mathbb R$ and $c_2,c_4\in\mathrm{i}\mathbb R$. The following expansions are for a fixed domain $\Omega$ and fixed circulation $\Gamma$.
    
    \begin{lemma}
    	\label{lem4-17}
    	Under \eqref{4-27}, as $a\to\infty$,
    	\[
    	Z'(-\mathrm{i}a)=1+\frac{c_1}{a^2}+\frac{2\mathrm{i}c_2}{a^3}-\frac{3c_3}{a^4}-\frac{4\mathrm{i}c_4}{a^5}+O(a^{-6}),
    	\]
    	and
    	\refstepcounter{equation}
    	\[
    	\mathrm{i}\frac{Z''(-\mathrm{i}a)}{Z'(-\mathrm{i}a)}=\frac{2c_1}{a^3}+\frac{6\mathrm{i}c_2}{a^4}-\frac{2c_1^2+12c_3}{a^5}-\frac{10\mathrm{i}c_1c_2+20\mathrm{i}c_4}{a^6}+O(a^{-7}).
    	\tag{\theequation}\label{4-28}
    	\]
    	Moreover, if $Z(-\mathrm{i}a)=-\mathrm{i}h(a)$, then
    	\refstepcounter{equation}
    	\[
    	h(a)=a-\frac{c_1}{a}-\frac{\mathrm{i}c_2}{a^2}+\frac{c_3}{a^3}+\frac{\mathrm{i}c_4}{a^4}+O(a^{-5}).
    	\tag{\theequation}\label{4-29}
    	\]
    \end{lemma}
    
    \begin{proof}
    	Evaluating \eqref{4-27} at $z=-\mathrm{i}a$ gives \eqref{4-29}; differentiating yields the expansion of $Z'$. For the drift correction, use
    	\[
    	\mathrm{i}\frac{Z''(-\mathrm{i}a)}{Z'(-\mathrm{i}a)}=-\frac{h''(a)}{h'(a)}.
    	\]
    	Expanding the numerator and reciprocal denominator gives
    	\[
    	-h''(a)=\frac{2c_1}{a^3}+\frac{6\mathrm{i}c_2}{a^4}-\frac{12c_3}{a^5}-\frac{20\mathrm{i}c_4}{a^6}+O(a^{-7}),
    	\]
    	and
    	\[
    	\frac1{h'(a)}=1-\frac{c_1}{a^2}-\frac{2\mathrm{i}c_2}{a^3}+O(a^{-4}).
    	\]
    	Multiplication yields \eqref{4-28}.
    \end{proof}
    
    The far-field coefficients determine the corrections to the unique maximizing depth from Theorem~\ref{thm4-16}.
    
    \begin{theorem}
    	\label{thm4-18}
    	Suppose that $\Omega\in\mathfrak A^{k,\alpha}$ satisfies \eqref{4-27}. Let $a_{\Omega,W}$ denote its unique maximizing depth for sufficiently small $W>0$. Then, as $W\to0^+$,
    	\refstepcounter{equation}
    	\[
    	\begin{aligned}
    		a_{\Omega,W}={}&\frac{\Gamma}{4\pi W}-\frac{8\pi c_1}{\Gamma}W-6\mathrm{i}c_2\left(\frac{4\pi}{\Gamma}\right)^2W^2\\
    		&+(12c_3-6c_1^2)\left(\frac{4\pi}{\Gamma}\right)^3W^3\\
    		&+(20\mathrm{i}c_4-50\mathrm{i}c_1c_2)\left(\frac{4\pi}{\Gamma}\right)^4W^4+O(W^5).
    	\end{aligned}
    	\tag{\theequation}\label{4-30}
    	\]
    \end{theorem}
    
    \begin{proof}
    	Set $A:=\Gamma/(4\pi)$, $t:=W/A$, and $a_0:=t^{-1}$. By Theorem~\ref{thm4-16}, $a_{\Omega,W}=a_0+o(1)$. On $[a_0/2,2a_0]$, Lemma~\ref{lem4-17} gives
    	\[
    	\partial_a\mathcal P_\Omega(a,W)=\frac{A}{a^2}+O(a^{-4})\ge ct^2
    	\]
    	for some $c>0$ independent of small $W$.
    	
    	Seek an approximation of the form
    	\[
    	\widetilde a_W=t^{-1}+b_1t+b_2t^2+b_3t^3+b_4t^4.
    	\]
    	Substitution into the drift equation, using \eqref{4-28}, gives
    	\[
    	\begin{aligned}
    		A^{-1}\mathcal P_\Omega(\widetilde a_W,W)={}&(b_1+2c_1)t^3+(b_2+6\mathrm{i}c_2)t^4\\
    		&+(b_3-b_1^2-6c_1b_1-2c_1^2-12c_3)t^5\\
    		&+(b_4-2b_1b_2-6c_1b_2-24\mathrm{i}c_2b_1-10\mathrm{i}c_1c_2-20\mathrm{i}c_4)t^6+O(t^7).
    	\end{aligned}
    	\]
    	Choosing
    	\[
    	b_1=-2c_1,\qquad b_2=-6\mathrm{i}c_2,\qquad b_3=12c_3-6c_1^2,\qquad b_4=20\mathrm{i}c_4-50\mathrm{i}c_1c_2
    	\]
    	gives $\mathcal P_\Omega(\widetilde a_W,W)=O(t^7)$. Since both $\widetilde a_W$ and $a_{\Omega,W}$ lie in $[a_0/2,2a_0]$ for small $W$, the derivative bound and the mean value theorem yield
    	\[
    	|a_{\Omega,W}-\widetilde a_W|\le Ct^{-2}|\mathcal P_\Omega(\widetilde a_W,W)|=O(t^5).
    	\]
    	Replacing $t$ by $4\pi W/\Gamma$ proves \eqref{4-30}.
    \end{proof}
    
    To recover the physical depth, we evaluate the conformal map at the selected center.
    
    \begin{corollary}
    	\label{cor4-19}
    	Under the assumptions of Theorem~\ref{thm4-18}, write $Z(-\mathrm{i}a_{\Omega,W})=-\mathrm{i}h_{\Omega,W}$. Then, as $W\to0^+$,
    	\refstepcounter{equation}
    	\[
    	\begin{aligned}
    		h_{\Omega,W}={}&\frac{\Gamma}{4\pi W}-\frac{12\pi c_1}{\Gamma}W-7\mathrm{i}c_2\left(\frac{4\pi}{\Gamma}\right)^2W^2\\
    		&+(13c_3-8c_1^2)\left(\frac{4\pi}{\Gamma}\right)^3W^3\\
    		&+(21\mathrm{i}c_4-60\mathrm{i}c_1c_2)\left(\frac{4\pi}{\Gamma}\right)^4W^4+O(W^5).
    	\end{aligned}
    	\tag{\theequation}\label{4-31}
    	\]
    \end{corollary}
    
    \begin{proof}
    	Set $t:=4\pi W/\Gamma$. By \eqref{4-30},
    	\[
    	a_{\Omega,W}^{-1}=t+2c_1t^3+6\mathrm{i}c_2t^4+O(t^5),\qquad a_{\Omega,W}^{-2}=t^2+4c_1t^4+O(t^5).
    	\]
    	Also, $a_{\Omega,W}^{-3}=t^3+O(t^5)$ and $a_{\Omega,W}^{-4}=t^4+O(t^6)$. Substituting these expansions into \eqref{4-29} gives
    	\[
    	\begin{aligned}
    		h_{\Omega,W}={}&t^{-1}-3c_1t-7\mathrm{i}c_2t^2+(13c_3-8c_1^2)t^3\\
    		&+(21\mathrm{i}c_4-60\mathrm{i}c_1c_2)t^4+O(t^5),
    	\end{aligned}
    	\]
    	which proves \eqref{4-31}.
    \end{proof}
    
    The equilibrium equation selects the conformal depth; the physical depth is found by evaluating $Z$ at the selected point. For the coupled waves of Section~\ref{sec3}, the domain also varies with $W$, so the fixed-domain expansion above cannot be applied directly. Instead, the small-speed estimates from Section~\ref{sec3} give
    \[
    \mathrm{i}\frac{Z_{Y_W}''(-\mathrm{i}a_W)}{Z_{Y_W}'(-\mathrm{i}a_W)}=O(W^4).
    \]
    Since the derivative of the flat drift $W-\Gamma/(4\pi a)$ at $a=a_0(W)$ is of order $W^2$, the corresponding correction to the conformal depth is of order $W^2$, as in Theorem~\ref{thm1}. Applying the fixed-domain expansions to a varying family would require uniform control of the far-field remainder at depths comparable to the horizontal wave scale.
    
    We turn to families of prescribed domains, where uniform estimates require only the assumptions stated below, rather than a Laurent expansion.
    
    \begin{definition}
    	\label{def4-20}
    	Let $\mathfrak F\subset\mathfrak A^{k,\alpha}$, and denote by $Z_\Omega:\Pi_-\to\Omega$ the normalized conformal map. We call $\mathfrak F$ \emph{uniformly admissible} if there exist constants $0<c_*<C_*<\infty$, $M_*<\infty$, and a nonnegative function $\varepsilon_{\mathfrak F}(R)\to0$ as $R\to\infty$ such that, uniformly for $\Omega\in\mathfrak F$,
    	\begin{equation}
    		c_*\le |Z_\Omega'(z)|\le C_*\qquad\text{for all }z\in\overline{\Pi_-},
    		\label{4-32}
    	\end{equation}
    	\[
    	\sup_{0\le a\le1}|Z_\Omega''(-\mathrm{i}a)|\le M_*,
    	\]
    	and
    	\begin{equation}
    		\sup_{\substack{z\in\overline{\Pi_-}\\ |z|\ge R}}\left(|Z_\Omega(z)-z|+|Z_\Omega'(z)-1|\right)\le\varepsilon_{\mathfrak F}(R).
    		\label{4-33}
    	\end{equation}
    \end{definition}
    
    The assumptions control the conformal scale, near-boundary drift, and behavior at infinity uniformly across the family. Constants in the following estimates may depend on $\Gamma$ and the family bounds, but not on the individual domain.
    
    \begin{lemma}
    	\label{lem4-21}
    	If $\mathfrak F$ is uniformly admissible, then
    	\begin{equation}
    		\sup_{\Omega\in\mathfrak F}a\left|\frac{Z_\Omega''(-\mathrm{i}a)}{Z_\Omega'(-\mathrm{i}a)}\right|\longrightarrow0
    		\label{4-34}
    	\end{equation}
    	as $a\to\infty$.
    \end{lemma}
    
    \begin{proof}
    	The disk $B_{a/2}(-\mathrm{i}a)$ lies in $\Pi_-$. Applying the Cauchy estimate to $Z_\Omega'-1$ gives
    	\[
    	|Z_\Omega''(-\mathrm{i}a)|\le\frac2a\sup_{B_{a/2}(-\mathrm{i}a)}|Z_\Omega'(z)-1|
    	\le\frac2a\varepsilon_{\mathfrak F}(a/2).
    	\]
    	Dividing by the lower bound in \eqref{4-32} proves \eqref{4-34}.
    \end{proof}
    
    Consequently,
    \[
    \mathcal P_\Omega(a,W)=W-\frac{\Gamma}{4\pi a}+o_{\mathfrak F}(a^{-1})
    \]
    as $a\to\infty$, uniformly over $\Omega\in\mathfrak F$.
    
    \begin{theorem}
    	\label{thm4-22}
    	Let $\mathfrak F$ be uniformly admissible and fix $0<W_-<W_+<\infty$. There exist constants $0<a_*<a^*<\infty$ such that, for every $\Omega\in\mathfrak F$, $W\in[W_-,W_+]$, and symmetric equilibrium depth $a$,
    	\begin{equation}
    		a_*\le a\le a^*.
    		\label{4-35}
    	\end{equation}
    \end{theorem}
    
    \begin{proof}
    	Set $A:=\Gamma/(4\pi)$. At an equilibrium,
    	\[
    	W-\frac{A}{a}+\mathrm{i}A\frac{Z_\Omega''(-\mathrm{i}a)}{Z_\Omega'(-\mathrm{i}a)}=0.
    	\]
    	For $0<a\le1$, Definition~\ref{def4-20} gives
    	\[
    	\left|\frac{Z_\Omega''(-\mathrm{i}a)}{Z_\Omega'(-\mathrm{i}a)}\right|\le\frac{M_*}{c_*}.
    	\]
    	Hence the drift is negative when $a$ is sufficiently small, uniformly over the family and the specified speed interval. For example, we may take
    	\[
    	a_*:=\min\left\{1,\frac{A}{2(W_++AM_*/c_*)}\right\}.
    	\]
    	For large $a$, Lemma~\ref{lem4-21} gives
    	\[
    	\left|a\frac{Z_\Omega''(-\mathrm{i}a)}{Z_\Omega'(-\mathrm{i}a)}\right|\le1
    	\]
    	uniformly over the family. Thus
    	\[
    	\mathcal P_\Omega(a,W)\ge W_- -\frac{2A}{a}>0
    	\]
    	for all sufficiently large $a$, uniformly in $\Omega$ and $W$. Choosing $a^*$ large enough excludes all equilibria beyond it and proves \eqref{4-35}.
    \end{proof}
    
    For maximizing equilibria, the small-speed estimate can be sharpened.
    
    \begin{corollary}
    	\label{cor4-23}
    	Let $\mathfrak F$ be uniformly admissible. Then
    	\begin{equation}
    		\sup_{\Omega\in\mathfrak F}\sup_{a\in\mathcal A_{\Omega,W}}
    		\left|a-\frac{\Gamma}{4\pi W}\right|=o(1)
    		\label{4-36}
    	\end{equation}
    	as $W\to0^+$. The same estimate holds for the physical depths.
    \end{corollary}
    
    \begin{proof}
    	The comparison identity \eqref{4-24} and the uniform bounds \eqref{4-32} first give
    	\[
    	ca_0(W)\le a\le Ca_0(W),\qquad a_0(W):=\frac{\Gamma}{4\pi W},
    	\]
    	for all $\Omega\in\mathfrak F$ and $a\in\mathcal A_{\Omega,W}$, with constants independent of the domain and speed. The Cauchy estimate applied to $Z_\Omega-\operatorname{Id}$ then yields
    	\[
    	|Z_\Omega''(-\mathrm{i}a)|\le\frac{C}{a^2}\varepsilon_{\mathfrak F}(a/2)=o_{\mathfrak F}(a^{-2}).
    	\]
    	At every maximizing depth, the equilibrium equation therefore gives
    	\[
    	W-\frac{\Gamma}{4\pi a}=o_{\mathfrak F}(a^{-2}).
    	\]
    	Multiplying by $aa_0(W)$ and using the uniform comparability of $a$ and $a_0(W)$ proves \eqref{4-36}. The physical-depth estimate follows from the uniform convergence $Z_\Omega(z)-z\to0$ at infinity.
    \end{proof}
    
    \begin{corollary}
    	\label{cor4-24}
    	Under the assumptions of Theorem~\ref{thm4-22}, there exist constants $0<d_*<d^*<\infty$ such that every symmetric equilibrium point $p=Z_\Omega(-\mathrm{i}a)$ satisfies
    	\[
    	d_*\le\operatorname{dist}(p,\partial\Omega)\le d^*.
    	\]
    \end{corollary}
    
    \begin{proof}
    	The distortion estimates for conformal maps give universal constants $C_1,C_2>0$ such that
    	\[
    	C_1a|Z_\Omega'(-\mathrm{i}a)|\le\operatorname{dist}(p,\partial\Omega)\le C_2a|Z_\Omega'(-\mathrm{i}a)|.
    	\]
    	The derivative bounds \eqref{4-32} and depth bounds \eqref{4-35} yield the desired estimates.
    \end{proof}
    
    When the speed stays in a compact positive interval, symmetric equilibria in uniformly admissible domains remain away from the boundary and cannot escape to infinite depth. At small speeds, all maximizing equilibria instead have depth $\Gamma/(4\pi W)+o(1)$, uniformly over the family.
    
    Theorem~\ref{thm2} follows from the existence and compactness of the maximizing set, the small-speed uniqueness and depth asymptotics of Theorem~\ref{thm4-16}, and the refined expansions in Theorem~\ref{thm4-18} and Corollary~\ref{cor4-19}. Theorem~\ref{thm4-22} and Corollaries~\ref{cor4-23}--\ref{cor4-24} provide the uniform location and separation estimates needed for continuation in Part~II \cite{DuZouPartII}.
    
    \bigskip
    
    \section{Regularization of the point vortex}
    \label{sec5}
    
    We replace the point vortex in Theorem~\ref{thm1} with a smooth vortex core of the same circulation. Fix $g,\sigma,\Gamma>0$ and a wave speed $0<W\le W_0$, and let $(Y_W,a_W)$ be the corresponding point-vortex solution. Throughout this section, $W$ and $\Gamma$ remain fixed while the core scale $\varepsilon$ tends to zero. We work near this nonflat wave and use its regular-drift equilibrium to select the limiting core position.
    
    The regularized problem has two free boundaries: the water surface and the boundary of the active vortex region. They vary on different scales: the free surface approaches that of the fixed point-vortex wave, while the core has diameter $O(\varepsilon)$. A joint Lyapunov--Schmidt reduction determines both boundaries, with a finite-dimensional depth equation close to the point-vortex regular-drift condition.
    
    The construction draws on the theory of concentrated Euler vortices, including the variational methods of Turkington \cite{Turkington1983I,Turkington1983II} and the analysis of elliptic free-boundary problems \cite{CaffarelliFriedman1980,FlucherWei1998}. Related regularization and gluing methods were developed in \cite{SmetsVanSchaftingen2010,DeValeriolaVanSchaftingen2013,CaoLiuWei2014,DavilaDelPinoMussoWei2020}. In our setting the water surface remains unknown and affects the vortex equilibrium through the regular drift.
    
    Let $\Omega$ be the physical fluid domain and $Z:\Pi_-\to\Omega$ its normalized conformal map. We write the relative stream function as
    \[
    \Psi=\Psi^{\rm vort}+\Psi_{\Omega,W}^0,\qquad \Psi_{\Omega,W}^0(p):=W\Im Z^{-1}(p),
    \]
    where $\Psi^{\rm vort}$ is generated by the localized vorticity and vanishes on $\partial\Omega$. We seek solutions of
    \begin{equation}
    	-\Delta_p\Psi=\omega,\qquad \omega=\frac1{\varepsilon^2}f(\Psi-\mu)\quad\text{in }\Omega,\qquad \Psi=0\quad\text{on }\partial\Omega,
    	\label{5-1}
    \end{equation}
    with prescribed circulation
    \begin{equation}
    	\int_\Omega\omega\,dp=\Gamma.
    	\label{5-2}
    \end{equation}
    The flux constant $\mu$ is chosen to satisfy \eqref{5-2}. Since $f$ vanishes on $(-\infty,0]$ and is positive on $(0,\infty)$, the active vortex region is
    \[
    D_\varepsilon:=\{p\in\Omega:\Psi(p)>\mu\}=\{p\in\Omega:\omega(p)>0\},
    \]
    Extending $\omega$ by zero outside $\Omega$, we have $\operatorname{supp}\omega=\overline{D_\varepsilon}$. The boundary of $D_\varepsilon$ is the inner free boundary, given by the level curve $\Psi=\mu$. The outer free boundary satisfies
    \begin{equation}
    	\frac12|\nabla_p\Psi|^2+gy+\sigma\kappa=\frac{W^2}{2}\qquad\text{on }\partial\Omega.
    	\label{5-3}
    \end{equation}
    
    To work on the fixed lower half-plane, write $Z=Z_Y$ and set
    \[
    \Phi:=\Psi\circ Z_Y,\qquad \psi:=\Psi^{\rm vort}\circ Z_Y,\qquad \Phi=\psi+Wy.
    \]
    Conformal covariance of the Laplacian transforms \eqref{5-1} into
    \begin{equation}
    	-\Delta\psi=\frac{J(Y)}{\varepsilon^2}f(\psi+Wy-\mu)\qquad\text{in }\Pi_-,\qquad \psi=0\qquad\text{on }y=0,
    	\label{5-4}
    \end{equation}
    while the circulation constraint becomes
    \begin{equation}
    	\frac1{\varepsilon^2}\int_{\Pi_-}J(Y)f(\psi+Wy-\mu)\,dz=\Gamma.
    	\label{5-5}
    \end{equation}
    The Bernoulli equation takes the form
    \begin{equation}
    	gY+\sigma\kappa(Y)+\frac{|\nabla(\psi+Wy)|^2}{2J(Y)}-\frac{W^2}{2}=0\qquad\text{on }y=0.
    	\label{5-6}
    \end{equation}
    The inner elliptic equation and circulation constraint contain the conformal factor $J(Y)$, while the Bernoulli velocity depends on the concentrated vorticity. Both the free surface and vortex position are varied in solving the coupled system.
    
    \subsection{The vortex-core profile and the approximate solution}
    
    The inner construction starts from a radial whole-plane vortex profile.
    
    \begin{assumption}
    	\label{ass5-1}
    	Let $f\in C_{\rm loc}^{2,\alpha}(\mathbb R)$ for some $\alpha\in(0,1)$. We assume the following.
    	\begin{enumerate}
    		\item[(i)] The nonlinearity satisfies
    		\[
    		f(s)=0\quad\text{for }s\le0,\qquad f(s)>0\quad\text{for }s>0.
    		\]
    		
    		\item[(ii)] There exists a radial function $U\in C_{\rm loc}^{2,\alpha}(\mathbb R^2)$ satisfying
    		\begin{equation}
    			-\Delta U=f(U)\qquad\text{in }\mathbb R^2,
    			\label{5-7}
    		\end{equation}
    		with
    		\[
    		U>0\quad\text{in }B_1,\qquad U=0\quad\text{on }\partial B_1,\qquad U<0\quad\text{in }\mathbb R^2\setminus\overline{B_1}.
    		\]
    		
    		\item[(iii)] The profile has total circulation $\Gamma$:
    		\begin{equation}
    			\int_{\mathbb R^2}f(U)\,dX=\Gamma.
    			\label{5-8}
    		\end{equation}
    		
    		\item[(iv)] The additive normalization of $U$ is chosen so that
    		\begin{equation}
    			U(X)=-\frac{\Gamma}{2\pi}\log|X|\qquad\text{for }|X|\ge1.
    			\label{5-9}
    		\end{equation}
    		
    		\item[(v)] The linearized operator
    		\[
    		L_0\phi:=-\Delta\phi-f'(U)\phi
    		\]
    		is nondegenerate modulo translations: every bounded solution of $L_0\phi=0$ belongs to
    		\begin{equation}
    			\ker L_0=\operatorname{span}\{\partial_1U,\partial_2U\}.
    			\label{5-10}
    		\end{equation}
    		
    		\item[(vi)] The circulation constraint is transversal at $U$:
    		\[
    		\int_{B_1}f'(U)\,dX\ne0.
    		\]
    	\end{enumerate}
    \end{assumption}
    
    The power-law vorticity $f(\tau)=c\tau_+^p$ provides an example when $p>2$, $0<\alpha<\min\{1,p-2\}$, and $c>0$ is chosen to match the prescribed circulation and unit core radius. A smaller H\"older exponent also allows certain $C_{\rm loc}^{1,\alpha}$ nonlinearities. Discontinuous vortex-patch laws require a different treatment of the inner free boundary, as in the domain-variation approach of Cao, Peng, and Yan \cite{CaoPengYan2015}, and are not considered here.
    
    Since $f(U)$ is supported in $\overline{B_1}$, the profile is harmonic outside the unit disk. Radial integration of \eqref{5-7} and \eqref{5-8} gives
    \[
    -2\pi rU_r(r)=\Gamma,\qquad r>1,
    \]
    and hence the logarithmic behavior \eqref{5-9}. In particular,
    \[
    U_r(1)=-\frac{\Gamma}{2\pi}\ne0,
    \]
    which makes the core boundary a nondegenerate level set near the unit circle.
    
    \begin{proof}[Verification of the power-law example]
    	Let $V$ be the positive radial solution of $-\Delta V=V^p$ in $B_1$ with $V=0$ on $\partial B_1$, and set $\Gamma_*:=-2\pi V_r(1)>0$. For any prescribed $\Gamma>0$, choose $c=(\Gamma_*/\Gamma)^{p-1}$ and define
    	\[
    	U(r):=\begin{cases}
    		\displaystyle\frac{\Gamma}{\Gamma_*}V(r),&0\le r\le1,\\[2mm]
    		\displaystyle-\frac{\Gamma}{2\pi}\log r,&r\ge1.
    	\end{cases}
    	\]
    	The values and first derivatives match at $r=1$, and $-\Delta U=c(U_+)^p$ in $\mathbb R^2$. Moreover, $\int_{\mathbb R^2}c(U_+)^p\,dX=\Gamma$ and $\int_{B_1}cp(U_+)^{p-1}\,dX>0$. We verify \eqref{5-10} by separating angular modes.
    	
    	Let $L_0=-\Delta-cp(U_+)^{p-1}$ and decompose a bounded solution of $L_0\phi=0$ into angular Fourier modes. For each angular order, the space of regular radial solutions of the corresponding second-order ODE is one-dimensional. In the radial mode, scaling invariance gives
    	\[
    	\Lambda U:=\frac{2}{p-1}U+rU_r,\qquad L_0\Lambda U=0.
    	\]
    	Every regular radial solution is a multiple of $\Lambda U$, but for $r\ge1$,
    	\[
    	\Lambda U(r)=-\frac{\Gamma}{2\pi}\left(1+\frac{2}{p-1}\log r\right),
    	\]
    	which is unbounded. Thus there is no nonzero bounded radial zero mode. For angular order one, $w:=-U_r$ is a regular radial solution with $w(r)=\Gamma/(2\pi r)$ for $r\ge1$. Uniqueness of the regular radial ODE shows that the bounded modes are precisely $\partial_1U$ and $\partial_2U$.
    	
    	For angular order $m\ge2$, let $v_m$ be a nonzero regular radial coefficient. Since $w$ solves the equation for angular order one, writing $v_m=wq_m$ gives
    	\[
    	(rw^2q_m')'=\frac{m^2-1}{r}w^2q_m.
    	\]
    	After choosing the sign, $q_m(r)\sim C_mr^{m-1}$ near the origin with $C_m>0$, and the identity implies $q_m,q_m'>0$ on $(0,1]$. Since $w'(1)/w(1)=-1$, we obtain
    	\[
    	\frac{v_m'(1)}{v_m(1)}=-1+\frac{q_m'(1)}{q_m(1)}>-1.
    	\]
    	A bounded exterior harmonic mode of order $m$ is proportional to $r^{-m}$ and has logarithmic derivative $-m$ at $r=1$, a contradiction. Hence the bounded kernel of $L_0$ consists exactly of the two translations.
    \end{proof}
    
    Assumption~\ref{ass5-1}(v) identifies translations as the only neutral modes of the radial core. Reflection symmetry excludes horizontal translation; the remaining vertical mode selects the depth. The transversality condition in part~{\rm(vi)} allows the circulation constraint to determine the flux constant $\mu$ locally.
    
    We place the radial core near the fixed point-vortex wave. Let $(Y,a)$ vary in a sufficiently small neighborhood of $(Y_W,a_W)$ and set
    \[
    z_a:=-\mathrm{i}a,\qquad p_{Y,a}:=Z_Y(z_a),\qquad \alpha(Y,a):=Z_Y'(z_a)>0.
    \]
    Since $\varepsilon$ measures the physical core size, the expansion
    \[
    Z_Y(z_a+\zeta)=p_{Y,a}+\alpha(Y,a)\zeta+O(|\zeta|^2)
    \]
    suggests the inner coordinate
    \begin{equation}
    	X:=\frac{\alpha(Y,a)}{\varepsilon}(z-z_a).
    	\label{5-11}
    \end{equation}
    For bounded $X$,
    \[
    Z_Y\left(z_a+\frac{\varepsilon}{\alpha(Y,a)}X\right)=p_{Y,a}+\varepsilon X+O(\varepsilon^2).
    \]
    The unit disk in the inner coordinates therefore has physical radius $\varepsilon$ to leading order. Since $J(Y)(z_a)=\alpha(Y,a)^2$, the leading inner equation is precisely \eqref{5-7}.
    
    Define
    \[
    U_{\varepsilon,Y,a}(z):=U\left(\frac{\alpha(Y,a)}{\varepsilon}(z-z_a)\right).
    \]
    For sufficiently small $\varepsilon$, the source $f(U_{\varepsilon,Y,a})$ is compactly supported in $\Pi_-$. Its Dirichlet projection is
    \begin{equation}
    	P_{\varepsilon,Y,a}(z):=\int_{\Pi_-}G_{\Pi_-}(z,\zeta)\frac{\alpha(Y,a)^2}{\varepsilon^2}f(U_{\varepsilon,Y,a}(\zeta))\,d\zeta.
    	\label{5-12}
    \end{equation}
    It satisfies
    \[
    -\Delta P_{\varepsilon,Y,a}=\frac{\alpha(Y,a)^2}{\varepsilon^2}f(U_{\varepsilon,Y,a})\qquad\text{in }\Pi_-,\qquad P_{\varepsilon,Y,a}=0\qquad\text{on }y=0.
    \]
    By \eqref{5-8} and the change of variables \eqref{5-11},
    \[
    \int_{\Pi_-}\frac{\alpha(Y,a)^2}{\varepsilon^2}f(U_{\varepsilon,Y,a})\,dz=\Gamma.
    \]
    Moreover,
    \begin{equation}
    	P_{\varepsilon,Y,a}\longrightarrow\Gamma G_{\Pi_-}(\,\cdot\,,z_a)
    	\label{5-13}
    \end{equation}
    locally smoothly away from $z_a$, uniformly for $(Y,a)$ near $(Y_W,a_W)$.
    
    The approximate relative stream function is
    \[
    \Phi_{\varepsilon,Y,a}^{\rm app}(z):=P_{\varepsilon,Y,a}(z)+Wy.
    \]
    On the inner scale, the Green-function expansion gives
    \begin{equation}
    	\Phi_{\varepsilon,Y,a}^{\rm app}\left(z_a+\frac{\varepsilon}{\alpha(Y,a)}X\right)-\mu_{\varepsilon,Y,a}^{\rm app}=U(X)+O(\varepsilon)
    	\label{5-14}
    \end{equation}
    for bounded $X$, where the approximate flux constant is
    \begin{equation}
    	\mu_{\varepsilon,Y,a}^{\rm app}:=\frac{\Gamma}{2\pi}\log\frac{2a\alpha(Y,a)}{\varepsilon}-Wa.
    	\label{5-15}
    \end{equation}
    Using the Robin transformation law from Section~\ref{sec4}, we may also write
    \begin{equation}
    	\mu_{\varepsilon,Y,a}^{\rm app}=\frac{\Gamma}{2\pi}\log\frac1\varepsilon-\Gamma R_{\Omega_Y}(p_{Y,a})+\Psi_{\Omega_Y,W}^0(p_{Y,a}).
    	\label{5-16}
    \end{equation}
    The flux constant is thereby expressed through the core scale, the Robin function, and the background flow.
    
    The projected source tends to a point vortex away from the core, while \eqref{5-14} recovers $U$ after inner rescaling. We solve for corrections to this approximate configuration. The remaining vertical translation condition will determine the core depth through a finite-dimensional equation approaching the regular-drift balance of the fixed point-vortex wave.
    
    \smallskip
    
    \subsection{Linear theory around the concentrated core}
    We linearize the inner equation after eliminating the flux constant through the circulation constraint. Assumption~\ref{ass5-1} leaves only translational zero modes, of which reflection symmetry excludes the horizontal one. Throughout this subsection, $W>0$ is fixed, and all estimates are uniform as $\varepsilon\to0^+$ and $(Y,a)$ varies in a sufficiently small neighborhood of $(Y_W,a_W)$.
    
    Fix $\nu\in(0,1)$. For $m\in\{0,1,2\}$, define
    \[
    \|u\|_{C^{m,\alpha}_\nu}:=\|u\|_{C^{m,\alpha}(B_2)}+\sup_{R\ge1}R^\nu\|u(R\,\cdot)\|_{C^{m,\alpha}(B_2\setminus B_1)}.
    \]
    Let $C^{m,\alpha}_{\nu,\mathrm e}(\mathbb R^2)$ be the subspace of functions even in $X_1$, and set
    \[
    \mathcal X_*:=C^{2,\alpha}_{\nu,\mathrm e}(\mathbb R^2),\qquad \mathcal Y_*^0:=\left\{h\in C^{0,\alpha}_{\nu+2,\mathrm e}(\mathbb R^2):\int_{\mathbb R^2}h\,dX=0\right\}.
    \]
    With $0<\nu<1$, the translation modes $O(|X|^{-1})$ belong to the space, but constants and logarithmically growing functions do not.
    
    Set $I_0:=\int_{\mathbb R^2}f'(U)\,dX\ne0$ and define
    \[
    m_0[\phi]:=I_0^{-1}\int_{\mathbb R^2}f'(U)\phi\,dX.
    \]
    The mass-constrained linearization is
    \begin{equation}
    	\mathscr L_0\phi:=-\Delta\phi-f'(U)\bigl(\phi-m_0[\phi]\bigr).
    	\label{5-17}
    \end{equation}
    Since $\int_{\mathbb R^2}\Delta\phi\,dX=0$ for $\phi\in\mathcal X_*$ and $\int f'(U)(\phi-m_0[\phi])\,dX=0$, the operator $\mathscr L_0$ maps $\mathcal X_*$ into $\mathcal Y_*^0$.
    
    For $h\in\mathcal Y_*^0$, consider the logarithmic potential
    \[
    u(X):=-\frac1{2\pi}\int_{\mathbb R^2}\log|X-T|\,h(T)\,dT
    \]
    solves $-\Delta u=h$. Since $h$ has zero integral, the leading logarithmic term cancels; splitting the integral into $|T|\le |X|/2$ and its complement gives $|u(X)|\le C(1+|X|)^{-\nu}\|h\|_{\mathcal Y_*^0}$. Rescaled interior Schauder estimates on unit annuli then yield
    \[
    \|u\|_{\mathcal X_*}\le C\|h\|_{\mathcal Y_*^0}.
    \]
    The potential is even in $X_1$, and a decaying entire harmonic function must vanish. Thus $-\Delta:\mathcal X_*\to\mathcal Y_*^0$ is an isomorphism. Since $f'(U)$ is supported in $\overline{B_1}$, the difference $\mathscr L_0+\Delta$ is compact from $\mathcal X_*$ to $\mathcal Y_*^0$, and $\mathscr L_0$ is Fredholm of index zero.
    
    Let $Z_0:=\partial_2U$. Radial symmetry gives $m_0[Z_0]=0$, and differentiation of \eqref{5-7} yields $\mathscr L_0Z_0=0$. Conversely, if $\mathscr L_0\phi=0$, then $L_0(\phi-m_0[\phi])=0$, where $L_0=-\Delta-f'(U)$. By Assumption~\ref{ass5-1}(v), the bounded function $\phi-m_0[\phi]$ is a combination of $\partial_1U$ and $\partial_2U$. Since $\phi$ and both translations tend to zero at infinity, $m_0[\phi]=0$; evenness in $X_1$ then excludes $\partial_1U$. Hence
    \[
    \ker\mathscr L_0=\operatorname{span}\{Z_0\}.
    \]
    
    Choose $\eta_0\in C_c^\infty(\mathbb R^2)$ even in $X_1$ and odd in $X_2$, with $\int_{\mathbb R^2}\eta_0Z_0\,dX=1$, and define
    \[
    \Pi_0\phi:=\left(\int_{\mathbb R^2}\eta_0\phi\,dX\right)Z_0,\qquad \ell_0(h):=\int_{\mathbb R^2}hZ_0\,dX,\qquad Q_0h:=h-\ell_0(h)\eta_0.
    \]
    Because $\int\eta_0=0$, $Q_0$ preserves the zero-mass space. Moreover, $\int f'(U)Z_0\,dX=0$ and $L_0Z_0=0$, so integration by parts gives $\ell_0(\mathscr L_0\phi)=0$. As $\mathscr L_0$ has index zero and a one-dimensional kernel, its range is exactly $\ker\ell_0$. The restriction to $\ker\Pi_0$ therefore gives
    \begin{equation}
    	Q_0\mathscr L_0:\{\phi\in\mathcal X_*:\Pi_0\phi=0\}\longrightarrow\{h\in\mathcal Y_*^0:\ell_0(h)=0\}\quad\text{is an isomorphism}.
    	\label{5-18}
    \end{equation}
    
    The Dirichlet boundary moves far from the core on the inner scale. Under \eqref{5-11}, the reference half-plane becomes
    \[
    \mathcal H_{\varepsilon,Y,a}:=\{X\in\mathbb R^2:X_2<H_{\varepsilon,Y,a}\},\qquad H_{\varepsilon,Y,a}:=\frac{a\alpha(Y,a)}{\varepsilon}.
    \]
    Here $H_{\varepsilon,Y,a}\to\infty$ uniformly for $(Y,a)$ near $(Y_W,a_W)$. For $\phi\in\mathcal X_*$, set
    \[
    \mathfrak b_{\varepsilon,Y,a}[\phi](X_1,X_2):=-e^{-(H_{\varepsilon,Y,a}-X_2)|D_1|}\phi(\,\cdot\,,H_{\varepsilon,Y,a})(X_1),\qquad \mathfrak P_{\varepsilon,Y,a}\phi:=\phi+\mathfrak b_{\varepsilon,Y,a}[\phi].
    \]
    The function $\mathfrak b_{\varepsilon,Y,a}[\phi]$ is harmonic in $\mathcal H_{\varepsilon,Y,a}$ and cancels the boundary trace of $\phi$. Consequently, $\mathfrak P_{\varepsilon,Y,a}\phi=0$ on $X_2=H_{\varepsilon,Y,a}$ and $-\Delta\mathfrak P_{\varepsilon,Y,a}\phi=-\Delta\phi$ in the inner half-plane. For every fixed $R>0$, the weighted decay and Poisson estimates imply
    \[
    \|\mathfrak b_{\varepsilon,Y,a}[\phi]\|_{C^{2,\alpha}(B_R)}\le C_RH_{\varepsilon,Y,a}^{-\nu}\|\phi\|_{\mathcal X_*}=o(1)\|\phi\|_{\mathcal X_*}.
    \]
    For a sharper linear estimate, we separate the constant part of the harmonic correction. Set $c_{\varepsilon,Y,a}[\phi]:=\mathfrak b_{\varepsilon,Y,a}[\phi](0)$. The boundary trace at height $H:=H_{\varepsilon,Y,a}$ is $O(H^{-\nu})$ in the weighted norm. Interior estimates for its harmonic Poisson extension, on balls of radius comparable to $H$, therefore give
    \[
    \|\mathfrak b_{\varepsilon,Y,a}[\phi]-c_{\varepsilon,Y,a}[\phi]\|_{C^{2,\alpha}(B_R)}\le C_RH^{-1-\nu}\|\phi\|_{\mathcal X_*}.
    \]
    The estimate also holds after differentiating once in $(Y,a)$ at fixed inner coordinates. Indeed, differentiating the boundary height contributes $D_{(Y,a)}H=O(H)$, while the weighted bound on $\nabla\phi$ at height $H$ supplies the compensating factor $H^{-1}$. The interior gradient estimate then gives
    \[
    \bigl\|D_{(Y,a)}(\mathfrak b_{\varepsilon,Y,a}[\phi]-c_{\varepsilon,Y,a}[\phi])\bigr\|_{\mathcal L(H_{\mathrm e}^3\times\mathbb R,C^{2,\alpha}(B_R))}\le C_RH^{-1-\nu}\|\phi\|_{\mathcal X_*}.
    \]
    In particular, $H^{-1-\nu}=O(\varepsilon^{1+\nu})$ uniformly on the parameter neighborhood. It is the nonconstant part that enters the refined linear estimate, and this is smaller than $O(\varepsilon)$.
    
    To control the nonlinear source globally, let
    \[
    V_{\varepsilon,Y,a}(X):=\Phi_{\varepsilon,Y,a}^{\rm app}\left(z_a+\frac{\varepsilon}{\alpha(Y,a)}X\right)-\mu_{\varepsilon,Y,a}^{\rm app}.
    \]
    Since $f(U)$ is radial, the mean-value property gives $P_{\varepsilon,Y,a}(z)=\Gamma G_{\Pi_-}(z,z_a)$ outside the approximate core. Hence $V_{\varepsilon,Y,a}$ is harmonic in $\mathcal H_{\varepsilon,Y,a}\setminus\overline{B_R}$ for any fixed $R>1$. The inner expansion \eqref{5-14} and \eqref{5-9} give
    \[
    V_{\varepsilon,Y,a}=-\frac{\Gamma}{2\pi}\log R+O_R(\varepsilon)<0\qquad\text{on }\partial B_R.
    \]
    On the distant boundary $V_{\varepsilon,Y,a}=-\mu_{\varepsilon,Y,a}^{\rm app}<0$ for small $\varepsilon$, and its limit superior at spatial infinity is no larger than $-\mu_{\varepsilon,Y,a}^{\rm app}$. The maximum principle gives
    \[
    V_{\varepsilon,Y,a}\le-c_R<0\qquad\text{in }\mathcal H_{\varepsilon,Y,a}\setminus B_R,
    \]
    uniformly for $(Y,a)$ in a small fixed neighborhood. The Poisson maximum principle also yields $\|\mathfrak P_{\varepsilon,Y,a}\phi\|_{L^\infty(\mathcal H_{\varepsilon,Y,a})}\le C\|\phi\|_{\mathcal X_*}$, so this strict negativity persists under sufficiently small inner corrections and flux-constant variations. In particular, all nonlinear sources encountered below are supported in the same fixed ball $B_R$.
    
    We determine the exact reference flux constant before linearizing. Write $\alpha:=\alpha(Y,a)$ and $z(X):=z_a+\varepsilon X/\alpha$, and define the mass residual
    \[
    \mathfrak M_{\varepsilon,Y,a}(s):=\int_{\mathcal H_{\varepsilon,Y,a}}\frac{J(Y)(z(X))}{\alpha^2}f\bigl(V_{\varepsilon,Y,a}(X)-s\bigr)\,dX-\Gamma.
    \]
    On $B_R$, Taylor expansion at fixed inner coordinates gives, with $d_{a,W}:=W-\Gamma/(4\pi a)$ and $b_{Y,a}:=-\mathrm{i}Z_Y''(z_a)/\alpha^2\in\mathbb R$,
    \[
    V_{\varepsilon,Y,a}(X)=U(X)+\varepsilon\frac{d_{a,W}}{\alpha}X_2+O(\varepsilon^2),\qquad \frac{J(Y)(z(X))}{\alpha^2}=1-2\varepsilon b_{Y,a}X_2+O(\varepsilon^2).
    \]
    These expansions are uniform in $(Y,a)$, including one parameter derivative. Since $U$ is radial, the integrals of $X_2f'(U)$ and $X_2f(U)$ vanish. The fixed-support property therefore gives
    \[
    \mathfrak M_{\varepsilon,Y,a}(0)=O(\varepsilon^2),\qquad \partial_s\mathfrak M_{\varepsilon,Y,a}(0)=-I_0+O(\varepsilon).
    \]
    By Assumption~\ref{ass5-1}(vi), there is a unique small $s_{\varepsilon,Y,a}$ with $\mathfrak M_{\varepsilon,Y,a}(s_{\varepsilon,Y,a})=0$, and
    \[
    |s_{\varepsilon,Y,a}|\le C\varepsilon^2,\qquad \mu^0_{\varepsilon,Y,a}:=\mu_{\varepsilon,Y,a}^{\rm app}+s_{\varepsilon,Y,a}.
    \]
    Thus $\mu^0_{\varepsilon,Y,a}$ enforces the exact circulation at zero inner correction. It is the value $\mu_\varepsilon(Y,a,0)$ of the flux constant introduced in Section~5.3.
    
    Set $r_{\varepsilon,Y,a}(X):=V_{\varepsilon,Y,a}(X)-U(X)$ and define the exact linearization coefficient
    \[
    A_{\varepsilon,Y,a}(X):=\frac{J(Y)(z(X))}{\alpha^2}f'\bigl(U(X)+r_{\varepsilon,Y,a}(X)-s_{\varepsilon,Y,a}\bigr),
    \]
    extended by zero outside $\mathcal H_{\varepsilon,Y,a}$. The negativity estimate above and $s_{\varepsilon,Y,a}=O(\varepsilon^2)$ give $\operatorname{supp}A_{\varepsilon,Y,a}\subset\overline{B_R}$. The local metric and stream-function expansions, together with $s_{\varepsilon,Y,a}=O(\varepsilon^2)$, give the quantitative estimate
    \[
    \|A_{\varepsilon,Y,a}-f'(U)\|_{C^{0,\alpha}(B_R)}+\|D_{(Y,a)}A_{\varepsilon,Y,a}\|_{\mathcal L(H_{\mathrm e}^3\times\mathbb R,C^{0,\alpha}(B_R))}\le C\varepsilon.
    \]
    Derivatives are taken at fixed inner coordinates; $f'(U)$ is independent of $(Y,a)$. The derivative of the exact mass constraint at zero correction determines the flux-constant variation
    \[
    m_{\varepsilon,Y,a}[\phi]:=\frac{\displaystyle\int_{\mathbb R^2}A_{\varepsilon,Y,a}\mathfrak P_{\varepsilon,Y,a}\phi\,dX}{\displaystyle\int_{\mathbb R^2}A_{\varepsilon,Y,a}\,dX}.
    \]
    The denominator tends to $I_0\ne0$. To use the preceding sharpened Poisson estimate, introduce
    \[
    \overline m_{\varepsilon,Y,a}[\phi]:=\frac{\int_{B_R}A_{\varepsilon,Y,a}\phi\,dX}{\int_{B_R}A_{\varepsilon,Y,a}\,dX}.
    \]
    Writing $\mathfrak b_{\varepsilon,Y,a}[\phi]=c_{\varepsilon,Y,a}[\phi]+\widetilde{\mathfrak b}_{\varepsilon,Y,a}[\phi]$ and using the local estimate for $\widetilde{\mathfrak b}$ gives
    \[
    m_{\varepsilon,Y,a}[\phi]=\overline m_{\varepsilon,Y,a}[\phi]+c_{\varepsilon,Y,a}[\phi]+O(H^{-1-\nu})\|\phi\|_{\mathcal X_*}.
    \]
    Consequently, the constant harmonic term cancels from the mass-constrained combination:
    \[
    \bigl\|\mathfrak P_{\varepsilon,Y,a}\phi-m_{\varepsilon,Y,a}[\phi]-(\phi-\overline m_{\varepsilon,Y,a}[\phi])\bigr\|_{C^{0,\alpha}(B_R)}\le CH^{-1-\nu}\|\phi\|_{\mathcal X_*}.
    \]
    Both estimates persist under one parameter derivative. Since $\overline m_{\varepsilon,Y,a}-m_0=O(\varepsilon)$ in operator norm, the resulting exact mass-constrained linearization
    \[
    \mathscr L_{\varepsilon,Y,a}\phi:=-\Delta\phi-A_{\varepsilon,Y,a}\bigl(\mathfrak P_{\varepsilon,Y,a}\phi-m_{\varepsilon,Y,a}[\phi]\bigr)
    \]
    is the derivative at zero correction of the whole-plane inner residual after eliminating the flux constant. By construction $\int\mathscr L_{\varepsilon,Y,a}\phi\,dX=0$, and the preceding estimates yield
    \[
    \|\mathscr L_{\varepsilon,Y,a}-\mathscr L_0\|_{\mathcal L(\mathcal X_*,\mathcal Y_*^0)}+\|D_{(Y,a)}\mathscr L_{\varepsilon,Y,a}\|_{\mathcal L(H_{\mathrm e}^3\times\mathbb R,\mathcal L(\mathcal X_*,\mathcal Y_*^0))}\le C(\varepsilon+H^{-1-\nu})\le C\varepsilon
    \]
    uniformly on the parameter neighborhood. The improvement over the $O(H^{-\nu})$ bound for the harmonic correction comes from the circulation constraint.
    
    \begin{lemma}
    	\label{lem5-2}
    	There exist $\varepsilon_0>0$, a neighborhood $\mathcal U$ of $(Y_W,a_W)$, and $C>0$ such that, for every $0<\varepsilon<\varepsilon_0$ and $(Y,a)\in\mathcal U$, the operator
    	\begin{equation}
    		Q_0\mathscr L_{\varepsilon,Y,a}:\{\phi\in\mathcal X_*:\Pi_0\phi=0\}\longrightarrow\{h\in\mathcal Y_*^0:\ell_0(h)=0\}
    		\label{5-19}
    	\end{equation}
    	is an isomorphism, with
    	\[
    	\|\phi\|_{\mathcal X_*}\le C\|Q_0\mathscr L_{\varepsilon,Y,a}\phi\|_{\mathcal Y_*^0}\qquad\text{whenever }\Pi_0\phi=0.
    	\]
    \end{lemma}
    
    \begin{proof}
    	Set $\mathcal X_*^\perp:=\ker\Pi_0$ and $\mathcal Y_*^\perp:=\ker\ell_0\cap\mathcal Y_*^0$. By \eqref{5-18}, $Q_0\mathscr L_0:\mathcal X_*^\perp\to\mathcal Y_*^\perp$ has a bounded inverse. The operator convergence just proved gives
    	\[
    	\|Q_0(\mathscr L_{\varepsilon,Y,a}-\mathscr L_0)\|_{\mathcal L(\mathcal X_*^\perp,\mathcal Y_*^\perp)}=o(1)
    	\]
    	uniformly on $\mathcal U$. Decreasing $\varepsilon_0$ if necessary, we have
    	\[
    	\left\|(Q_0\mathscr L_0)^{-1}Q_0(\mathscr L_{\varepsilon,Y,a}-\mathscr L_0)\right\|_{\mathcal L(\mathcal X_*^\perp)}<\frac12.
    	\]
    	The Neumann series proves the isomorphism and the uniform inverse bound.
    \end{proof}
    
    The physical Dirichlet correction is $\mathfrak P_{\varepsilon,Y,a}\phi$. Lemma~\ref{lem5-2} solves the inner equation modulo vertical translation, leaving one scalar condition to select the vortex depth. We couple this condition to the free-surface equation in the next subsection.
    
    \smallskip
    
    \subsection{Joint Lyapunov--Schmidt reduction}
    We apply a joint Lyapunov--Schmidt reduction near the point-vortex wave $(Y_W,a_W)$. For each $a$ near $a_W$, we solve the Bernoulli equation and the inner equation orthogonal to vertical translation; the remaining scalar condition selects the depth. Throughout this subsection, $W>0$ is fixed and all parameter derivatives are taken at fixed inner coordinates.
    
    Let $I$ be a sufficiently small closed interval with $a_W$ in its interior. The point-vortex Bernoulli equation $\mathcal B(Y,a,W)=0$ determines a smooth curve $Y=Y_0(a)$ for $a\in I$, with $Y_0(a_W)=Y_W$. Define
    \begin{equation}
    	g_0(a):=\mathcal P(Y_0(a),a,W).
    	\label{5-20}
    \end{equation}
    We have $g_0(a_W)=0$, and this zero is nondegenerate. Indeed, let $U_W=(u_W,\lambda_W)$ be the rescaled solution of Proposition~\ref{prop3-9}. By Lemmas~\ref{lem3-3}, \ref{lem3-5}, and \ref{lem3-8}, together with \eqref{3-9}, the Bernoulli block $D_u\mathscr B_{\delta_W,\beta_W}(U_W)$ has a uniformly bounded inverse, while
    \[
    \|D_u\mathscr P_{\delta_W}(U_W)\|=O(\delta_W),\qquad D_\lambda\mathscr P_{\delta_W}(U_W)=1+O(\delta_W),\qquad \|D_\lambda\mathscr B_{\delta_W,\beta_W}(U_W)\|\le C.
    \]
    Its scalar Schur complement is therefore $1+O(\delta_W)$. Returning through the scaling \eqref{3-1} gives
    \begin{equation}
    	g_0'(a_W)=\partial_a\mathcal P-D_Y\mathcal P\,(D_Y\mathcal B)^{-1}\partial_a\mathcal B=\frac{W}{a_0(W)}\bigl(1+O(\delta_W)\bigr)\ne0,
    	\label{5-21}
    \end{equation}
    where the derivatives on the left are evaluated at $(Y_W,a_W,W)$. The small-speed threshold is taken within the range for which these inverse estimates hold.
    
    The exact circulation constraint determines the flux constant. For a whole-plane correction $\varphi$, let $\phi=\phi_{\varepsilon,Y,a}[\varphi]$ be its Dirichlet projection, as defined below, and set
    \[
    \mathcal M_\varepsilon(Y,a,\phi,\mu):=\frac1{\varepsilon^2}\int_{\Pi_-}J(Y)f\bigl(P_{\varepsilon,Y,a}+\phi+Wy-\mu\bigr)\,dz-\Gamma.
    \]
    Section~5.2 constructs the exact zero-correction level $\mu^0_{\varepsilon,Y,a}=\mu^{\rm app}_{\varepsilon,Y,a}+s_{\varepsilon,Y,a}$, with $s_{\varepsilon,Y,a}=O(\varepsilon^2)$ and $\mathcal M_\varepsilon(Y,a,0,\mu^0_{\varepsilon,Y,a})=0$. At this level,
    \[
    \partial_\mu\mathcal M_\varepsilon(Y,a,0,\mu^0_{\varepsilon,Y,a})=-\int_{\mathbb R^2}A_{\varepsilon,Y,a}\,dX\longrightarrow-I_0\ne0.
    \]
    The implicit function theorem thus determines a unique $C^1$ flux constant $\mu_\varepsilon(Y,a,\phi)$ for small inner corrections. Its derivative at $\varphi=0$ is precisely $m_{\varepsilon,Y,a}$ from Section~5.2. Taylor's formula, uniformly with respect to $(Y,a)$ in a fixed neighborhood, gives
    \begin{equation}
    	\mu_\varepsilon\bigl(Y,a,\phi_{\varepsilon,Y,a}[\varphi]\bigr)=\mu^0_{\varepsilon,Y,a}+m_{\varepsilon,Y,a}[\varphi]+O\bigl(\|\varphi\|_{\mathcal X_*}^2\bigr).
    	\label{5-22}
    \end{equation}
    Here and below, the remainder bounds also hold after one derivative along $C^1$ parameter curves whose corrections and derivatives are uniformly small. In particular, $\mu_\varepsilon=\mu^{\rm app}_{\varepsilon,Y,a}+o(1)$ along the branch constructed below.
    
    With this choice of $\mu$, the elliptic residual in conformal coordinates is
    \begin{equation}
    	\mathcal E_\varepsilon(Y,a,\phi):=-\Delta\phi-\frac1{\varepsilon^2}\left[J(Y)f\bigl(P_{\varepsilon,Y,a}+\phi+Wy-\mu_\varepsilon(Y,a,\phi)\bigr)-\alpha(Y,a)^2f(U_{\varepsilon,Y,a})\right].
    	\label{5-23}
    \end{equation}
    We impose $\phi=0$ on $y=0$. For $X=\alpha(Y,a)(z-z_a)/\varepsilon$, write $\widetilde\phi(X):=\phi(z_a+\varepsilon X/\alpha(Y,a))$ and define
    \begin{equation}
    	\widetilde{\mathcal E}_\varepsilon(Y,a,\widetilde\phi)(X):=\frac{\varepsilon^2}{\alpha(Y,a)^2}\mathcal E_\varepsilon(Y,a,\phi)\left(z_a+\frac{\varepsilon}{\alpha(Y,a)}X\right).
    	\label{5-24}
    \end{equation}
    The outer Bernoulli residual is
    \begin{equation}
    	\mathcal B_\varepsilon(Y,a,\phi):=gY+\sigma\kappa(Y)+\frac{|\nabla(P_{\varepsilon,Y,a}+\phi+Wy)|^2}{2J(Y)}-\frac{W^2}{2}.
    	\label{5-25}
    \end{equation}
    The radial source in \eqref{5-12} has total mass $\Gamma$; the mean-value property therefore gives $P_{\varepsilon,Y,a}(z)=\Gamma G_{\Pi_-}(z,z_a)$ outside its core. In particular, the identity holds in a fixed neighborhood of the free surface, and hence $\mathcal B_\varepsilon(Y,a,0)=\mathcal B(Y,a,W)$ there, including all derivatives in $(Y,a)$.
    
    For the Dirichlet correction, we use
    \[
    \phi_{\varepsilon,Y,a}[\varphi](z):=\bigl(\mathfrak P_{\varepsilon,Y,a}\varphi\bigr)\left(\frac{\alpha(Y,a)}{\varepsilon}(z-z_a)\right).
    \]
    Choose $R_*>1$ large enough to contain the supports of all nonlinear inner sources and of $\eta_0$ for the small corrections considered below. Introduce the closed subspaces
    \[
    \begin{aligned}
    	\mathcal X_{*,R_*}^{\perp}&:=\{\varphi\in\mathcal X_*:\Pi_0\varphi=0,\ \Delta\varphi=0\text{ on }\mathbb R^2\setminus\overline{B_{R_*}}\},\\
    	\mathcal Y_{*,R_*}^{\perp}&:=\{h\in\mathcal Y_*^0:\ell_0(h)=0,\ \operatorname{supp}h\subset\overline{B_{R_*}}\}.
    \end{aligned}
    \]
    Let $\mathfrak E_\varepsilon(Y,a,\varphi)$ denote the whole-plane residual associated with \eqref{5-24}: substitute $\widetilde\phi=\mathfrak P_{\varepsilon,Y,a}\varphi$, extend the compact nonlinear source by zero outside $\mathcal H_{\varepsilon,Y,a}$, and use $-\Delta\varphi$ as the Laplacian term. The exact circulation constraint gives $\int_{\mathbb R^2}\mathfrak E_\varepsilon\,dX=0$. If $\varphi\in\mathcal X_{*,R_*}^{\perp}$, both $\mathfrak E_\varepsilon$ and $Q_0\mathfrak E_\varepsilon$ are supported in $\overline{B_{R_*}}$. Conversely, the inverse in \eqref{5-18} of a source in $\mathcal Y_{*,R_*}^{\perp}$ is harmonic outside that ball, because both the coefficient $f'(U)$ and $\eta_0$ are supported there. Thus \eqref{5-18} and Lemma~\ref{lem5-2} restrict to these subspaces with the same uniform bounds.
    
    \begin{proposition}
    	\label{prop5-3}
    	After decreasing $I$ if necessary, there exist $\varepsilon_0>0$ and $C>0$ such that, for every $0<\varepsilon<\varepsilon_0$ and $a\in I$, the system
    	\begin{equation}
    		\mathcal B_\varepsilon\bigl(Y,a,\phi_{\varepsilon,Y,a}[\varphi]\bigr)=0,\qquad Q_0\mathfrak E_\varepsilon(Y,a,\varphi)=0,\qquad \varphi\in\mathcal X_{*,R_*}^{\perp}
    		\label{5-26}
    	\end{equation}
    	has a unique solution $(Y_{\varepsilon,a},\varphi_{\varepsilon,a})$ near $(Y_0(a),0)$. It depends $C^1$ on $a$ and satisfies
    	\begin{equation}
    		\|Y_{\varepsilon,\cdot}-Y_0\|_{C^1(I;H_{\mathrm e}^3)}+\|\varphi_{\varepsilon,\cdot}\|_{C^1(I;\mathcal X_*)}\le C\varepsilon.
    		\label{5-27}
    	\end{equation}
    	The speed $W$ is fixed and the constant is independent of $\varepsilon$.
    \end{proposition}
    
    \begin{proof}
    	We use $H_{\mathrm e}^3$ for the surface variable and $H_{\mathrm e}^1$ for the Bernoulli residual. Set $\mathbb X:=H_{\mathrm e}^3\times\mathcal X_{*,R_*}^{\perp}$ and $\mathbb Y:=H_{\mathrm e}^1\times\mathcal Y_{*,R_*}^{\perp}$. For $V=(\zeta,\varphi)\in\mathbb X$, with $Y=Y_0(a)+\zeta$, denote the left-hand side of \eqref{5-26} by $\mathcal H_\varepsilon(a,V)\in\mathbb Y$. The estimates below are uniform for $a\in I$, and parameter derivatives are evaluated at fixed $X$.
    	
    	We establish the estimates needed for a uniform implicit-function argument. On the fixed inner ball, the expansion \eqref{5-14}, its derivative in $(Y,a)$, and the corresponding expansion of $J(Y)/\alpha(Y,a)^2$ imply
    	\[
    	\left\|V_{\varepsilon,Y,a}-U\right\|_{C^{1}(I;C^{2,\alpha}(B_{R_*}))}+\left\|\frac{J(Y)(z_a+\varepsilon X/\alpha)}{\alpha^2}-1\right\|_{C^{1}(I;C^{1,\alpha}(B_{R_*}))}\le C\varepsilon
    	\]
    	when evaluated along $Y=Y_0(a)$; the same estimates hold for one variation in $Y$ in a fixed neighborhood. By the exact mass normalization in Section~5.2, $s_{\varepsilon,Y,a}$ and its $(Y,a)$ derivatives are $O(\varepsilon^2)$. The fixed support and the $C^{2,\alpha}_{\rm loc}$ regularity of $f$ consequently give
    	\[
    	\|\mathfrak E_\varepsilon(Y_0(a),a,0)\|_{\mathcal Y_*^0}+\left\|\frac{d}{da}\mathfrak E_\varepsilon(Y_0(a),a,0)\right\|_{\mathcal Y_*^0}\le C\varepsilon.
    	\]
    	The outer residual at $\varphi=0$ vanishes exactly by the mean-value identity, so the first of the three estimates needed below follows:
    	\[
    	\sup_{a\in I}\bigl(\|\mathcal H_\varepsilon(a,0)\|_{\mathbb Y}+\|\partial_a\mathcal H_\varepsilon(a,0)\|_{\mathbb Y}\bigr)\le C\varepsilon.
    	\]
    	
    	To estimate the coupling from the inner correction to the boundary, let $h:=-\Delta\varphi$. For $\varphi\in\mathcal X_{*,R_*}^{\perp}$, $h$ is supported in $\overline{B_{R_*}}$ and $\int h=0$. Uniqueness of the Dirichlet problem gives
    	\[
    	\phi_{\varepsilon,Y,a}[\varphi](z)=\int_{B_{R_*}}G_{\Pi_-}\left(z,z_a+\frac{\varepsilon}{\alpha(Y,a)}X\right)h(X)\,dX.
    	\]
    	The zero-mass identity allows us to subtract the kernel at $z_a$ before differentiating at the surface. Smoothness of the Green kernel away from the vortex yields, also after one derivative in $(Y,a)$,
    	\[
    	\left\|\partial_y\phi_{\varepsilon,Y,a}[\varphi](\cdot,0)\right\|_{H^1}\le C\varepsilon\|h\|_{C^{0,\alpha}(B_{R_*})}\le C\varepsilon\|\varphi\|_{\mathcal X_*}.
    	\]
    	Thus the inner-to-outer block is $O(\varepsilon)$. For the reverse block, the local inner expansions at fixed $X$ give $\|D_Y\mathfrak E_\varepsilon(Y_0(a),a,0)\|\le C\varepsilon$. For the inner linearization, we use the improved mass-constrained Poisson estimate from Section~5.2: the harmonic boundary correction contains a constant term of size $O(H^{-\nu})$, but the same constant occurs in the variation of the flux level and cancels from $\mathfrak P_{\varepsilon,Y,a}\varphi-m_{\varepsilon,Y,a}[\varphi]$. In the remaining expression, the harmonic error is $O(H^{-1-\nu})=O(\varepsilon^{1+\nu})$; moreover $A_{\varepsilon,Y,a}-f'(U)=O(\varepsilon)$ on the fixed core ball. Thus
    	\[
    	\begin{aligned}
    		\|D_Y\mathfrak E_\varepsilon(Y_0(a),a,0)\|_{\mathcal L(H_{\mathrm e}^3,\mathcal Y_*^0)}&\le C\varepsilon,\\
    		\|D_\varphi\mathfrak E_\varepsilon(Y_0(a),a,0)-\mathscr L_0\|_{\mathcal L(\mathcal X_*,\mathcal Y_*^0)}&\le C(\varepsilon+H^{-1-\nu})\le C\varepsilon.
    	\end{aligned}
    	\]
    	The same bounds hold after one $a$ derivative: the parameter derivative of the nonconstant Poisson correction is $O(H^{-1-\nu})$, while the exact mass coefficient and its parameter derivative differ from their whole-plane limits by $O(\varepsilon)$. Restriction to the compact-source subspaces and application of $Q_0$ preserve these bounds.
    	
    	Define
    	\[
    	\mathbb A(a):=\begin{pmatrix}D_Y\mathcal B(Y_0(a),a,W)&0\\0&Q_0\mathscr L_0\end{pmatrix}:\mathbb X\longrightarrow\mathbb Y.
    	\]
    	The point-vortex Bernoulli block is invertible for $a\in I$ by the estimates of Section~3, and the second block is invertible by \eqref{5-18}. Shrinking $I$ if needed, we have $\sup_{a\in I}\|\mathbb A(a)^{-1}\|\le M$. The preceding estimates yield the second required bound,
    	\[
    	\sup_{a\in I}\|D_V\mathcal H_\varepsilon(a,0)-\mathbb A(a)\|_{\mathcal L(\mathbb X,\mathbb Y)}\le C\varepsilon.
    	\]
    	The fixed support, lower bounds for the mass derivative and conformal Jacobian, and product estimates give, for $V,V'$ in a fixed sufficiently small ball of $\mathbb X$,
    	\[
    	\begin{aligned}
    		\|D_V\mathcal H_\varepsilon(a,V)-D_V\mathcal H_\varepsilon(a,V')\|_{\mathcal L(\mathbb X,\mathbb Y)}&\le C\|V-V'\|_{\mathbb X},\\
    		\|\partial_a\mathcal H_\varepsilon(a,V)-\partial_a\mathcal H_\varepsilon(a,V')\|_{\mathbb Y}&\le C\|V-V'\|_{\mathbb X}.
    	\end{aligned}
    	\]
    	The constants are independent of $\varepsilon$ and $a$. For the inner equation these bounds follow by differentiating the exact mass identity and using the nonvanishing denominator $\int A_{\varepsilon,Y,a}$; the quadratic Taylor remainder for $f\in C_{\mathrm{loc}}^{2,\alpha}$ is controlled in $C^{0,\alpha}(B_{R_*})$. The parameter derivative acts at fixed inner coordinates and all nonlinear sources remain supported in the same ball, so no moving-support boundary term occurs. For the outer equation, the corresponding bounds follow from the Bernoulli map $H_{\mathrm e}^3\to H_{\mathrm e}^1$ and the zero-mass Green-kernel estimate. These establish the required local Lipschitz controls without derivative loss.
    	
    	For sufficiently small $\varepsilon$, the inverse of $D_V\mathcal H_\varepsilon(a,0)$ is uniformly bounded. Applying its frozen Newton map on a ball $\|V\|_{\mathbb X}\le C_0\varepsilon$, the residual bound and the Lipschitz estimate make the map invariant and contractive after choosing $C_0$ sufficiently large and then $\varepsilon$ sufficiently small. Thus there is a unique $V_\varepsilon(a)$ with $\|V_\varepsilon(a)\|_{\mathbb X}\le C\varepsilon$. The parameter-dependent implicit function theorem gives $C^1$ dependence on $a$. Differentiating $\mathcal H_\varepsilon(a,V_\varepsilon(a))=0$ gives
    	\[
    	D_V\mathcal H_\varepsilon(a,V_\varepsilon(a))\,\partial_aV_\varepsilon(a)=-\partial_a\mathcal H_\varepsilon(a,V_\varepsilon(a)).
    	\]
    	The first and third estimates above imply $\|\partial_a\mathcal H_\varepsilon(a,V_\varepsilon(a))\|\le C\varepsilon$. Applying the uniform inverse yields $\|\partial_aV_\varepsilon(a)\|_{\mathbb X}\le C\varepsilon$, which proves \eqref{5-27}.
    \end{proof}
    
    For the solution of Proposition~\ref{prop5-3}, write
    \[
    \phi_{\varepsilon,a}:=\phi_{\varepsilon,Y_{\varepsilon,a},a}[\varphi_{\varepsilon,a}],\qquad \widetilde\phi_{\varepsilon,a}:=\mathfrak P_{\varepsilon,Y_{\varepsilon,a},a}\varphi_{\varepsilon,a}.
    \]
    The vertical translation component is the only unsolved part of the inner equation. We normalize it by
    \begin{equation}
    	\mathcal G_\varepsilon(a):=\frac{\alpha(Y_{\varepsilon,a},a)}{\Gamma\varepsilon}\,\ell_0\!\left(\mathfrak E_\varepsilon(Y_{\varepsilon,a},a,\varphi_{\varepsilon,a})\right).
    	\label{5-28}
    \end{equation}
    Since the complementary inner equation has been solved, the full inner equation holds if and only if $\mathcal G_\varepsilon(a)=0$. We identify the limiting scalar equation.
    
    \begin{proposition}
    	\label{prop5-4}
    	As $\varepsilon\to0^+$,
    	\begin{equation}
    		\mathcal G_\varepsilon(a)=g_0(a)+o(1)
    		\label{5-29}
    	\end{equation}
    	in $C^1(I)$. More precisely, the remainder is $O(\varepsilon)$ in the same norm.
    \end{proposition}
    
    \begin{proof}
    	Set $Y=Y_{\varepsilon,a}$, $\varphi=\varphi_{\varepsilon,a}$, $\alpha:=Z_Y'(-\mathrm{i}a)>0$, and
    	\[
    	d:=W-\frac{\Gamma}{4\pi a},\qquad b:=-\mathrm{i}\frac{Z_Y''(-\mathrm{i}a)}{\alpha^2}\in\mathbb R.
    	\]
    	The estimates below are uniform in $C^1(I)$ at fixed inner coordinates. On $B_{R_*}$, the expansions from Section~5.2 give
    	\[
    	\frac{J(Y)(z_a+\varepsilon X/\alpha)}{\alpha^2}=1-2\varepsilon bX_2+O_{C^1(I;C^{0,\alpha}(B_{R_*}))}(\varepsilon^2),
    	\]
    	and
    	\[
    	V_{\varepsilon,Y,a}(X)=U(X)+\frac{\varepsilon d}{\alpha}X_2+O_{C^1(I;C^{0,\alpha}(B_{R_*}))}(\varepsilon^2).
    	\]
    	By \eqref{5-27}, $\|\varphi_{\varepsilon,\cdot}\|_{C^1(I;\mathcal X_*)}\le C\varepsilon$. Since the correction is harmonic outside $B_{R_*}$ and has zero total source, its exterior decay and the Poisson estimate from the proof of Proposition~\ref{prop5-3} give
    	\[
    	\left\|\mathfrak b_{\varepsilon,Y_{\varepsilon,\cdot},\cdot}[\varphi_{\varepsilon,\cdot}]\right\|_{C^1(I;C^{2,\alpha}(B_{R_*}))}\le C\varepsilon^2.
    	\]
    	
    	Write $s_\varepsilon(a):=\mu_\varepsilon(Y_{\varepsilon,a},a,\phi_{\varepsilon,a})-\mu^{\rm app}_{\varepsilon,Y_{\varepsilon,a},a}$. Expanding the exact circulation constraint on the fixed ball and using the preceding formulas yields
    	\[
    	\begin{aligned}
    		0={}&\int_{\mathbb R^2}f'(U)\bigl(\varphi_{\varepsilon,a}-s_\varepsilon(a)\bigr)\,dX\\
    		&+\varepsilon\int_{\mathbb R^2}X_2\left(\frac{d}{\alpha}f'(U)-2bf(U)\right)\,dX+O_{C^1(I)}(\varepsilon^2).
    	\end{aligned}
    	\]
    	The explicit geometric terms in the second integral are odd in $X_2$; the unknown correction $\varphi_{\varepsilon,a}$ need not be odd. Radial symmetry of $U$ makes the second integral zero, while the first is canceled by the flux-constant variation. To check the derivative estimate, differentiate the mass identity at fixed $X$. The local expansions of the metric and approximate stream function have $C^1(I)$ remainders of size $O(\varepsilon^2)$, and $\varphi_{\varepsilon,\cdot}$ and $\partial_a\varphi_{\varepsilon,\cdot}$ are both $O(\varepsilon)$ by \eqref{5-27}. The smoothness of $f$ and the fixed support then give
    	\[
    	\begin{aligned}
    		\left\|s_\varepsilon-m_0[\varphi_{\varepsilon,\cdot}]\right\|_{C^1(I)}&\le C\left[\varepsilon^2+\varepsilon\|\varphi_{\varepsilon,\cdot}\|_{C^1(I;\mathcal X_*)}+\|\varphi_{\varepsilon,\cdot}\|_{C^1(I;\mathcal X_*)}^2\right]\\
    		&\le C\varepsilon^2.
    	\end{aligned}
    	\]
    	The reference shift $s_{\varepsilon,Y,a}$ is therefore $O_{C^1}(\varepsilon^2)$, and the exact projected mass derivative differs from $m_0$ by the $O(\varepsilon)$ coefficient error together with the harmonic constant already canceled in Section~5.2.
    	
    	Taylor expansion of the nonlinear source, using $f\in C_{\rm loc}^{2,\alpha}$ and the uniform support in $B_{R_*}$, now gives
    	\[
    	\mathfrak E_\varepsilon(Y,a,\varphi)=\mathscr L_0\varphi-\frac{\varepsilon d}{\alpha}X_2f'(U)+2\varepsilon bX_2f(U)+\mathcal R_{\varepsilon,a},
    	\]
    	with
    	\[
    	\|\mathcal R_{\varepsilon,\cdot}\|_{C^1(I;\mathcal Y_*^0)}\le C\varepsilon^2.
    	\]
    	More precisely, the remainder consists of the second-order metric and approximate stream-function errors, the difference between the exact mass shift and $m_0[\varphi]$, the nonconstant Dirichlet correction, and the quadratic Taylor remainder of $f$. On the fixed core ball, the $C^{0,\alpha}$ product and composition estimates give
    	\[
    	\begin{aligned}
    		\|\mathcal R_{\varepsilon,a}\|_{\mathcal Y_*^0}+\|\partial_a\mathcal R_{\varepsilon,a}\|_{\mathcal Y_*^0}
    		\le C\bigl[&\varepsilon^2+\varepsilon(\|\varphi_{\varepsilon,a}\|_{\mathcal X_*}+\|\partial_a\varphi_{\varepsilon,a}\|_{\mathcal X_*})\\
    		&+(\|\varphi_{\varepsilon,a}\|_{\mathcal X_*}+\|\partial_a\varphi_{\varepsilon,a}\|_{\mathcal X_*})^2
    		+\|s_\varepsilon-m_0[\varphi_{\varepsilon,\cdot}]\|_{C^1(I)}\bigr].
    	\end{aligned}
    	\]
    	The $a$ derivative of the harmonic correction is included here through the fixed-coordinate Green-kernel estimate, and that of the mass level through the differentiated exact circulation constraint. By \eqref{5-27} and the preceding mass estimate, the right-hand side is $O(\varepsilon^2)$. All remainders have support in $B_{R_*}$, so the local H\"older bounds give the weighted $\mathcal Y_*^0$ estimate; the remainder has zero mass because the exact residual and the displayed first-order terms do.
    	
    	Projecting onto $Z_0=\partial_2U$, The mass-constrained operator satisfies $\ell_0(\mathscr L_0\varphi)=0$ by Section~5.2. Integration by parts and \eqref{5-8} give
    	\[
    	\int_{\mathbb R^2}X_2f'(U)\partial_2U\,dX=-\int_{\mathbb R^2}f(U)\,dX=-\Gamma.
    	\]
    	For the other integral, put $F(t):=\int_0^t f(s)\,ds$. The radial equation for $U$, multiplied by $r^2U_r(r)$ and integrated over $(0,1)$, yields the Pohozaev identity
    	\[
    	\int_{\mathbb R^2}F(U)\,dX=\frac{\pi}{2}U_r(1)^2=\frac{\Gamma^2}{8\pi}.
    	\]
    	Consequently,
    	\[
    	\int_{\mathbb R^2}X_2f(U)\partial_2U\,dX=-\int_{\mathbb R^2}F(U)\,dX=-\frac{\Gamma^2}{8\pi}.
    	\]
    	The projected residual is therefore
    	\[
    	\ell_0\bigl(\mathfrak E_\varepsilon(Y,a,\varphi)\bigr)=\frac{\Gamma\varepsilon}{\alpha}d-\frac{\Gamma^2\varepsilon}{4\pi}b+O_{C^1(I)}(\varepsilon^2).
    	\]
    	By the definition of $b$,
    	\[
    	d-\frac{\Gamma\alpha}{4\pi}b=W-\frac{\Gamma}{4\pi a}+\frac{\mathrm{i}\Gamma}{4\pi}\frac{Z_Y''(-\mathrm{i}a)}{Z_Y'(-\mathrm{i}a)}=\mathcal P(Y,a,W).
    	\]
    	Using \eqref{5-28}, we conclude that
    	\[
    	\mathcal G_\varepsilon(a)=\mathcal P(Y_{\varepsilon,a},a,W)+O_{C^1(I)}(\varepsilon)=g_0(a)+O_{C^1(I)}(\varepsilon),
    	\]
    	where the final equality follows from \eqref{5-27} and smooth dependence of $\mathcal P$ on $(Y,a)$. This proves \eqref{5-29}.
    \end{proof}
    
    By \eqref{5-21}, $g_0(a_W)=0$ and $g_0'(a_W)\ne0$. By Proposition~\ref{prop5-4}, $\mathcal G_\varepsilon\to g_0$ in $C^1(I)$, so this simple zero persists uniquely in a sufficiently small neighborhood of $a_W$. Thus there is $a_\varepsilon\to a_W$ such that
    \begin{equation}
    	\mathcal G_\varepsilon(a_\varepsilon)=0.
    	\label{5-30}
    \end{equation}
    At this depth, both the complementary and translational inner equations hold, and the Bernoulli equation and exact circulation constraint have already been solved. The resulting configuration satisfies the full regularized system.
    
    \smallskip
    
    \subsection{Physical reconstruction and the two free boundaries}
    Let $a_\varepsilon$ be the zero of the reduced equation in \eqref{5-30}, and set
    \[
    Y_\varepsilon:=Y_{\varepsilon,a_\varepsilon},\qquad \varphi_\varepsilon:=\varphi_{\varepsilon,a_\varepsilon},\qquad \phi_\varepsilon:=\phi_{\varepsilon,a_\varepsilon},\qquad Z_\varepsilon:=Z_{Y_\varepsilon},\qquad \Omega_\varepsilon:=Z_\varepsilon(\Pi_-).
    \]
    By Proposition~\ref{prop5-3},
    \begin{equation}
    	a_\varepsilon\to a_W,\qquad Y_\varepsilon\to Y_W\ \text{in }H_{\mathrm e}^3,\qquad \|\varphi_\varepsilon\|_{\mathcal X_*}=O(\varepsilon).
    	\label{5-31}
    \end{equation}
    The inner Dirichlet correction $\widetilde\phi_\varepsilon:=\mathfrak P_{\varepsilon,Y_\varepsilon,a_\varepsilon}\varphi_\varepsilon$ consequently tends to zero locally in $C^{2,\alpha}$ on the inner scale. The speed $W$ remains fixed throughout this section.
    
    Define the conformal stream functions $\psi_\varepsilon:=P_{\varepsilon,Y_\varepsilon,a_\varepsilon}+\phi_\varepsilon$ and $\Phi_\varepsilon:=\psi_\varepsilon+Wy$. Let $\mu_\varepsilon$ be the flux constant determined by the circulation constraint. In physical coordinates, set
    \begin{equation}
    	\Psi_\varepsilon:=\Phi_\varepsilon\circ Z_\varepsilon^{-1},\qquad \omega_\varepsilon(p):=\frac1{\varepsilon^2}f\bigl(\Psi_\varepsilon(p)-\mu_\varepsilon\bigr).
    	\label{5-32}
    \end{equation}
    The relative and laboratory-frame velocities are $v_{{\rm rel},\varepsilon}=-\nabla^\perp\Psi_\varepsilon$ and $v_\varepsilon=v_{{\rm rel},\varepsilon}-We_1$, respectively. To recover the pressure, write $F(t):=\int_0^t f(s)\,ds$ and define
    \[
    P_\varepsilon(p):=\frac{W^2}{2}-\frac12|\nabla_p\Psi_\varepsilon(p)|^2-gp_2-\varepsilon^{-2}F\bigl(\Psi_\varepsilon(p)-\mu_\varepsilon\bigr).
    \]
    Since $\omega_\varepsilon=-\Delta_p\Psi_\varepsilon$ and
    \[
    (v_{{\rm rel},\varepsilon}\cdot\nabla_p)v_{{\rm rel},\varepsilon}=\nabla_p\left(\frac12|v_{{\rm rel},\varepsilon}|^2\right)+\omega_\varepsilon\nabla_p\Psi_\varepsilon,
    \]
    the steady Euler momentum equation follows. Moreover, \eqref{5-15} and \eqref{5-22} give $\mu_\varepsilon=\frac{\Gamma}{2\pi}\log(1/\varepsilon)+O(1)>0$. On $\partial\Omega_\varepsilon$ we have $\Psi_\varepsilon=0$, hence $F(\Psi_\varepsilon-\mu_\varepsilon)=0$, and the Bernoulli equation yields $P_\varepsilon=\sigma\kappa_\varepsilon$. The kinematic condition is automatic because the free surface is a streamline.
    
    The core center selected by the reduction is
    \begin{equation}
    	p_\varepsilon:=Z_\varepsilon(-\mathrm{i}a_\varepsilon).
    	\label{5-33}
    \end{equation}
    By \eqref{5-31},
    \begin{equation}
    	p_\varepsilon\longrightarrow p_W:=Z_{Y_W}(-\mathrm{i}a_W).
    	\label{5-34}
    \end{equation}
    The outer free boundary is $\partial\Omega_\varepsilon$, while the inner one bounds the active vortex region $D_\varepsilon:=\{p\in\Omega_\varepsilon:\Psi_\varepsilon(p)>\mu_\varepsilon\}$. Since $f$ is positive precisely on $(0,\infty)$, we have $D_\varepsilon=\{\omega_\varepsilon>0\}$. Extending $\omega_\varepsilon$ by zero outside $\Omega_\varepsilon$, its support is $\overline{D_\varepsilon}$. Neither boundary is prescribed in advance.
    
    \begin{theorem}
    	\label{thm5-5}
    	Fix $W$ as at the beginning of Section~\ref{sec5}, and let $(\Omega_W,p_W)$ be the corresponding point-vortex wave. Suppose the prescribed vorticity law $f$, independent of $\varepsilon$, and its radial profile $U$ satisfy Assumption~\ref{ass5-1}. For every sufficiently small $\varepsilon>0$, there exists a symmetric traveling wave $(\Omega_\varepsilon,\Psi_\varepsilon,\mu_\varepsilon)$ of speed $-W$ satisfying
    	\[
    	-\Delta_p\Psi_\varepsilon=\omega_\varepsilon=\varepsilon^{-2}f(\Psi_\varepsilon-\mu_\varepsilon)\quad\text{in }\Omega_\varepsilon,
    	\]
    	with $\Psi_\varepsilon=0$ and $\frac12|\nabla_p\Psi_\varepsilon|^2+gp_2+\sigma\kappa_\varepsilon=\frac{W^2}{2}$ on $\partial\Omega_\varepsilon$. The circulation is preserved exactly:
    	\refstepcounter{equation}
    	\[
    	\omega_\varepsilon\ge0,\qquad \int_{\Omega_\varepsilon}\omega_\varepsilon\,dp=\Gamma.
    	\tag{\theequation}\label{5-35}
    	\]
    	The centers $p_\varepsilon$ defined in \eqref{5-33} converge to $p_W$. Extending $\omega_\varepsilon$ by zero outside $\Omega_\varepsilon$, there is a constant $C>0$, independent of $\varepsilon$, such that
    	\refstepcounter{equation}
    	\[
    	\operatorname{supp}\omega_\varepsilon=\overline{D_\varepsilon}\subset B_{C\varepsilon}(p_\varepsilon).
    	\tag{\theequation}\label{5-36}
    	\]
    	Consequently,
    	\refstepcounter{equation}
    	\[
    	\omega_\varepsilon\,dp\rightharpoonup\Gamma\delta_{p_W}
    	\tag{\theequation}\label{5-37}
    	\]
    	as Radon measures. On the physical inner scale, let $\widehat D_\varepsilon:=(D_\varepsilon-p_\varepsilon)/\varepsilon$. Then
    	\refstepcounter{equation}
    	\[
    	\overline{\widehat D_\varepsilon}\longrightarrow\overline{B_1}
    	\tag{\theequation}\label{5-38}
    	\]
    	in the Hausdorff sense. More precisely, for sufficiently small $\varepsilon$,
    	\[
    	\partial\widehat D_\varepsilon=\left\{\bigl(1+\rho_\varepsilon(\theta)\bigr)(\cos\theta,\sin\theta):\theta\in\mathbb S^1\right\},
    	\]
    	where, for every $0<\alpha'<\alpha$,
    	\refstepcounter{equation}
    	\[
    	\|\rho_\varepsilon\|_{C^{1,\alpha'}(\mathbb S^1)}\longrightarrow0.
    	\tag{\theequation}\label{5-39}
    	\]
    	Moreover,
    	\refstepcounter{equation}
    	\[
    	\varepsilon^2\omega_\varepsilon(p_\varepsilon+\varepsilon X)\longrightarrow f(U(X))
    	\tag{\theequation}\label{5-40}
    	\]
    	locally uniformly in $X$. On the outer scale, after identifying nearby domains through their conformal parametrizations,
    	\refstepcounter{equation}
    	\[
    	\Psi_\varepsilon\longrightarrow\Gamma G_{\Omega_W}(\,\cdot\,,p_W)+\Psi_{\Omega_W,W}^0
    	\tag{\theequation}\label{5-41}
    	\]
    	locally in $C^1$ away from $p_W$, and the corresponding velocity fields converge there. Finally, the free surface is a graph $p_2=\eta_\varepsilon(p_1)$, and
    	\refstepcounter{equation}
    	\[
    	\eta_\varepsilon\longrightarrow\eta_W\qquad\text{in }H^3(\mathbb R).
    	\tag{\theequation}\label{5-42}
    	\]
    \end{theorem}
    
    \begin{proof}
    	At $a=a_\varepsilon$, Proposition~\ref{prop5-3} solves the Bernoulli and complementary inner equations, while \eqref{5-30} removes the remaining translation component. The flux constant fixes the circulation, so \eqref{5-32} gives a solution of the full physical system and proves \eqref{5-35}.
    	
    	For the inner limit, set $\widehat\Psi_\varepsilon(X):=\Psi_\varepsilon(p_\varepsilon+\varepsilon X)-\mu_\varepsilon$. The Taylor expansion of the conformal map at $-\mathrm{i}a_\varepsilon$ gives
    	\[
    	Z_\varepsilon\left(-\mathrm{i}a_\varepsilon+\frac{\varepsilon}{\alpha(Y_\varepsilon,a_\varepsilon)}X\right)=p_\varepsilon+\varepsilon X+O(\varepsilon^2)
    	\]
    	on bounded sets. Equations \eqref{5-14}, \eqref{5-22}, and \eqref{5-31}, together with the local estimate for $\widetilde\phi_\varepsilon$, yield
    	\begin{equation}
    		\widehat\Psi_\varepsilon\longrightarrow U\qquad\text{locally in }C^{1,\alpha'}(\mathbb R^2),\quad 0<\alpha'<\alpha.
    		\label{5-43}
    	\end{equation}
    	Since $\widehat D_\varepsilon=\{\widehat\Psi_\varepsilon>0\}$ and $U_r(1)=-\Gamma/(2\pi)\ne0$, the implicit function theorem gives a unique zero level near $\partial B_1$, represented by $r=1+\rho_\varepsilon(\theta)$ with $\rho_\varepsilon\to0$ in $C^{1,\alpha'}$. The uniform negativity outside a fixed inner ball proved in Section~5.2 persists under the small correction and level adjustment; within that ball, \eqref{5-43} excludes other active components away from the zero-level annulus. This proves \eqref{5-38}--\eqref{5-39} and yields a uniform core radius $C\varepsilon$, establishing \eqref{5-36}. The vorticity convergence \eqref{5-40} follows from \eqref{5-43} and the continuity of $f$. Together with \eqref{5-34} and the exact circulation, \eqref{5-36} implies the measure convergence \eqref{5-37}.
    	
    	Away from the core, the Dirichlet projection converges locally smoothly to $\Gamma G_{\Pi_-}(\,\cdot\,,-\mathrm{i}a_W)$, while $\phi_\varepsilon\to0$ and $Z_\varepsilon\to Z_{Y_W}$. Conformal invariance of the Green function gives \eqref{5-41} and the corresponding convergence of velocities.
    	
    	Finally, write $Z_\varepsilon(x)=\Xi_\varepsilon(x)+\mathrm{i}Y_\varepsilon(x)$, where $\Xi_\varepsilon=x-HY_\varepsilon$. From \eqref{5-31},
    	\[
    	\Xi_\varepsilon-\Xi_W=-H(Y_\varepsilon-Y_W)\longrightarrow0\qquad\text{in }H^3(\mathbb R).
    	\]
    	Since $\Xi_W'\ge\frac12$, Sobolev embedding gives $\Xi_\varepsilon'>0$ uniformly for small $\varepsilon$. Hence $\eta_\varepsilon=Y_\varepsilon\circ\Xi_\varepsilon^{-1}$ defines a physical graph, and continuity of Sobolev inversion and composition yields
    	\[
    	\eta_\varepsilon\longrightarrow Y_W\circ\Xi_W^{-1}=\eta_W\qquad\text{in }H^3(\mathbb R),
    	\]
    	which proves \eqref{5-42}.
    \end{proof}
    
    Theorem~\ref{thm3} follows from Propositions~\ref{prop5-3}--\ref{prop5-4} and Theorem~\ref{thm5-5}. The reduction gives the limiting regular-drift equation; physical reconstruction supplies both free boundaries and the fixed-circulation limit.
    
    \bigskip
    
    \section{Point-vortex waves in periodic domains}
    \label{sec6}
    We turn to periodic waves with one point vortex in each period cell. On the covering plane, this corresponds to an infinite row of vortices, whose regular drift approaches a nonzero limit as the vortex depth increases. We construct a coupled branch near the deep-vortex limit and study equilibria in prescribed periodic domains.
    
    Fix a horizontal period $L>0$ and write
    \[
    \mathbb T_L:=\mathbb R/L\mathbb Z,\qquad k:=\frac{2\pi}{L},\qquad \Pi_-^L:=\mathbb T_L\times(-\infty,0).
    \]
    We use the covering half-plane for conformal maps and vortex positions, with the identification $z\sim z+L$.
    
    \subsection{Periodic conformal calculus and symmetry spaces}
    For an $L$-periodic function
    \[
    f(x)=\sum_{n\in\mathbb Z}\widehat f_n e^{\mathrm{i}nkx},\qquad \widehat f_n:=\frac1L\int_{-L/2}^{L/2}f(x)e^{-\mathrm{i}nkx}\,dx,
    \]
    define
    \[
    H_Lf:=\sum_{n\in\mathbb Z}\mathrm{i}\operatorname{sgn}(n)\widehat f_n e^{\mathrm{i}nkx},\qquad |D_L|f:=\sum_{n\in\mathbb Z}|n|k\,\widehat f_n e^{\mathrm{i}nkx}.
    \]
    As in the solitary problem, $|D_L|=-H_L\partial_x$. The harmonic extension into $\Pi_-^L$ is
    \[
    e^{y|D_L|}f(x)=\sum_{n\in\mathbb Z}e^{|n|ky}\widehat f_n e^{\mathrm{i}nkx},\qquad y<0,
    \]
    and its boundary normal derivative is $|D_L|f$. For a periodic graph domain $\Omega_\eta=\{(x,y):y<\eta(x)\}$, the corresponding Dirichlet--Neumann operator is
    \[
    \mathcal N_L(\eta)\varphi:=\bigl(\partial_y\Phi-\eta'\partial_x\Phi\bigr)\big|_{y=\eta(x)},
    \]
    where $\Phi$ is the bounded periodic harmonic extension of $\varphi$.
    
    For a periodic conformal surface profile $Y$, set
    \[
    Z_Y(x+\mathrm{i}y)=x+\mathrm{i}y-e^{y|D_L|}H_LY(x)+\mathrm{i}e^{y|D_L|}Y(x).
    \]
    Then $Z_Y(z+L)=Z_Y(z)+L$. If $Y$ is even, the oddness of $H_LY$ gives
    \[
    Z_Y(-\overline z)=-\overline{Z_Y(z)}.
    \]
    On the boundary,
    \[
    Z_Y(x)=\Xi_Y(x)+\mathrm{i}Y(x),\qquad \Xi_Y(x)=x-H_LY(x),\qquad \Xi_Y'(x)=1+|D_L|Y(x).
    \]
    The conformal Jacobian and curvature are
    \[
    J_L(Y)=(1+|D_L|Y)^2+(Y')^2,\qquad \kappa_L(Y)=-\frac{(1+|D_L|Y)Y''-Y'(|D_L|Y)'}{J_L(Y)^{3/2}}.
    \]
    In particular, $D\kappa_L(0)[h]=-h''$.
    
    We impose the normalization $Z_Y(z)-z\to0$ as $\Im z\to-\infty$. Since the nonzero Fourier modes decay at infinite depth, this is equivalent to
    \[
    \langle Y\rangle_L:=\frac1L\int_{\mathbb T_L}Y(x)\,dx=0.
    \]
    Define
    \[
    H^s_{\mathrm e,0}(\mathbb T_L):=\{f\in H^s(\mathbb T_L):f(-x)=f(x),\ \langle f\rangle_L=0\},
    \]
    and, for $s>3/2$, set
    \[
    \mathcal X_L^s:=H^{s+2}_{\mathrm e,0}(\mathbb T_L),\qquad \mathcal Y_L^s:=H^s_{\mathrm e,0}(\mathbb T_L).
    \]
    If $Y$ is sufficiently small in $\mathcal X_L^s$, then $1+|D_L|Y>0$. Its harmonic extension remains positive, and the argument of Lemma~\ref{lem2-1} shows that $Z_Y$ is univalent and maps the reference half-plane onto a periodic graph domain.
    
    Finally, we eliminate the Bernoulli constant by the mean-zero projection
    \[
    \mathbb P_Lf:=f-\langle f\rangle_L.
    \]
    The constant is recovered from the mean after solving the projected Bernoulli equation. Thus the unknowns are the surface profile $Y$ and the vortex depth $a$; for the deep-vortex construction, we use $\rho=e^{-ka}$ in place of $a$.
    
    \smallskip
    
    \subsection{The coupled periodic point-vortex problem}
    The periodic Dirichlet Green function in $\Pi_-^L$ is
    \begin{equation}
    	G_L(z,\zeta)=\frac1{2\pi}\log\left|\frac{\sin\frac{\pi}{L}(z-\overline\zeta)}{\sin\frac{\pi}{L}(z-\zeta)}\right|=\frac1{4\pi}\log\frac{\cosh k(y+\eta)-\cos k(x-\xi)}{\cosh k(y-\eta)-\cos k(x-\xi)},
    	\label{6-1}
    \end{equation}
    where $z=x+\mathrm{i}y$ and $\zeta=\xi+\mathrm{i}\eta$. The formula sums the periodic vortex row and its reflected images. It vanishes on $y=0$, has the prescribed logarithmic singularity at $\zeta$, and remains bounded at infinite depth.
    
    Place a vortex of circulation $\Gamma$ at $z_a=-\mathrm{i}a$ and set $\psi_{a,L}(z):=\Gamma G_L(z,z_a)$. Differentiating \eqref{6-1} on the boundary gives
    \begin{equation}
    	q_{a,W}^L(x):=\partial_y(\psi_{a,L}+Wy)(x,0)=W-\frac{\Gamma k}{2\pi}\frac{\sinh(ka)}{\cosh(ka)-\cos(kx)}.
    	\label{6-2}
    \end{equation}
    As $L\to\infty$ with $a$ and $x$ fixed, this reduces to the solitary boundary velocity $W-\Gamma a/[\pi(x^2+a^2)]$.
    
    The squared boundary speed transforms by the factor $J_L(Y)^{-1}$. Eliminating the Bernoulli constant through the mean-zero projection, we obtain
    \begin{equation}
    	\mathcal B_L(Y,a,W):=\mathbb P_L\left[gY+\sigma\kappa_L(Y)+\frac{(q_{a,W}^L)^2}{2J_L(Y)}\right].
    	\label{6-3}
    \end{equation}
    The projected equation $\mathcal B_L(Y,a,W)=0$ imposes Bernoulli's condition, with the constant recovered from its mean.
    
    We subtract the logarithmic self-field to obtain the regular velocity at the vortex. The periodic image contribution gives the flat drift $-\Gamma k\coth(ka)/(4\pi)$, while the physical singularity subtraction produces the same conformal correction as in Section~\ref{sec2}. Hence the point-vortex equilibrium equation is
    \begin{equation}
    	\mathcal P_L(Y,a,W):=W-\frac{\Gamma k}{4\pi}\coth(ka)+\frac{\mathrm{i}\Gamma}{4\pi}\frac{Z_Y''(-\mathrm{i}a)}{Z_Y'(-\mathrm{i}a)}=0.
    	\label{6-4}
    \end{equation}
    For even $Y$, this expression is real, and the periodic harmonic extension gives
    \begin{equation}
    	\mathcal P_L(Y,a,W)=W-\frac{\Gamma k}{4\pi}\coth(ka)-\frac{\Gamma}{4\pi}\frac{(e^{-a|D_L|}Y'')(0)}{1+(e^{-a|D_L|}|D_L|Y)(0)}.
    	\label{6-5}
    \end{equation}
    We therefore seek solutions of the coupled system
    \begin{equation}
    	\mathcal B_L(Y,a,W)=0,\qquad \mathcal P_L(Y,a,W)=0.
    	\label{6-6}
    \end{equation}
    
    For the deep-vortex limit, set $\rho:=e^{-ka}$. The relevant kernels become
    \[
    \coth(ka)=\frac{1+\rho^2}{1-\rho^2},\qquad \frac{\sinh(ka)}{\cosh(ka)-\cos(kx)}=\frac{1-\rho^2}{1-2\rho\cos(kx)+\rho^2}=1+2\sum_{n\ge1}\rho^n\cos(nkx).
    \]
    The latter is the periodic Poisson kernel, whose nonconstant modes vanish as $\rho\to0$.
    
    For $|\rho|<1$, define
    \[
    \mathcal R_\rho f:=\sum_{n\ne0}\rho^{|n|}\widehat f_n e^{\mathrm{i}nkx}.
    \]
    When $\rho=e^{-ka}>0$, this agrees with $e^{-a|D_L|}$ on mean-zero functions. Using $\mathcal R_\rho$ in \eqref{6-5}, together with the preceding expressions for the periodic kernels, extends the coupled equations smoothly to $\rho=0$.
    
    The limiting drift determines the critical speed
    \begin{equation}
    	W_c:=\frac{\Gamma k}{4\pi}=\frac{\Gamma}{2L}.
    	\label{6-7}
    \end{equation}
    Indeed, for a flat surface the center equation reduces to
    \[
    W-W_c\frac{1+\rho^2}{1-\rho^2}=0.
    \]
    The boundary velocity can now be written as
    \[
    q_{\rho,W}^L(x)=W-2W_c\left(1+2\sum_{n\ge1}\rho^n\cos(nkx)\right),
    \]
    so that $q_{0,W_c}^L=-W_c$. Consequently, $(Y,\rho,W)=(0,0,W_c)$ solves the limiting coupled system. Unlike in the solitary small-speed limit, the relative boundary velocity remains nonzero and contributes to the Bernoulli linearization, hence to the nonresonance condition.
    
    \smallskip
    
    \subsection{A local deep-vortex branch}
    Write $\mathcal B_L(Y,\rho,W)$ and $\mathcal P_L(Y,\rho,W)$ for the smooth extensions from Section~6.2, and set
    \[
    \mathfrak F_L(\rho,Y,W):=\bigl(\mathcal B_L(Y,\rho,W),\mathcal P_L(Y,\rho,W)\bigr).
    \]
    This map is smooth from a neighborhood of $(0,0,W_c)$ in $\mathbb R\times\mathcal X_L^s\times\mathbb R$ to $\mathcal Y_L^s\times\mathbb R$, with $\mathfrak F_L(0,0,W_c)=0$. Negative values of $\rho$ are used only to apply the implicit function theorem; physical depths correspond to $\rho>0$.
    
    Since $q_{0,W_c}^L=-W_c$ and $DJ_L(0)[h]=2|D_L|h$, the surface linearization is
    \[
    D_Y\mathcal B_L(0,0,W_c)[h]=\mathscr L_ch,\qquad \mathscr L_c:=g-\sigma\partial_x^2-W_c^2|D_L|.
    \]
    The derivative with respect to $W$ is constant and is removed by $\mathbb P_L$, while the conformal contribution to the vortex equation vanishes at $\rho=0$. Hence
    \[
    D_{(Y,W)}\mathfrak F_L(0,0,W_c)[h,\omega]=(\mathscr L_ch,\omega).
    \]
    On the even mean-zero Fourier space,
    \[
    \mathscr L_c\cos(nkx)=\lambda_n\cos(nkx),\qquad \lambda_n:=g+\sigma n^2k^2-W_c^2nk,\quad n\ge1.
    \]
    We assume
    \begin{equation}
    	\lambda_n\ne0\qquad\text{for every }n\ge1.
    	\label{6-8}
    \end{equation}
    Since $\lambda_n\sim\sigma k^2n^2$, the operator $\mathscr L_c:\mathcal X_L^s\to\mathcal Y_L^s$ is then invertible. Resonance occurs precisely when $W_c^2=g/(nk)+\sigma nk$ for some $n\ge1$, or equivalently when
    \[
    \Gamma=\frac{4\pi}{k}\left(\frac{g}{nk}+\sigma nk\right)^{1/2}.
    \]
    For fixed $g,\sigma,L>0$, these exceptional circulations form a discrete set tending to infinity. In particular, $W_c^2<2\sqrt{g\sigma}$ is a sufficient nonresonance condition. The individual $\lambda_n$ need not be positive.
    
    \begin{theorem}
    	\label{thm6-1}
    	Assume \eqref{6-8}. Then there exist $\rho_0>0$ and a unique smooth local curve
    	\[
    	(-\rho_0,\rho_0)\ni\rho\longmapsto(Y_\rho,W_\rho)\in\mathcal X_L^s\times\mathbb R
    	\]
    	satisfying $Y_0=0$, $W_0=W_c$, and $\mathfrak F_L(\rho,Y_\rho,W_\rho)=0$. For every sufficiently small $\rho>0$, the depth $a_\rho=-k^{-1}\log\rho$ gives a symmetric periodic point-vortex water wave. Moreover,
    	\begin{equation}
    		Y_\rho=-\frac{4W_c^2}{\lambda_1}\rho\cos(kx)+O_{\mathcal X_L^s}(\rho^2),
    		\label{6-9}
    	\end{equation}
    	and
    	\begin{equation}
    		W_\rho=W_c+c_L\rho^2+O(\rho^3),\qquad c_L:=2W_c\frac{g+\sigma k^2+W_c^2k}{g+\sigma k^2-W_c^2k}.
    		\label{6-10}
    	\end{equation}
    \end{theorem}
    
    \begin{proof}
    	The linearization $D_{(Y,W)}\mathfrak F_L(0,0,W_c)=(\mathscr L_c,\mathrm{Id})$ is invertible by \eqref{6-8}. The implicit function theorem therefore yields the asserted smooth local curve. For $\rho>0$ sufficiently small, the graph condition $1+|D_L|Y_\rho>0$ holds, so the conformal map reconstructs a physical periodic wave.
    	
    	Differentiate the coupled equations at $\rho=0$ to find the leading surface profile. The vortex equation gives $W_0'=0$, because $\coth(ka)=(1+\rho^2)/(1-\rho^2)$ has no linear term and the conformal correction is at least quadratic along the branch. Moreover,
    	\[
    	\left.\partial_\rho q_{\rho,W_c}^L\right|_{\rho=0}=-4W_c\cos(kx).
    	\]
    	Since $q_{0,W_c}^L=-W_c$, differentiation of the projected Bernoulli equation gives
    	\[
    	\mathscr L_cY_0'=-4W_c^2\cos(kx).
    	\]
    	Thus $Y_0'=-4W_c^2\lambda_1^{-1}\cos(kx)$, proving \eqref{6-9}.
    	
    	Set $A:=-4W_c^2/\lambda_1$, so $Y_\rho=A\rho\cos(kx)+O_{\mathcal X_L^s}(\rho^2)$. By the definition of $\mathcal R_\rho$,
    	\[
    	(\mathcal R_\rho Y_\rho'')(0)=-Ak^2\rho^2+O(\rho^3),\qquad 1+(\mathcal R_\rho|D_L|Y_\rho)(0)=1+Ak\rho^2+O(\rho^3).
    	\]
    	Together with $\coth(ka)=1+2\rho^2+O(\rho^4)$, the vortex equation \eqref{6-5} yields
    	\[
    	W_\rho-W_c=\left(2W_c-\frac{\Gamma Ak^2}{4\pi}\right)\rho^2+O(\rho^3).
    	\]
    	Using $W_c=\Gamma k/(4\pi)$, we obtain
    	\[
    	2W_c-\frac{\Gamma Ak^2}{4\pi}=2W_c+\frac{4W_c^3k}{\lambda_1}=2W_c\frac{g+\sigma k^2+W_c^2k}{g+\sigma k^2-W_c^2k},
    	\]
    	which proves \eqref{6-10}.
    \end{proof}
    
    The numerator of $c_L$ is positive and its denominator equals $\lambda_1$. Therefore $c_L\ne0$ and $\operatorname{sgn}c_L=\operatorname{sgn}\lambda_1$: the branch approaches $W_c$ from above when $\lambda_1>0$ and from below when $\lambda_1<0$. Smoothness gives $W_\rho'=2c_L\rho+O(\rho^2)$, so $W_\rho$ is strictly monotone for sufficiently small $\rho>0$. Since
    \[
    |W_\rho-W_c|=|c_L|\rho^2(1+O(\rho)),
    \]
    taking logarithms and using $a_\rho=-k^{-1}\log\rho$ yields
    \begin{equation}
    	a(W)=\frac1{2k}\log\frac{|c_L|}{|W-W_c|}+o(1).
    	\label{6-11}
    \end{equation}
    This proves Theorem~\ref{thm4}. Table~\ref{tab:solitary-periodic} compares the periodic and solitary deep-vortex limits.
    
    \begin{table}[!htbp]
    	\centering
    	\caption{Comparison of the solitary and periodic coupled deep-vortex branches.}
    	\label{tab:solitary-periodic}
    	\small
    	\setlength{\tabcolsep}{5pt}
    	\renewcommand{\arraystretch}{1.25}
    	\begin{tabular*}{\textwidth}{@{\extracolsep{\fill}}lll@{}}
    		\toprule
    		& \textbf{Solitary problem} & \textbf{Periodic problem} \\
    		\midrule
    		Reference geometry & $\mathbb R\times(-\infty,0)$ & $\mathbb T_L\times(-\infty,0)$ \\
    		\addlinespace[5pt]
    		Deep-vortex limit & $W\to0^+$ & $W\to W_c$, \quad $W_c=\dfrac{\Gamma}{2L}$ \\
    		\addlinespace[5pt]
    		Reference drift & $\dfrac{\Gamma}{4\pi a}\to0$ & $W_c\coth(ka)\to W_c$ \\
    		\addlinespace[5pt]
    		Vortex depth & $a\sim\dfrac{\Gamma}{4\pi W}$ & $a\sim\dfrac{1}{2k}\log\dfrac{1}{|W-W_c|}$ \\
    		\bottomrule
    	\end{tabular*}
    \end{table}
    \FloatBarrier
    
    \smallskip
    
   \subsection{Kirchhoff--Routh reduction in periodic domains}
   \label{subsec6.4}
   We consider point-vortex equilibria in a prescribed periodic fluid domain. Let $\Omega$ be represented on the covering space by a conformal map $Z$ satisfying
   \[
   Z(z+L)=Z(z)+L,\qquad Z(-\overline z)=-\overline{Z(z)},
   \]
   and
   \[
   Z(z)-z\to0,\qquad Z'(z)\to1
   \]
   as $\Im z\to-\infty$, uniformly over one period. We assume that $Z$ extends as a $C^{3,\alpha}$ diffeomorphism to the boundary and that $Z'$ does not vanish on the closed reference half-plane. We call such domains admissible symmetric periodic water-wave domains. As in Section~\ref{sec4}, the boundary is prescribed without imposing the Bernoulli condition. Symmetry and normalization imply $Z'(-\mathrm{i}a)>0$ for $a\ge0$.
   
   Write the periodic Green function locally near the diagonal as
   \[
   G_L(z,\zeta)=\frac1{2\pi}\log\frac1{|z-\zeta|}-H_L^{\rm G}(z,\zeta).
   \]
   The superscript distinguishes the regular part from the periodic Hilbert transform. Expanding \eqref{6-1} at $\zeta=z$ gives
   \begin{equation}
   	R_L(z):=H_L^{\rm G}(z,z)=-\frac1{2\pi}\log\left(\frac{2\sinh(k|\Im z|)}{k}\right).
   	\label{6-12}
   \end{equation}
   At fixed depth, this tends to the half-plane Robin function as $L\to\infty$. The conformal transformation law, proved as in Lemma~\ref{lem4-3}, is
   \begin{equation}
   	R_\Omega(Z(z))=R_L(z)-\frac1{2\pi}\log|Z'(z)|.
   	\label{6-13}
   \end{equation}
   Thus, for $p_a:=Z(-\mathrm{i}a)$,
   \[
   R_\Omega(p_a)=-\frac1{2\pi}\log\left(\frac{2\sinh(ka)}{k}Z'(-\mathrm{i}a)\right).
   \]
   
   The background stream function and Kirchhoff--Routh function are
   \[
   \Psi_{\Omega,W}^0(p):=W\Im Z^{-1}(p),\qquad \mathscr K_{\Omega,W}^L(p):=-\frac{\Gamma^2}{2}R_\Omega(p)+\Gamma\Psi_{\Omega,W}^0(p).
   \]
   Restricting to the symmetry axis yields
   \begin{equation}
   	K_{\Omega,W}^L(a):=\mathscr K_{\Omega,W}^L(p_a)=\frac{\Gamma^2}{4\pi}\log\left(\frac{2\sinh(ka)}{k}Z'(-\mathrm{i}a)\right)-\Gamma Wa.
   	\label{6-14}
   \end{equation}
   Differentiation gives
   \[
   (K_{\Omega,W}^L)'(a)=\frac{\Gamma^2}{4\pi}\left(k\coth(ka)-\mathrm{i}\frac{Z''(-\mathrm{i}a)}{Z'(-\mathrm{i}a)}\right)-\Gamma W.
   \]
   Defining
   \begin{equation}
   	\mathcal P_{\Omega,L}(a,W):=W-\frac{\Gamma k}{4\pi}\coth(ka)+\frac{\mathrm{i}\Gamma}{4\pi}\frac{Z''(-\mathrm{i}a)}{Z'(-\mathrm{i}a)},
   	\label{6-15}
   \end{equation}
   we obtain
   \[
   (K_{\Omega,W}^L)'(a)=-\Gamma\mathcal P_{\Omega,L}(a,W).
   \]
   As in the solitary case, symmetry lifts critical points on the axis to critical points of the full Kirchhoff--Routh function. Moreover,
   \[
   v_{\rm reg}(p_a)=\frac{\mathcal P_{\Omega,L}(a,W)}{Z'(-\mathrm{i}a)}e_1,
   \]
   so $\mathcal P_{\Omega,L}(a,W)=0$ is precisely the symmetric point-vortex equilibrium condition.
   
   The periodic reference drift approaches $\Gamma k/(4\pi)=W_c$ as $a\to\infty$, rather than vanishing as in the solitary case. We first consider the flat domain, where the equilibrium equation can be solved explicitly.
   
   \begin{proposition}
   	\label{prop6-2}
   	The flat periodic problem has a finite-depth equilibrium if and only if $W>W_c$. In that case the equilibrium is unique and
   	\begin{equation}
   		a_L(W)=\frac1k\operatorname{arccoth}\frac{W}{W_c}=\frac1{2k}\log\frac{W+W_c}{W-W_c}.
   		\label{6-16}
   	\end{equation}
   	Moreover,
   	\begin{equation}
   		a_L(W)=\frac1{2k}\log\frac{2W_c}{W-W_c}+o(1)\qquad\text{as }W\downarrow W_c.
   		\label{6-17}
   	\end{equation}
   \end{proposition}
   
   \begin{proof}
   	For $Z(z)=z$, equation \eqref{6-15} reduces to $W=W_c\coth(ka)$. Since $\coth(ka)$ decreases strictly from $+\infty$ to $1$, there is exactly one finite solution when $W>W_c$ and none otherwise. Inverting this relation gives \eqref{6-16}, and \eqref{6-17} follows immediately.
   \end{proof}
   
   We consider maximizing equilibria in a general admissible periodic domain. The boundary regularity of $Z$ and $\sinh(ka)\sim ka$ imply
   \[
   K_{\Omega,W}^L(a)=\frac{\Gamma^2}{4\pi}\log(2aZ'(0))+o(1)\longrightarrow-\infty
   \]
   as $a\to0^+$. At infinite depth, the normalization of $Z$ gives
   \begin{equation}
   	K_{\Omega,W}^L(a)=-\Gamma(W-W_c)a-\frac{\Gamma^2}{4\pi}\log k+o(1).
   	\label{6-18}
   \end{equation}
   Indeed,
   \[
   \log\frac{2\sinh(ka)}{k}=ka-\log k+\log(1-e^{-2ka}),
   \]
   while $\log Z'(-\mathrm{i}a)\to0$.
   
   \begin{proposition}
   	\label{prop6-3}
   	Let $\Omega$ be an admissible symmetric periodic water-wave domain and let $W>W_c$. Then the maximizing set
   	\[
   	\mathcal A_{\Omega,W}^L:=\left\{a>0:K_{\Omega,W}^L(a)=\sup_{b>0}K_{\Omega,W}^L(b)\right\}
   	\]
   	is nonempty and compact in $(0,\infty)$. Every maximizing depth determines a symmetric stationary point vortex in the traveling frame.
   \end{proposition}
   
   \begin{proof}
   	For $W>W_c$, the boundary behavior and \eqref{6-18} show that $K_{\Omega,W}^L(a)\to-\infty$ at both endpoints of $(0,\infty)$. Thus every maximizing sequence remains in a compact subinterval, and continuity gives a nonempty compact maximizing set. Each maximizer satisfies $(K_{\Omega,W}^L)'(a)=0$, which is equivalent to the point-vortex equilibrium condition by \eqref{6-15}.
   \end{proof}
   
   At the critical speed $W=W_c$, the linear term in \eqref{6-18} vanishes and
   \[
   K_{\Omega,W_c}^L(a)\longrightarrow K_L^\infty:=-\frac{\Gamma^2}{4\pi}\log k.
   \]
   If $K_{\Omega,W_c}^L(a)>K_L^\infty$ at some finite depth, the supremum is attained at an interior point. Otherwise, the supremum equals $K_L^\infty$; it may be attained at finite depth, but maximizing sequences can also escape to infinity. At the critical speed, the existence of an interior maximizer depends on the domain geometry.
   
   For $W<W_c$, equation \eqref{6-18} gives $K_{\Omega,W}^L(a)\to+\infty$ as $a\to\infty$, and global maximization no longer produces a finite-depth equilibrium. Critical points of the Kirchhoff--Routh function may nevertheless exist. In particular, the coupled branch of Theorem~\ref{thm6-1} provides such equilibria below $W_c$ when $\lambda_1<0$.
   
   The two drift corrections have different decay rates. For a fixed normalized periodic conformal map,
   \[
   Z'(-\mathrm{i}a)-1=O(e^{-ka}),\qquad \frac{Z''(-\mathrm{i}a)}{Z'(-\mathrm{i}a)}=O(e^{-ka}),
   \]
   whereas
   \[
   \coth(ka)-1=O(e^{-2ka}).
   \]
   In a prescribed domain, the geometric correction can therefore dominate that of the flat periodic vortex row. Along the coupled branch of Theorem~\ref{thm6-1}, however, $Y_\rho=O(\rho)$ and evaluation at the vortex contributes another factor $\rho=e^{-ka}$. Both drift corrections are then of order $\rho^2$, producing the coefficient $c_L$ and the logarithmic depth law \eqref{6-11}.
   
   The joint Lyapunov--Schmidt framework of Section~\ref{sec5} suggests a route to smooth-core regularization of periodic point-vortex waves. The inner limit remains the whole-plane radial core, while the outer Green function and reduced depth equation take their periodic forms. However, a periodic regularization theorem would still require the projected inner inverse and coupled nondegeneracy at the chosen wave.
   
   \bigskip

   \noindent{\bf Declaration of generative AI and AI-assisted technologies in the manuscript preparation process:}
   The authors developed the mathematical framework and proof strategies. ChatGPT (OpenAI) was used to assist with numerical computations, figure preparation, and the revision of the mathematical exposition and language of the manuscript. The authors reviewed and checked all AI-assisted material and take full responsibility for the content of this work.
   
   \bigskip

   \noindent{\bf Conflict of interest statement:} On behalf of all authors, the corresponding author states that there is no conflict of interest.
   
   \bigskip
   
   \noindent{\bf Data availability statement:} Our manuscript has no associated data.
   
   \bigskip
   \bigskip
   
   \phantom{s}
   \thispagestyle{empty}
   
   \bibliographystyle{plain}
   \bibliography{reference}

@misc{SleaseWalshZhong2026,
	author       = {Slease, G. and Walsh, S. and Zhong, R.},
	title        = {Existence and stability of capillary-gravity wave-borne vortices},
	howpublished = {arXiv:2609.29791},
	year         = {2026},
	url          = {https://arxiv.org/abs/2609.29791}
}

@article{AasenVarholm2018,
  author  = {Aasen, A. and Varholm, K.},
  title   = {Traveling gravity water waves with critical layers},
  journal = {J. Math. Fluid Mech.},
  volume  = {20},
  pages   = {161--187},
  year    = {2018}
}

@article{AmickFraenkelToland1982,
  author  = {Amick, C. J. and Fraenkel, L. E. and Toland, J. F.},
  title   = {On the {Stokes} conjecture for the wave of extreme form},
  journal = {Acta Math.},
  volume  = {148},
  number  = {1},
  pages   = {193--214},
  year    = {1982},
  doi     = {10.1007/BF02392728}
}

@article{BertiFranzoiMaspero2021,
  author  = {Berti, M. and Franzoi, L. and Maspero, A.},
  title   = {Traveling quasi-periodic water waves with constant vorticity},
  journal = {Arch. Ration. Mech. Anal.},
  volume  = {240},
  pages   = {99--202},
  year    = {2021}
}

@article{BertiFranzoiMaspero2024,
  author  = {Berti, M. and Franzoi, L. and Maspero, A.},
  title   = {Pure gravity traveling quasi-periodic water waves with constant vorticity},
  journal = {Comm. Pure Appl. Math.},
  volume  = {77},
  pages   = {990--1064},
  year    = {2024}
}

@article{BurtonToland2011,
  author  = {Burton, G. R. and Toland, J. F.},
  title   = {Surface waves on steady perfect-fluid flows with vorticity},
  journal = {Comm. Pure Appl. Math.},
  volume  = {64},
  pages   = {975--1007},
  year    = {2011}
}

@article{CaffarelliFriedman1980,
  author  = {Caffarelli, L. A. and Friedman, A.},
  title   = {Asymptotic estimates for the plasma problem},
  journal = {Duke Math. J.},
  volume  = {47},
  pages   = {705--742},
  year    = {1980}
}

@article{CaoLiuWei2014,
  author  = {Cao, D. and Liu, Z. and Wei, J.},
  title   = {Regularization of point vortex pairs for the {Euler} equation in dimension two},
  journal = {Arch. Ration. Mech. Anal.},
  volume  = {212},
  pages   = {179--217},
  year    = {2014}
}

@article{CaoPengYan2015,
  author  = {Cao, D. and Peng, S. and Yan, S.},
  title   = {Planar vortex patch problem in incompressible steady flow},
  journal = {Adv. Math.},
  volume  = {270},
  pages   = {263--301},
  year    = {2015},
  doi     = {10.1016/j.aim.2014.09.027}
}

@article{ChenVarholmWalshWheeler2025,
  author  = {Chen, R. M. and Varholm, K. and Walsh, S. and Wheeler, M. H.},
  title   = {Vortex-carrying solitary gravity waves of large amplitude},
  journal = {Comm. Math. Phys.},
  volume  = {406},
  number  = {7},
  pages   = {149},
  year    = {2025},
  doi     = {10.1007/s00220-025-05329-y}
}

@article{ChenWalshWheeler2019,
  author  = {Chen, R. M. and Walsh, S. and Wheeler, M. H.},
  title   = {Existence, nonexistence, and asymptotics of deep water solitary waves with localized vorticity},
  journal = {Arch. Ration. Mech. Anal.},
  volume  = {234},
  pages   = {595--633},
  year    = {2019}
}

@article{Constantin2001,
  author  = {Constantin, A.},
  title   = {Edge waves along a sloping beach},
  journal = {J. Phys. A},
  volume  = {34},
  pages   = {9723--9731},
  year    = {2001}
}

@article{Constantin2010,
  author  = {Constantin, A.},
  title   = {On the particle paths in solitary water waves},
  journal = {Quart. Appl. Math.},
  volume  = {68},
  pages   = {81--90},
  year    = {2010}
}

@article{ConstantinEhrnstromWahlen2007,
  author  = {Constantin, A. and Ehrnstr{\"o}m, M. and Wahl{\'e}n, E.},
  title   = {Symmetry of steady periodic gravity water waves with vorticity},
  journal = {Duke Math. J.},
  volume  = {140},
  pages   = {591--603},
  year    = {2007}
}

@article{ConstantinEscher2007,
  author  = {Constantin, A. and Escher, J.},
  title   = {Particle trajectories in solitary water waves},
  journal = {Bull. Amer. Math. Soc.},
  volume  = {44},
  pages   = {423--431},
  year    = {2007}
}

@article{ConstantinStrauss2004,
  author  = {Constantin, A. and Strauss, W.},
  title   = {Exact steady periodic water waves with vorticity},
  journal = {Comm. Pure Appl. Math.},
  volume  = {57},
  pages   = {481--527},
  year    = {2004}
}

@article{ConstantinStraussVarvaruca2016,
  author  = {Constantin, A. and Strauss, W. and V{\u a}rv{\u a}ruc{\u a}, E.},
  title   = {Global bifurcation of steady gravity water waves with critical layers},
  journal = {Acta Math.},
  volume  = {217},
  pages   = {195--262},
  year    = {2016}
}

@article{Crapper1957,
  author  = {Crapper, G. D.},
  title   = {An exact solution for progressive capillary waves of arbitrary amplitude},
  journal = {J. Fluid Mech.},
  volume  = {2},
  pages   = {532--540},
  year    = {1957}
}

@article{DavilaDelPinoMussoWei2020,
  author  = {D{\'a}vila, J. and del Pino, M. and Musso, M. and Wei, J.},
  title   = {Gluing methods for vortex dynamics in {Euler} flows},
  journal = {Arch. Ration. Mech. Anal.},
  volume  = {235},
  pages   = {1467--1530},
  year    = {2020}
}

@article{DavilaDelPinoMussoWheeler2026,
  author  = {D{\'a}vila, J. and del Pino, M. and Musso, M. and Wheeler, M. H.},
  title   = {Overhanging solitary water waves},
  journal = {Invent. Math.},
  volume  = {244},
  pages   = {643--767},
  year    = {2026}
}

@article{DeValeriolaVanSchaftingen2013,
  author  = {de Valeriola, S. and Van Schaftingen, J.},
  title   = {Desingularization of vortex rings and shallow water vortices by a semilinear elliptic problem},
  journal = {Arch. Ration. Mech. Anal.},
  volume  = {210},
  pages   = {409--450},
  year    = {2013}
}

@unpublished{DuZouPartII,
  author = {Du, L. and Zou, C.},
  title  = {{Capillary--gravity water waves with compactly supported vorticity II. Continuation and folds of the solution curve}},
  year   = {2026},
  note   = {Preprint}
}

@unpublished{DuZouPartIII,
  author = {Du, L. and Zou, C.},
  title  = {{Capillary--gravity water waves with compactly supported vorticity III. Joint dual variational construction and effective Hessian}},
  year   = {2026},
  note   = {Preprint}
}

@unpublished{DuZouPartIV,
	author = {Du, L. and Zou, C.},
	title  = {{Capillary--gravity water waves with compactly supported vorticity IV. Rigidity and uniqueness of concentrated waves}},
	year   = {2026},
	note   = {Preprint}
}

@unpublished{DuZouPartV,
	author = {Du, L. and Zou, C.},
	title  = {{Capillary--gravity water waves with compactly supported vorticity V. Fixed-impulse coercivity and logarithmic-time orbital tracking}},
	year   = {2026},
	note   = {Preprint}
}

@article{EhrnstromEscherWahlen2011,
  author  = {Ehrnstr{\"o}m, M. and Escher, J. and Wahl{\'e}n, E.},
  title   = {Steady water waves with multiple critical layers},
  journal = {SIAM J. Math. Anal.},
  volume  = {43},
  pages   = {1436--1456},
  year    = {2011}
}

@article{EhrnstromVillari2008,
  author  = {Ehrnstr{\"o}m, M. and Villari, G.},
  title   = {Linear water waves with vorticity: rotational features and particle paths},
  journal = {J. Differential Equations},
  volume  = {244},
  pages   = {1888--1909},
  year    = {2008}
}

@article{EhrnstromWahlen2015,
  author  = {Ehrnstr{\"o}m, M. and Wahl{\'e}n, E.},
  title   = {Trimodal steady water waves},
  journal = {Arch. Ration. Mech. Anal.},
  volume  = {216},
  pages   = {449--471},
  year    = {2015}
}

@article{EhrnstromWalshZeng2023,
  author  = {Ehrnstr{\"o}m, M. and Walsh, S. and Zeng, C.},
  title   = {Smooth stationary water waves with exponentially localized vorticity},
  journal = {J. Eur. Math. Soc.},
  volume  = {25},
  pages   = {1045--1090},
  year    = {2023}
}

@article{ElcratMiller2006,
  author  = {Elcrat, A. R. and Miller, K. G.},
  title   = {Free surface waves in equilibrium with a vortex},
  journal = {Eur. J. Mech. B Fluids},
  volume  = {25},
  number  = {2},
  pages   = {255--266},
  year    = {2006},
  doi     = {10.1016/j.euromechflu.2005.09.001}
}

@article{FeolaGiuliani2024,
  author  = {Feola, R. and Giuliani, F.},
  title   = {Quasi-periodic traveling waves on an infinitely deep perfect fluid under gravity},
  journal = {Mem. Amer. Math. Soc.},
  volume  = {295},
  number  = {1471},
  pages   = {1--158},
  year    = {2024},
  doi     = {10.1090/memo/1471}
}

@article{Filippov1960,
  author  = {Filippov, I. G.},
  title   = {Solution of the problem of the motion of a vortex under the surface of a fluid, for {Froude} numbers near unity},
  journal = {J. Appl. Math. Mech.},
  volume  = {24},
  pages   = {698--716},
  year    = {1960}
}

@article{Filippov1961,
  author  = {Filippov, I. G.},
  title   = {On the motion of a vortex below the surface of a liquid},
  journal = {J. Appl. Math. Mech.},
  volume  = {25},
  pages   = {357--365},
  year    = {1961}
}

@article{FlucherWei1998,
  author  = {Flucher, M. and Wei, J.},
  title   = {Asymptotic shape and location of small cores in elliptic free-boundary problems},
  journal = {Math. Z.},
  volume  = {228},
  pages   = {683--703},
  year    = {1998}
}

@article{Forbes1985,
  author  = {Forbes, L. K.},
  title   = {On the effects of non-linearity in free-surface flow about a submerged point vortex},
  journal = {J. Eng. Math.},
  volume  = {19},
  pages   = {139--155},
  year    = {1985},
  doi     = {10.1007/BF00042737}
}

@article{GaoWangVandenBroeck2017,
  author  = {Gao, T. and Wang, Z. and Vanden-Broeck, J.-M.},
  title   = {Investigation of symmetry breaking in periodic gravity-capillary waves},
  journal = {J. Fluid Mech.},
  volume  = {811},
  pages   = {622--641},
  year    = {2017}
}

@article{Gerstner1809,
  author  = {Gerstner, F.},
  title   = {Theorie der {Wellen}},
  journal = {Ann. Phys.},
  volume  = {32},
  number  = {8},
  pages   = {412--445},
  year    = {1809},
  doi     = {10.1002/andp.18090320808}
}

@article{GrovesWahlen2008,
  author  = {Groves, M. D. and Wahl{\'e}n, E.},
  title   = {Small-amplitude {Stokes} and solitary gravity water waves with an arbitrary distribution of vorticity},
  journal = {Phys. D},
  volume  = {237},
  number  = {10--12},
  pages   = {1530--1538},
  year    = {2008},
  doi     = {10.1016/j.physd.2008.03.015}
}

@article{GuiWangYangZhang2026,
  author  = {Gui, C. and Wang, J. and Yang, W. and Zhang, Y.},
  title   = {Bifurcation analysis of {Stokes} waves with piecewise smooth vorticity in deep water},
  journal = {Math. Ann.},
  volume  = {396},
  year    = {2026},
  doi     = {10.1007/s00208-026-03584-5},
  url     = {https://doi.org/10.1007/s00208-026-03584-5},
  note    = {Article 50}
}

@article{HaziotEtAl2022,
  author  = {Haziot, S. V. and Hur, V. M. and Strauss, W. A. and Toland, J. F. and Wahl{\'e}n, E. and Walsh, S. and Wheeler, M. H.},
  title   = {Traveling water waves---the ebb and flow of two centuries},
  journal = {Quart. Appl. Math.},
  volume  = {80},
  pages   = {317--401},
  year    = {2022}
}

@article{Henry2008,
  author  = {Henry, D.},
  title   = {On {Gerstner}'s water wave},
  journal = {J. Nonlinear Math. Phys.},
  volume  = {15},
  pages   = {87--95},
  year    = {2008}
}

@article{Hur2012,
  author  = {Hur, V. M.},
  title   = {No solitary waves exist on two-dimensional deep water},
  journal = {Nonlinearity},
  volume  = {25},
  pages   = {3301--3312},
  year    = {2012}
}

@article{HurWheeler2022,
  author  = {Hur, V. M. and Wheeler, M. H.},
  title   = {Overhanging and touching waves in constant vorticity flows},
  journal = {J. Differential Equations},
  volume  = {338},
  pages   = {572--590},
  year    = {2022}
}

@article{Kinnersley1976,
  author  = {Kinnersley, W.},
  title   = {Exact large amplitude capillary waves on sheets of fluid},
  journal = {J. Fluid Mech.},
  volume  = {77},
  pages   = {229--241},
  year    = {1976}
}

@article{KozlovLokharu2018,
  author  = {Kozlov, V. and Lokharu, E.},
  title   = {{$N$}-modal steady water waves with vorticity},
  journal = {J. Math. Fluid Mech.},
  volume  = {20},
  pages   = {853--867},
  year    = {2018}
}

@article{Le2019,
  author  = {Le, H.},
  title   = {On the existence and instability of solitary water waves with a finite dipole},
  journal = {SIAM J. Math. Anal.},
  volume  = {51},
  pages   = {4074--4104},
  year    = {2019}
}

@article{LeviCivita1925,
  author  = {Levi-Civita, T.},
  title   = {D{\'e}termination rigoureuse des ondes permanentes d'ampleur finie},
  journal = {Math. Ann.},
  volume  = {93},
  pages   = {264--314},
  year    = {1925}
}

@article{Lin1941I,
  author  = {Lin, C. C.},
  title   = {On the motion of vortices in two dimensions. {I}. {Existence} of the {Kirchhoff--Routh} function},
  journal = {Proc. Natl. Acad. Sci. USA},
  volume  = {27},
  pages   = {570--575},
  year    = {1941}
}

@article{Lin1941II,
  author  = {Lin, C. C.},
  title   = {On the motion of vortices in two dimensions. {II}. {Some} further investigations on the {Kirchhoff--Routh} function},
  journal = {Proc. Natl. Acad. Sci. USA},
  volume  = {27},
  pages   = {575--577},
  year    = {1941}
}

@article{MarchioroPulvirenti1983,
  author  = {Marchioro, C. and Pulvirenti, M.},
  title   = {{Euler} evolution for singular initial data and vortex theory},
  journal = {Comm. Math. Phys.},
  volume  = {91},
  pages   = {563--572},
  year    = {1983}
}

@book{MarchioroPulvirinti1994,
  author    = {Marchioro, C. and Pulvirenti, M.},
  title     = {Mathematical Theory of Incompressible Nonviscous Fluids},
  series    = {Applied Mathematical Sciences},
  volume    = {96},
  publisher = {Springer},
  address   = {New York},
  year      = {1994}
}

@article{MartinMatioc2016,
  author  = {Martin, C. I. and Matioc, B.-V.},
  title   = {Gravity water flows with discontinuous vorticity and stagnation points},
  journal = {Comm. Math. Sci.},
  volume  = {14},
  pages   = {415--441},
  year    = {2016}
}

@book{Nekrasov1951,
  author    = {Nekrasov, A. I.},
  title     = {The Exact Theory of Steady Waves on the Surface of a Heavy Fluid},
  publisher = {Izdat. Akad. Nauk SSSR},
  address   = {Moscow},
  year      = {1951},
  note      = {In Russian}
}

@book{Newton2001,
  author    = {Newton, P. K.},
  title     = {The {$N$}-Vortex Problem: Analytical Techniques},
  series    = {Applied Mathematical Sciences},
  volume    = {145},
  publisher = {Springer},
  address   = {New York},
  year      = {2001}
}

@article{ShatahWalshZeng2013,
  author  = {Shatah, J. and Walsh, S. and Zeng, C.},
  title   = {Travelling water waves with compactly supported vorticity},
  journal = {Nonlinearity},
  volume  = {26},
  pages   = {1529--1564},
  year    = {2013}
}

@article{SheltonTrinh2023,
  author  = {Shelton, J. and Trinh, P. H.},
  title   = {Exponential asymptotics and the generation of free-surface flows by submerged line vortices},
  journal = {J. Fluid Mech.},
  volume  = {958},
  pages   = {A29},
  year    = {2023},
  doi     = {10.1017/jfm.2023.94}
}

@article{SimmenSaffman1985,
  author  = {Simmen, J. A. and Saffman, P. G.},
  title   = {Steady deep-water waves on a linear shear current},
  journal = {Stud. Appl. Math.},
  volume  = {73},
  pages   = {35--57},
  year    = {1985}
}

@article{SmetsVanSchaftingen2010,
  author  = {Smets, D. and Van Schaftingen, J.},
  title   = {Desingularization of vortices for the {Euler} equation},
  journal = {Arch. Ration. Mech. Anal.},
  volume  = {198},
  pages   = {869--925},
  year    = {2010}
}

@article{Sun1997,
  author  = {Sun, S. M.},
  title   = {Some analytical properties of capillary-gravity waves in two-fluid flows of infinite depth},
  journal = {Proc. Roy. Soc. London Ser. A},
  volume  = {453},
  pages   = {1153--1175},
  year    = {1997}
}

@article{TelesDaSilvaPeregrine1988,
  author  = {Teles da Silva, A. F. and Peregrine, D. H.},
  title   = {Steep, steady surface waves on water of finite depth with constant vorticity},
  journal = {J. Fluid Mech.},
  volume  = {195},
  pages   = {281--302},
  year    = {1988}
}

@article{TerKrikorov1958,
  author  = {Ter-Krikorov, A. M.},
  title   = {Exact solution of the problem of the motion of a vortex under the surface of a liquid},
  journal = {Izv. Akad. Nauk SSSR Ser. Mat.},
  volume  = {22},
  pages   = {177--200},
  year    = {1958},
  note    = {In Russian}
}

@article{Toland1996,
  author  = {Toland, J. F.},
  title   = {{Stokes} waves},
  journal = {Topol. Methods Nonlinear Anal.},
  volume  = {7},
  pages   = {1--48},
  year    = {1996}
}

@article{Turkington1983I,
  author  = {Turkington, B.},
  title   = {On steady vortex flow in two dimensions. {I}},
  journal = {Comm. Partial Differential Equations},
  volume  = {8},
  pages   = {999--1030},
  year    = {1983}
}

@article{Turkington1983II,
  author  = {Turkington, B.},
  title   = {On steady vortex flow in two dimensions. {II}},
  journal = {Comm. Partial Differential Equations},
  volume  = {8},
  pages   = {1031--1071},
  year    = {1983}
}

@article{Varholm2016,
  author  = {Varholm, K.},
  title   = {Solitary gravity-capillary water waves with point vortices},
  journal = {Discrete Contin. Dyn. Syst.},
  volume  = {36},
  pages   = {3927--3959},
  year    = {2016}
}

@article{Varholm2020,
  author  = {Varholm, K.},
  title   = {Global bifurcation of waves with multiple critical layers},
  journal = {SIAM J. Math. Anal.},
  volume  = {52},
  pages   = {5066--5089},
  year    = {2020}
}

@article{VarholmWahlenWalsh2020,
  author  = {Varholm, K. and Wahl{\'e}n, E. and Walsh, S.},
  title   = {On the stability of solitary water waves with a point vortex},
  journal = {Comm. Pure Appl. Math.},
  volume  = {73},
  pages   = {2634--2684},
  year    = {2020}
}

@article{VarvarucaWeiss2011,
  author  = {Varvaruca, E. and Weiss, G. S.},
  title   = {A geometric approach to generalized {Stokes} conjectures},
  journal = {Acta Math.},
  volume  = {206},
  number  = {2},
  pages   = {363--403},
  year    = {2011},
  doi     = {10.1007/s11511-011-0066-y}
}

@article{Wahlen2009,
  author  = {Wahl{\'e}n, E.},
  title   = {Steady water waves with a critical layer},
  journal = {J. Differential Equations},
  volume  = {246},
  pages   = {2468--2483},
  year    = {2009}
}

@article{WahlenWeber2024,
  author  = {Wahl{\'e}n, E. and Weber, J.},
  title   = {Large-amplitude steady gravity water waves with general vorticity and critical layers},
  journal = {Duke Math. J.},
  volume  = {173},
  pages   = {2197--2258},
  year    = {2024}
}

@article{WangVandenBroeckMilewski2014,
  author  = {Wang, Z. and Vanden-Broeck, J.-M. and Milewski, P. A.},
  title   = {Asymmetric gravity-capillary solitary waves on deep water},
  journal = {J. Fluid Mech.},
  volume  = {759},
  pages   = {R2},
  year    = {2014}
}

@misc{WangYanZhou2026,
  author       = {Wang, Y. and Yan, S. and Zhou, M.},
  title        = {Asymmetric periodic water waves in finite depth},
  howpublished = {arXiv:2609.08150},
  year         = {2026},
  url          = {https://arxiv.org/abs/2609.08150}
}

@article{Wheeler2018,
  author  = {Wheeler, M. H.},
  title   = {Integral and asymptotic properties of solitary waves in deep water},
  journal = {Comm. Pure Appl. Math.},
  volume  = {71},
  pages   = {1941--1956},
  year    = {2018}
}

@article{WilkeningZhao2021JFM,
  author  = {Wilkening, J. and Zhao, X.},
  title   = {Quasi-periodic travelling gravity-capillary waves},
  journal = {J. Fluid Mech.},
  volume  = {915},
  pages   = {A7},
  year    = {2021}
}

@article{WilkeningZhao2021JNS,
  author  = {Wilkening, J. and Zhao, X.},
  title   = {Spatially quasi-periodic water waves of infinite depth},
  journal = {J. Nonlinear Sci.},
  volume  = {31},
  pages   = {52},
  year    = {2021}
}
   
   \end{document}